\documentclass[11pt]{article}

\usepackage[english]{babel}
\usepackage{amssymb}
\usepackage{amsmath}
\usepackage{enumerate}
\usepackage{amsthm}
\usepackage{amsfonts}
\usepackage{listings}
\usepackage[curve]{xypic}
\usepackage{MnSymbol}
\usepackage{rotating}
\usepackage{hyperref}

\usepackage[numbers,sort&compress]{natbib}

\title{Six Functor Formalisms and Fibered Multiderivators}
\date{Februar 27, 2017}
\author{Fritz H\"ormann\\ Mathematisches Institut, Albert-Ludwigs-Universit\"at Freiburg}

\usepackage{color}
\usepackage{graphicx}
\usepackage[headsepline,footsepline]{scrpage2}

\usepackage{tikz}

\newtheorem{SATZ}{Theorem}[section]
\newtheorem{HAUPTSATZ}[SATZ]{Main theorem}
\newtheorem{LEMMA}[SATZ]{Lemma}
\newtheorem{DEF}[SATZ]{Definition}
\newtheorem{PROP}[SATZ]{Proposition}

\newtheorem{BEISPIEL}[SATZ]{Example}

\newtheorem{KOR}[SATZ]{Corollary}

\newtheorem{BEM}[SATZ]{Remark}

\newtheoremstyle{bare}        
  {}            
  {}            
  {\normalfont}                 
  {}                            
  {\bfseries}                   
  {}                            
  {.0em}                           
  {\thmnumber{#2}#1. \thmnote{\normalfont\textsc{(#3)}} } 

\theoremstyle{bare}
\newtheorem{PAR}[SATZ]{}

\newcommand{\Z}{ \mathbb{Z} }

\newcommand{\DD}{ \mathbb{D} }

\newcommand{\SSS}{ \mathbb{S} }

\newcommand{\iso}{\stackrel{\sim}{\longrightarrow}}

\DeclareMathOperator{\cor}{cor}

\DeclareMathOperator{\Ob}{Ob}
\DeclareMathOperator{\Span}{Cor}
\DeclareMathOperator{\tw}{tw}
\DeclareMathOperator{\op}{op}

\DeclareMathOperator{\id}{id}

\DeclareMathOperator{\Hom}{Hom}
\DeclareMathOperator{\Fun}{Fun}

\DeclareMathOperator{\pr}{pr}
\DeclareMathOperator{\Mor}{Mor}

\DeclareMathOperator{\Dia}{Dia}

\begin{document}

\maketitle


{\footnotesize  {\em Keywords:} derivators, fibered derivators, multiderivators, (op)fibered 2-multicategories, six-functor-formalisms, \\Grothendieck contexts, Wirthm\"uller contexts }

\section*{Abstract}

We develop the theory of (op)fibrations of 2-multicategories and
use it to define abstract six-functor-formalisms. 
We also give axioms 
for Wirthm\"uller and Grothendieck formalisms (where either $f^!=f^*$ or $f_!=f_*$) or intermediate formalisms where we have e.g.\@ a natural morphism $f_! \rightarrow f_*$. 
Finally, it is shown that a fibered multiderivator (in particular, a closed monoidal derivator) can be interpreted as a six-functor-formalism on diagrams (small categories). 
This gives, among other things, a considerable simplification of the axioms and of the proofs of basic properties, and clarifies
the relation between the internal and external monoidal products in a (closed) monoidal derivator. 
Our main motivation is the development of a theory of {\em derivator versions} of six-functor-formalisms. 

\tableofcontents

\section*{Introduction}

\subsection*{Six-functor-formalisms}

Let $\mathcal{S}$ be a category, for instance a suitable category of schemes, topological spaces, analytic manifolds, etc.
A Grothendieck six functor formalism on $\mathcal{S}$ consists of a collection of (derived) categories $\mathcal{D}_S$, one for each ``base space'' $S$ in $\mathcal{S}$
with the following six types of operations:
\vspace{0.5cm}
\[
\begin{array}{lcrp{1cm}l}
f^* & &f_* &  & \text{\em for each $f$ in $\Mor(\mathcal{S})$}  \\
\\
f_! & &f^! & & \text{\em for each $f$ in $\Mor(\mathcal{S})$} \\
\\
\otimes & & \mathcal{HOM} & & \text{\em in each fiber $\mathcal{D}_S$}
\end{array}
\]
\vspace{0.2cm}

The fiber $\mathcal{D}_S$ is, in general, a {\em derived} category of ``sheaves'' over $S$, for example coherent sheaves, $l$-adic sheaves, abelian sheaves, $D$-modules, motives, etc. 
The functors on the left hand side are left adjoints of the functors on the right hand side. 
The functor $f_!$ and its right adjoint $f^!$ are called ``push-forward with proper support'', and ``exceptional pull-back'', respectively.
The six functors come along with the following isomorphisms between them and it is not easy to make their axioms really precise. 

\begin{PAR}\label{INTRO6FUISO}
\begin{center}
\begin{tabular}{rlll}
& left adjoints & right adjoints \\
\hline
$(*,*)$ & $(fg)^* \iso g^* f^*$ & $(fg)_* \iso f_* g_*$ &\\
$(!,!)$ & $(fg)_! \iso f_! g_!$ & $(fg)^! \iso g^! f^!$ &\\ 
$(!,*)$ 
& $g^* f_! \iso F_! G^*$ & $G_* F^! \iso f^! g_*$ & \\
$(\otimes,*)$ & $f^*(- \otimes -) \iso f^*- \otimes f^* -$ & $f_* \mathcal{HOM}(f^*-, -) \iso \mathcal{HOM}(-, f_*-)$  & \\
$(\otimes,!)$ & $f_!(- \otimes f^* -) \iso  (f_! -) \otimes -$ & $f_* \mathcal{HOM}(-, f^!-) \iso \mathcal{HOM}(f_! -, -)$ & \\ 
& & $f^!\mathcal{HOM}(-, -) \iso \mathcal{HOM}(f^* -, f^!-)$ & \\
$(\otimes, \otimes)$ &  $(- \otimes -) \otimes - \iso - \otimes (- \otimes -)$ &  $\mathcal{HOM}(- \otimes -, -) \iso \mathcal{HOM}(-, \mathcal{HOM}(-, -))$ & 
\end{tabular}
\end{center}
\end{PAR}

Here $f, g, F, G$ are morphisms in $\mathcal{S}$ which, in the $(!,*)$-row, are related by the {\em Cartesian} diagram
\[ \xymatrix{ \cdot \ar[r]^G \ar[d]_F  & \cdot \ar[d]^f \\ \cdot \ar[r]_g & \cdot } \]

As we explained in \cite[Appendix~A.2]{Hor15}, a reasonable precise definition is the following:

\vspace{0.2cm}

{\bf Definition \ref{DEF6FU}. }{\em Let $\mathcal{S}$ be a category with fiber products. 
A (symmetric) Grothendieck six-functor-formalism on $\mathcal{S}$ is a 1-bifibration and 2-bifibration of (symmetric) 2-multicategories with 1-categorical fibers
\[ p: \mathcal{D} \rightarrow \mathcal{S}^{\mathrm{cor}} \]
where $\mathcal{S}^{\mathrm{cor}}$ is the symmetric 2-multicategory of correspondences in $\mathcal{S}$ (cf.\@ Definition \ref{DEFSCOR}).}

\vspace{0.2cm}

1-Morphisms of $\mathcal{S}^{\mathrm{cor}}$ are multicorrespondences of the form
\begin{equation}\label{acor}
 \xymatrix{ &&&U \ar[dlll]_{g_1} \ar[dl]^{g_n}  \ar[dr]^{f} \\ 
S_1 & \cdots & S_n &;& T.
} 
\end{equation}
The push-forward along the 1-morphism (\ref{acor}) corresponds in the classical language to the functor
\[ f_! ( (g^*_1 - ) \otimes \cdots \otimes ( g_n^* - )).  \]
 
Hence, from such a bifibration we obtain the operations $f_*$, $f^*$ (resp.\@ $f^!$, $f_!$) as pull-back and push-forward along the correspondences
\[ f^{\op}: \vcenter{\xymatrix{ 
 & \ar[ld]_f X \ar@{=}[rd] &\\
 Y & ; & X   } }\quad \text{and} \quad f:
 \vcenter{\xymatrix{ 
 & \ar@{=}[ld] X \ar[rd]^f &\\
 X & ; & Y ,  }}
 \]
respectively. We get $\mathcal{E} \otimes \mathcal{F}$ for objects $\mathcal{E}, \mathcal{F}$ above $X$ as the target of a Cartesian $2$-ary multimorphism 
from the pair $\mathcal{E}, \mathcal{F}$ over the correspondence
\[ \vcenter{\xymatrix{
& & X \ar@{=}[lld] \ar@{=}[ld] \ar@{=}[rd] &  \\
X & X & ; & X. }}
\]

In \cite{Hor15} we showed that all isomorphisms of \ref{INTRO6FUISO} follow from this definition. 

We explain in Section \ref{GROTHWIRTH} that enlarging the domain of 2-morphisms to all morphisms, or to a class of ``proper'', or ``etale'' morphisms, respectively, one can easily encode
all sorts of more strict six-functor-formalisms, where either $f_! = f_*$ or $f^*=f^!$ for all morphisms (so called Grothendieck or Wirthm\"uller contexts) or where
we have canonical natural transformations $f_! \rightarrow f_*$ for all morphisms $f$ (such that the diagonal is ``proper'') or canonical morphisms $f^* \rightarrow f^!$ for all morphisms $f$ (which are ``etale''). 

Definition~\ref{DEF6FU} has the huge advantage that it also encodes all compatibilities between the isomorphisms of \ref{INTRO6FUISO}. For example, it is just a matter of contemplating a diagram of 2-morphisms in $\mathcal{S}^{\cor}$ to see that the diagram

\[\xymatrix{
G_! F^* (A \otimes g^*B)  \ar[d]_{(\otimes, *)}  && f^* g_! (A \otimes g^*B) \ar[ll]_{(!,*)}  \ar[d]^{(\otimes, !)} \\
G_! ((F^* A) \otimes F^*g^*B)   && f^* ((g_!A) \otimes B)  \ar[dd]^{(\otimes, *)} \\
G_! ((F^* A) \otimes (gF)^*B) \ar[d]_{(*,*)} \ar[u]^{(*,*)}  \\
G_! ((F^* A) \otimes G^*f^*B)  \ar[dr]_{(\otimes, !)}  &&  (f^* g_!A)  \otimes f^*B  \ar[dl]^{(*, !)} \\  
& (G_! F^* A) \otimes f^*B  &
}\]

commutes or ---  in the proper/Grothendieck case ---  that the diagram 

 \[ \xymatrix{
f_* \Hom(A, f^! B)  \ar[r]^{\sim} \ar[d]^{} &  \Hom(f_!A,B)  \\
f_* \Hom(f^*f_! A, f^! B)  \ar[r]^\sim &  \Hom(f_!A, f_*f^! B) \ar[u]
} \]

depicted on the front cover of Lipman's book \cite{LH09} on Grothendieck duality, commutes. 

\subsection*{Fibered multiderivators}

For a detailed introduction to derivators and fibered multiderivators we refer to \cite{Hor15}. 
Stable derivators, among other things, simplify, enhance and conceptually explain triangulated categories. Instead of considering just one category,
 a derivator $\DD$ specifies a category $\DD(I)$ for each diagram shape $I$ (small category), and pull-back functors $\alpha^*: \DD(J) \rightarrow \DD(I)$ for each
functor $\alpha: I \rightarrow J$. This has the advantage that a triangulation on the categories $\DD(I)$ does not have to be specified explicitly. Rather
the operations of taking cones and shifting objects are encoded as abstract homotopy limit and colimit functors, which are just left and right adjoints to certain of the given pull-back functors. Triangles are reconstructed from squares, i.e.\@ objects of $\DD(\Box)$ which are Cartesian and coCartesian at the same time.
All the axioms of triangulated categories are consequences of a rather intuitive set of properties of Kan extensions. 

A monoidal derivator specifies in addition a monoidal structure on the categories $\DD(I)$ which satisfies some additional axioms as for example
\[ \alpha^* ( - \otimes -) = ((\alpha^*  -) \otimes (\alpha^*  -))  \]
and the projection formula
\[ \alpha_! ( - \otimes (\alpha^* -)) = ((\alpha_!  -) \otimes - )  \]
for {\em certain} functors $\alpha$. Together with the base change formula
\[ \beta^* \alpha_! = B_! A^* \]
for {\em certain} functors $\alpha, \beta, A, B$ forming a Cartesian square,
this resembles a lot the datum and axioms of a six-functor-formalism in which $f^*=f^!$, i.e.\@ a Wirthm\"uller context.
By defining a  2-multicategory $\Dia^{\cor}$ of multicorrespondences of diagrams we make this analogy precise by showing the following general theorem.
(Note that a monoidal derivator is the same as a left fibered multiderivator over $\{\cdot\}$. )

\vspace{0.2cm}

{\bf Main theorem \ref{MAINTHEOREMFIBDER}}.  {\em
Let $\DD$ and $\SSS$ be pre-multiderivators satisfying (Der1) and (Der2) (cf.\@ \cite[Definition~1.3.5.]{Hor15}).
 A strict morphism of pre-multiderivators $\DD \rightarrow \SSS$ is a left (resp.\@ right) fibered multiderivator if and only if $\Dia^{\cor}(\DD) \rightarrow \Dia^{\cor}(\SSS)$ is a 1-opfibration (resp.\@ 1-fibration) of 2-multicategories. } 
 
 \vspace{0.2cm}

 Here $\Dia^{\cor}(\DD)$ is defined for any pre-multiderivator as an extension of the 2-multicategory of correspondences of diagrams $\Dia^{\cor}$. We have
 $\Dia^{\cor} = \Dia^{\cor}(\{\cdot\})$. However it is not essential that a pre-multiderivator is given a priori. For {\em any} 1-(op)fibration and 2-fibration $\mathcal{D} \rightarrow \Dia^{\cor}(\SSS)$ with 1-categorical fibers a (non-strict) pre-multiderivator can be reconstructed (cf.\@ \ref{PARDERCANWIRTH}). 
 
 Using the correspondence (cf.\@ \ref{PROPGROTHCONSTR}) between 1-opfibrations (and 2-fibrations) of 2-multicategories over $\Dia^{\cor}(\SSS)$ with 1-categorical fibers and pseudo-functors 
 $\Dia^{\cor}(\SSS) \rightarrow \mathcal{CAT}$ into the 2-multicategory of all categories (where the multi-structure is given by multivalued functors) this
 formulation unifies in a nice way the two pseudofunctors
\[   \Dia(\SSS)^{1-\op} \rightarrow \mathcal{CAT} \quad  \text{ resp.\@ } \quad \Dia^{\op}(\SSS)^{1-\op} \rightarrow \mathcal{CAT} \]
that have been associated with a fibered multiderivator in \cite{Hor15}, because there are embeddings of $\Dia(\SSS)^{1-\op}$ and $\Dia^{\op}(\SSS)^{1-\op}$ into
$\Dia^{\cor}(\SSS)$ (cf.\@ \ref{PAREMBEDDING1}--\ref{PAREMBEDDING2}).

For example, Ayoub had defined in \cite{Ayo07I, Ayo07II}  an {\em algebraic derivator} as a pseudo-functor
$\Dia(\SSS)^{1-\op} \rightarrow \mathcal{CAT}$ satisfying certain axioms, mentioning that this involved a choice because $\Dia^{\op}(\SSS)$ is an equally justified
forming. This problem led the author in \cite{Hor15} to the definition of a fibered multiderivator instead of using Ayoub's notion of algebraic derivator. 
The viewpoint in this article has the advantage not only of clarifying the difference of these two approaches but also of encoding most axioms of a fibered multiderivator in a more elegant way. 

The formal equalization of six-functor-formalisms and monoidal derivators explains many of their similarities. For example,
in both cases there is an internal monoidal product $\otimes$ (with adjoint denoted $\mathcal{HOM}$) and an external monoidal product $\boxtimes$ (with adjoint denoted $\mathbf{HOM}$).
The external monoidal product and Hom are compatible with the one on $\mathcal{S}^{\cor}$ given by $S \otimes T = S \times T$ and $\mathcal{HOM}(S, T) = S \times T$,  
and with the one on $\Dia^{\cor}$ given by $I \otimes J = I \times J$ and $\mathcal{HOM}(I, J) = I^{\op} \times J$, respectively. 
This is just a common feature of 1-/2- (op-)fibrations of 2-multicategories: The notions are transitive. Hence, for instance, if $\mathcal{D} \rightarrow \mathcal{S} \rightarrow \{ \cdot \}$ is a sequence of 1-/2- (op-)fibrations of multicategories, where  $\{ \cdot \}$ is the final multicategory, also $\mathcal{D} \rightarrow  \{ \cdot \}$ is a 1-/2- (op-)fibration. 
While $\mathcal{D} \rightarrow \mathcal{S}$ being a 1-opfibration encodes the existence of the internal monoidal product, $\mathcal{D} \rightarrow  \{ \cdot \}$ being a
1-opfibration encodes the existence of the external monoidal product.

From the abstract properties of 1-/2- (op-)fibrations of 2-multicategories we can derive that
\begin{eqnarray*}
\mathcal{E} \otimes \mathcal{F} &=& \Delta_\bullet (\mathcal{E} \boxtimes \mathcal{F}), \\
\mathcal{HOM}(\mathcal{E}, \mathcal{F}) &=& (\Delta')^\bullet (\mathbf{HOM}(\mathcal{E} , \mathcal{F} )),
\end{eqnarray*}
and the external product, resp. external Hom, can also be reconstructed from the internal one in an analogous way. 
For the meaning of $\Delta$ and $\Delta'$ see Section~\ref{ABSRELMONOIDAL}. 
Explicitly, the first formula specializes to: 
\[ \mathcal{E} \otimes \mathcal{F} = \Delta^* (\mathcal{E} \boxtimes \mathcal{F}) \]
for the diagonal map $\Delta: S \rightarrow S \times S$ in the six-functor-formalism case, resp. $\Delta: I \rightarrow I \times I$ in the monoidal derivator case.
The second specializes to
\[ \mathcal{HOM}(\mathcal{E}, \mathcal{F}) = \Delta^! \mathbf{HOM}(\mathcal{E}, \mathcal{F}) \]
in the six-functor-formalism case and to
\[ \mathcal{HOM}(\mathcal{E}, \mathcal{F}) = \pr_{2,*} \pi_* \pi^* \mathbf{HOM}(\mathcal{E}, \mathcal{F}) \]
in the monoidal derivator case, with the following functors:
\[ \xymatrix{ \tw(I) \ar[r]^\pi & I^{\op} \times I \ar[r]^-{\pr_2} & I,  } \]
where $\tw(I)$ is the twisted arrow category. 
The slightly different behavior is due to the different definitions of $\mathcal{S}^{\cor}$ and $\Dia^{\cor}$. 
The definition of $\Dia^{\cor}$ takes the 2-categorical nature of $\Dia$ into account. For the same reason, it encodes
the more complicated base change formula of derivators involving comma categories as opposed to the 
simpler base change formula of a six-functor-formalism. And for the same reason the duality on $\Dia^{\cor}$ is not given by the identity $S \mapsto S$ as for $\mathcal{S}^{\cor}$ but by $I \mapsto I^{\op}$. 

The upshot is that the theory of 1-/2- (op-)fibrations of 2-multicategories is sufficiently powerful to 
treat classical six-functor formalisms and monoidal (resp.\@ multi-)derivators alike. 

\subsection*{Derivator six-functor formalisms}

However this is not the end of the story. Of course, classical six-functor-formalisms are mostly defined on families of triangulated categories.
Fibered multiderivators were defined to enhance and simplify the latter. In the preceding article \cite{Hor15} we already explained how questions of cohomological and homological descent can be treated nicely using this notion. Desirable is therefore a derivator six-functor-formalism which encodes not only the interplay
of the ``6 functors'' but also of the 3 additional functors: pull-back along functors of diagrams, left Kan extension and right Kan extension. One could say: a 9-functor-formalism.
Yet the theory of 1-/2- (op-)fibrations of 2-multicategories is still sufficiently powerful to deal with this situation. For this
we have to define pre-2-multiderivators. These are families of 2-multicategories rather than 1-multicategories. For example, the 2-multicategory of multicorrespondences $\mathcal{S}^{\cor}$ has an associated pre-2-multiderivator $\SSS^{\cor}$. A derivator six-functor-formalism, of course, should be a left and right fibered multiderivator over $\SSS^{\cor}$. 
Such will be defined and discussed in a subsequent article \cite{Hor16}. 










\subsection*{Overview}

This article is rather foundational. It develops in sections \ref{SECTION2MULTICAT} and \ref{SECTIONFIB} the basics of 1-/2- (op)fibrations of 2-multicategories which do not appear in the literature in this form. This is an extension and unification of existing work \cite{Her00, Her04, Bak09, Buc14, Her99}.
In Section~\ref{SECTIONCOR}, the 2-multicategory of correspondences $\mathcal{S}^{\cor}$
 is defined, and it is explained how classical six-functor-formalisms can be encoded as certain 1-bifibrations of 2-multicategories over $\mathcal{S}^{\cor}$.
 In Section~\ref{SECTIONDIACOR}, the 2-multicategory of correspondences of diagrams $\Dia^{\cor}$ is defined, which is slightly more complicated because it has
 to take the 2-categorical flavour of $\Dia$ into account. In Section~\ref{SECTIONWIRTHFIBDER} and \ref{YOGA}, it is explained that the notion of certain 1-bifibrations of 2-multicategories over $\Dia^{\cor}(\SSS)$ is basically equivalent to the notion of fibered multiderivator over $\SSS$. 
 
 In Section~\ref{ABSRELMONOIDAL}, the interplay between internal and external monoidal product is discussed from the abstract perspective of 
 1-/2- (op-)fibrations of 2-multicategories. In Section~\ref{GROTHWIRTH}, Grothendieck and Wirthm\"uller contexts, i.e.\@ those in which either $f_*=f_!$, or $f^!=f^*$ holds, are
 axiomatized as well as intermediate formalisms which we call proper, or etale, six-functor-formalisms, those in which there is still a canonical morphism $f_! \rightarrow f_*$, or $f^* \rightarrow f^!$, for appropriate morphisms $f$. 
 

\section{2-Multicategories}\label{SECTION2MULTICAT}

The notion of 2-multicategory is a straight-forward generalization of the notion of 2-category. 
For lack of reference and because we want to stick to the case of (strict) 2-categories as opposed to bicategories, we list the
relevant definitions here:

\begin{DEF}\label{DEF2MULTICAT}
A {\bf 2-multicategory} $\mathcal{D}$ consists of 
\begin{itemize}
\item a class of objects $\Ob(\mathcal{D})$;
\item for $n \in \Z_{\ge 0}$, and for objects $X_1, \dots, X_n, Y$ in $\Ob(\mathcal{D})$ a category
\[ \Hom(X_1, \dots, X_n; Y);  \]
\item a composition, i.e.\@ for objects $X_1, \dots, X_n$, $Y_1, \dots, Y_m$, $Z$ in $\Ob(\mathcal{D})$ and for $i \in \{1,\dots, m\}$ a functor:
\begin{eqnarray*} 
\Hom(X_1, \dots, X_n; Y_i) \times \Hom(Y_1, \dots, Y_m; Z) &\rightarrow& \Hom(Y_1, \dots, Y_{i-1}, X_1, \dots, X_n, Y_{i+1}, \dots, Y_m; Z) \\
f, g &\mapsto& g \circ_i f;
\end{eqnarray*}
\item for $X \in \Ob(\mathcal{D})$ an identity object $\id_X$ in the category $\Hom(X; X)$;
\end{itemize}
satisfying strict associativity and identity laws. The composition w.r.t.\@ independent slots is commutative, i.e.\@ for $1\le i<j \le m$ 
if $f \in \Hom(X_1, \dots, X_n; Y_i)$ and $f' \in \Hom(X_1', \dots, X_k'; Y_j)$ and $g \in \Hom(Y_1, \dots, Y_m; Z)$ then 
\[   (g \circ_i f) \circ_{j+n-1} f'  =  (g \circ_j f') \circ_{i} f.     \]

A symmetric (braided) 2-multicategory is given by an action of the symmetric (braid) groups, i.e.\@ isomorphisms of categories
\[  \alpha: \Hom(X_1, \dots, X_n; Y) \rightarrow \Hom(X_{\alpha(1)}, \dots, X_{\alpha(n)}; Y) \]
for $\alpha \in S_n$ (resp.\@ $\alpha \in B_n$) forming an action which is strictly compatible with composition in the obvious way (in the braided case: substitution of strings).
\end{DEF}

The 1-composition of 2-morphisms is (as for usual 2-categories) determined by the following whiskering operations: Let $f, g \in \Hom(X_1, \dots, X_n; Y_i)$ and $h \in \Hom(Y_1, \dots, Y_m; Z)$ be 1-morphisms and let $\mu: f \Rightarrow g$ be a 2-morphism, i.e.\@ a morphism in the category $\Hom(X_1, \dots, X_n; Y_i)$. Then we define
\[ h  \ast \mu := \id_h \cdot \mu \]
where the right hand side is the image of the 2-morphism $\id_h \times \mu$ under the composition functor.
Similarly we define $\mu \ast h$ for $\mu: f \Rightarrow g$ with $f, g \in \Hom(Y_1, \dots, Y_m; Z)$ and $h \in \Hom(X_1, \dots, X_n; Y_i)$. 

\begin{PAR}\label{OPMULTI}
In the same way, we define a {\bf 2-opmulticategory} having categories of 1-morphisms of the form
\[ \Hom(X; Y_1, \dots, Y_n).  \]
For each 2-multicategory $\mathcal{D}$ there is a natural 2-opmulticategory $\mathcal{D}^{1-\op}$, and vice versa, where the direction of the 1-morphisms is flipped. 
\end{PAR}

\begin{DEF}\label{DEFPSEUDOFUN}
A {\bf pseudo-functor} $F: \mathcal{C} \rightarrow \mathcal{D}$ between 2-multicategories is given by the following data:
\begin{itemize}
\item 
for $X \in \Ob(\mathcal{C})$ an object $F(X) \in \Ob(\mathcal{C})$;
\item 
for  $X_1, \dots, X_n; Y \in \Ob(\mathcal{C})$, a functor
\[ \Hom(X_1, \dots, X_n; Y) \rightarrow \Hom(F(X_1), \dots, F(X_n); F(Y)); \]
\item 
for $X \in \Ob(\mathcal{C})$ a 2-isomorphism
\[ F_X: F(\id_X) \Rightarrow \id_{F(X)}; \]
\item 
for $X_1, \dots, X_n; Y_1, \dots, Y_m; Z \in \Ob(\mathcal{C})$ and $i \in \{1, \dots, m\}$  a natural isomorphism 
\[ F_{-,-}: F(-) \circ_i F(-) \Rightarrow F(- \circ_i -) \]
of functors
\begin{gather*} 
\Hom(X_1, \dots, X_n; Y_i) \times \Hom(Y_1, \dots, Y_m; Z) \\
 \rightarrow \Hom(F(Y_1), \dots, F(Y_{i-1}), F(X_1), \dots, F(X_n), F(Y_{i+1}), \dots, F(Y_m); F(Z));
 \end{gather*}
\end{itemize}
satisfying 
\[ F_{\id_Y,f} = F_Y \ast F(f) \qquad F_{g, \id_{Y_i}} = F(g) \ast  F_{Y_i}  \]
 for $f \in  \Hom(X_1, \dots, X_n; Y)$, and $g \in \Hom(Y_1, \dots, Y_m; Z)$, respectively,
and for composable $f, g, h$ that
\[ \xymatrix{
F(h) \circ_j F(g) \circ_i F(f) \ar[rr]^-{F(h) \ast F_{g,f}} \ar[d]_{  F_{h,g} \ast F(f)} &&  F(h) \circ_j F(g \circ_i f) \ar[d]^{F_{h,gf}} \\
F(h \circ_j g) \circ_i F(f)  \ar[rr]^-{F_{hg,f}} &&  F(h \circ_j g \circ_i f) 
} \]
commutes. A pseudo functor is called a {\bf strict functor} if all $F_{g,f}$ and $F_X$ are identities. 
\end{DEF}

\begin{DEF}\label{DEFPSEUDONATTRANS}
A {\bf pseudo-natural transformation} $\alpha: F_1, \dots, F_m \Rightarrow G$ between pseudo-functors $F_1, \dots, F_m; G: \mathcal{C} \rightarrow \mathcal{D}$ is given by: 
\begin{itemize}
\item 
for $X \in \Ob(\mathcal{C})$ a 1-morphism $\alpha(X) \in \Hom(F_1(X), \dots, F_m(X); G(X))$;
\item 
for each 1-morphism $f$ in $\Hom(X_1, \dots, X_n; Y)$ a 2-isomorphism
\[ \alpha_f: \alpha(Y) \circ (F_1(f), \dots, F_m(f))  \Rightarrow  G(f) \circ (\alpha(X_1), \dots, \alpha(X_n)); \] 
\end{itemize}
such that all the following diagrams commute:
\begin{itemize}
\item
for $f \in  \Hom(X_1, \dots, X_n; Y_i)$ 
and $g \in \Hom(Y_1, \dots, Y_k; Z)$:
{\footnotesize \[ \xymatrix{
\alpha(Z) (F_1(g) F_1(f), \dots, F_m(g) F_m(f)) \ar[rrr]^-{ (G(g) \ast  \alpha_{f}) (\alpha_g \ast F(f))} \ar[d]_{\alpha(Z) \ast ((F_1)_{g,f}, \dots, (F_m)_{g,f})} &&& G(g) G(f) (\alpha(Y_1), \dots, \alpha(X_1), \dots, \alpha(X_n), \dots, \alpha(Y_k)) \ar[d]^{G_{g,f} \ast (\dots)}  \\
\alpha(Z) (F_1(gf), \dots, F_m(gf))  \ar[rrr]^{\alpha_{gf}} &&& G(gf) (\alpha(Y_1), \dots, \alpha(X_1), \dots, \alpha(X_n), \dots, \alpha(Y_k))
} \]}
\item for $X \in \Ob(\mathcal{C})$:
\[ \xymatrix{
\alpha(X) (F_1(\id_{X}), \dots, F_n(\id_{X})) \ar[rr]^-{\alpha_{\id_X}} \ar[d]^{ \alpha(X) \ast ((F_1)_X, \dots, (F_n)_X)} && G(\id_X) \alpha(X) \ar[d]^{G_X \ast \alpha(X)} \\
\alpha(X) (\id_{F_1(X)}, \dots, \id_{F_n(X)}) \ar@{=}[rr] && \id_{G(X)} \alpha(X)
} \]
\item 
for each 2-morphism $f \Rightarrow g$ in $\Hom(X_1, \dots, X_n; Y)$:
\[ \xymatrix{
\alpha(Y) \cdot (F_1(f), \dots, F_m(f)) \ar[r] \ar[d] & G(f) \cdot (\alpha(X_1), \dots, \alpha(X_n)) \ar[d] \\
\alpha(Y) \cdot (F_1(g), \dots, F_m(g)) \ar[r] &  G(g) \cdot (\alpha(X_1), \dots, \alpha(X_n))
} \]
\end{itemize}
Similarly we define an {\bf oplax natural transformation}, if the morphism $\alpha_f$ is no longer required to be a 2-isomorphism but can be any 2-morphism.
We define a {\bf lax natural transformation} requiring that the morphism $\alpha_f$ goes in the other direction, with the diagrams above changed suitably.
\end{DEF}

\begin{DEF}\label{DEFMODIFICATION}
A {\bf modification} $\mu: \alpha \Rrightarrow \beta$ between  $\alpha, \beta: F_1, \dots, F_m \Rightarrow G$ (pseudo-, lax-, or oplax-) natural transformations is given by the following data:
\begin{itemize}
\item For $X \in \Ob(\mathcal{C})$ a 2-morphism 
\[ \mu_X: \alpha(X) \Rightarrow \beta(X) \]
\end{itemize}
such that for each 1-morphism  $f  \in \Hom(X_1, \dots, X_n; Y)$
the following diagram commutes:
\[ \xymatrix{
\alpha(Y) \circ (F_1(f), \dots, F_m(f))  \ar[r] \ar[d] & G(f) \circ (\alpha(X_1), \dots, \alpha(X_n)) \ar[d] \\
\beta(Y) \circ (F_1(f), \dots, F_m(f))   \ar[r] & F(f) \circ (\beta(X_1), \dots, \beta(X_n))
} \]
resp.\@ (in the lax case) the analogue diagram with the horizontal arrows reversed. 
\end{DEF}

\begin{LEMMA}
Let $\mathcal{C}$, $\mathcal{D}$ be 2-multicategories. Then the collection 
\[ \Fun(\mathcal{C}, \mathcal{D}) \]
of pseudo-functors, pseudo-natural transformations and modifications forms a 2-multicategory. 
Similarly the collections
\[ \Fun^{\mathrm{lax}}(\mathcal{C}, \mathcal{D}) \qquad  \Fun^{\mathrm{oplax}}(\mathcal{C}, \mathcal{D}) \]
of pseudo-functors, (op)lax natural transformations and modifications form 2-multicategories.
\end{LEMMA}

\begin{proof}We leave the proof to the reader, but will explicitly spell out how pseudo-natural transformations are composed:

Let $\alpha: F_1, \dots, F_m \Rightarrow G_i$ and $\beta: G_1, \dots, G_n \Rightarrow H$ be pseudo-natural transformations.  
Then the pseudo-natural transformation 
\[ \beta \circ_i \alpha : G_1, \dots, G_{i-1}, F_1, \dots, F_m, G_{i+1},  \dots, G_n \Rightarrow H \]
is given as follows. $(\beta \circ_i \alpha)(X)$ is just the composition of $\beta(X) \circ \alpha(X)$ and the 2-morphism 
\[ (\beta \circ_i \alpha)_f:  \beta(X) \circ_i \alpha(X) \circ (G_1(f), \dots, F_1(f), \dots, F_m(f), \dots, G_n(f)) \Rightarrow H(f) \circ (\beta(X_1)\alpha(X_1), \dots, \beta(X_n)\alpha(X_n)) \]
is given by the composition
\[ \beta(X) \circ_i \alpha(X) \circ (G_1(f), \dots, F_1(f), \dots, F_m(f), \dots, G_n(f)) \Rightarrow  \]
\[ \beta(X) \circ_i  (G_1(f), \dots, G_i(f), \dots, G_n(f)) \circ (\alpha(X_1), \dots, \alpha(X_n)) \Rightarrow \]
\[ H(f) \circ (\beta(X_1) \circ_i  \alpha(X_1), \dots, \beta(X_n)  \circ_i  \alpha(X_n)). \]
\end{proof}

\section{(Op)fibrations of 2-multicategories}\label{SECTIONFIB}

For (op)fibrations of (usual) multicategories the reader may consult \cite{Her00, Her04}, and for (op)fibrations of 2-categories  
\cite{Bak09, Buc14, Her99}. The definitions in this section however are slightly different from those in any of these sources. 

\begin{PAR}
Let 
\[ \xymatrix{
\mathcal{A} \ar[r] \ar[d] \ar@{}[rd]|-{\Swarrow^\mu} & \mathcal{B} \ar[d]^\beta \\
\mathcal{C} \ar[r]_\gamma & \mathcal{D}
} \]
be a 2-commutative diagram of (usual) categories, where $\mu$ is a natural isomorphism. Then we say that
the diagram is {\bf 2-Cartesian} if it induces an equivalence 
\[ \mathcal{A} \cong \mathcal{B} \times_{/ \mathcal{D}}^\sim \mathcal{C}, \]
where $\mathcal{B} \times_{/ \mathcal{D}}^\sim \mathcal{C}$ is the full subcategory of the comma category $\mathcal{B} \times_{/ \mathcal{D}} \mathcal{C}$
consisting of those objects $(b, c, \nu: \beta(b) \rightarrow \gamma(c))$, with $b \in \mathcal{B}$, $c \in \mathcal{C}$ in which the morphism $\nu$ is an isomorphism.   

If $\mu$ is an identity then the diagram is said to be {\bf Cartesian}, if it induces an equivalence of categories
\[ \mathcal{A} \cong \mathcal{B} \times_{\mathcal{D}} \mathcal{C}. \]

\end{PAR}
\begin{LEMMA}\label{LEMMACARTSQUAREISOFIB}
If 
\begin{equation}\label{eqcartdia} \vcenter{ \xymatrix{
\mathcal{A} \ar[r]^\delta \ar[d]_\alpha & \mathcal{B} \ar[d]^\beta \\
\mathcal{C} \ar[r]_\gamma & \mathcal{D}
} } \end{equation}
is a strictly commutative diagram of categories then:
\begin{enumerate}
\item If $\beta$ is an iso-fibration (i.e.\@ the corresponding functor between the groupoids of isomorphisms is 
a fibration or, equivalently, an opfibration) then for (\ref{eqcartdia}) the two notions
\[ \text{ 2-Cartesian } \text{ and }  \text{ Cartesian } \]
are equivalent.
\item If $\alpha$ is an iso-fibration then (\ref{eqcartdia}) is Cartesian if and only if  
\begin{equation} \label{eqcart}
\mathcal{A} \rightarrow  \mathcal{B} \times_{\mathcal{D}} \mathcal{C} \end{equation}
is fully-faithful and for any $b \in \mathcal{B}$ and $c \in \mathcal{C}$ with $\beta(b) = \gamma(c)$ there exists an $a \in \mathcal{A}$ with 
$\alpha(a) = c$ and an isomorphism $\kappa: \delta(a) \rightarrow b$ with $\beta(\kappa) = \id_{\beta(b)}$.  
\item If $\alpha$ and $\beta$ are fibrations (resp.\@ opfibrations) and $\delta$ maps Cartesian (resp.\@ coCartesian) morphisms to Cartesian (resp.\@ coCartesian) morphisms then (\ref{eqcart}) is fully-faithful if and only if  $\delta$ induces an isomorphism
\[  \Hom_{\mathcal{A}, \id_c}(a, a') \cong \Hom_{\mathcal{B}, \id_{\gamma(c)}}(\delta(a), \delta(a'))  \]
for all $c \in \mathcal{C}$ and $a, a' \in \mathcal{A}$ with $\alpha(a) = \alpha(a') = c$. 
In particular (\ref{eqcartdia}) is Cartesian, or equivalently 2-Cartesian, if and only if $\delta$ induces an equivalence of categories between the fibers
\[   \mathcal{A}_c  \cong \mathcal{B}_{\gamma(c)} \]
for all objects $c \in \mathcal{C}$.  
\end{enumerate} 
\end{LEMMA}
\begin{proof}
1.\@ Indeed, if $\beta$ is an iso-fibration, the obvious functor 
\[ \mathcal{B} \times_{\mathcal{D}} \mathcal{C} \rightarrow \mathcal{B} \times_{/ \mathcal{D}}^\sim \mathcal{C} \]
has a quasi-inverse functor which maps an object $(b, c, \nu: \beta(b) \rightarrow \gamma(c))$ to $(b', c)$ for any choice
of coCartesian morphism $b \rightarrow b'$ (necessarily an isomorphism as well) over $\nu$. 

2.\@ Obviously if the condition is satisfied then the functor (\ref{eqcart}) is essentially surjective. If it is in turn essentially surjective,
for any $b \in \mathcal{B}$ and $c \in \mathcal{C}$ with $\beta(b) = \gamma(c)$ there exists an $a' \in \mathcal{A}$, an isomorphism
$\tau: \alpha(a') \rightarrow  c$, and an isomorphism $\kappa': \delta(a') \rightarrow b$ with $\beta(\kappa') = \gamma(\tau)$. 
Now choose a coCartesian morphism $\xi: a' \rightarrow a$ in $\mathcal{A}$  lying over $\tau$ which exists by assumption. It is  necessarily an isomorphism. Then 
we have $\alpha(a) = c$ and an isomorphism $\kappa:=   \kappa' \circ \delta(\xi^{-1}) $ with $\beta(\kappa) = \id_{\beta(b)}$. Hence the statement of 2.\@ is satisfied. 

3. The only if part is clear. For the if part, let $f: c \rightarrow c'$ be a morphism in $\mathcal{C}$. We have to show that   
\[ \Hom_{\mathcal{A}, f}(a, a'') \cong \Hom_{\mathcal{B}, \gamma(f)}(\delta(a), \delta(a'')).  \]
for any $a, a'' \in \mathcal{A}$ with $\alpha(a)=c, \alpha(a'')=c'$. Choose a Cartesian morphism $g: a' \rightarrow a''$ over $f$. Since $\delta$ maps $g$ to 
a Cartesian morphism we get a commutative diagram
\[ \xymatrix{
\Hom_{\mathcal{A}, \id_c}(a, a')  \ar[r]^-\delta \ar[d]_{ g\circ} & \Hom_{\mathcal{A}, \id_{\gamma(c)}}(\delta(a), \delta(a'))  \ar[d]^{ \delta(g) \circ}) \\
\Hom_{\mathcal{A}, f}(a, a'') \ar[r]_-\delta &  \Hom_{\mathcal{B}, \gamma(f)}(\delta(a), \delta(a''))
} \]
in which the vertical maps are isomorphisms. Hence it suffices to see the assertion of 3.\@ to show fully-faithfulness. If $\alpha, \beta$ are opfibrations one proceeds analogously choosing a coCartesian morphism. 
\end{proof}

\begin{DEF}\label{DEFCOCART}
Let $p: \mathcal{D} \rightarrow \mathcal{S}$ be a strict functor of 2-multicategories. 
A 1-morphism
\[ \xi \in \Hom_f(\mathcal{E}_1, \dots, \mathcal{E}_n; \mathcal{F}) \] 
in $\mathcal{D}$ over a 1-morphism $f \in \Hom(S_1, \dots, S_n; T)$
is called {\bf coCartesian} w.r.t.\@ $p$, 
 if for all $i$ and objects $\mathcal{F}_1,\dots,\mathcal{F}_m, \mathcal{G} \in \mathcal{D}$ with $\mathcal{F}_i=\mathcal{F}$, lying over $T_1, \dots, T_m, U \in \mathcal{S}$
 the diagram of categories
\[ \xymatrix{
\Hom_{\mathcal{D}}(\mathcal{F}_1, \dots, \mathcal{F}_m; \mathcal{G}) \ar[r]^-{\circ_i \xi} \ar[d] & \Hom_{\mathcal{D}}(\mathcal{F}_1, \dots, \mathcal{F}_{i-1}, \mathcal{E}_1, \dots, \mathcal{E}_n, \mathcal{F}_{i+1}, \dots, \mathcal{F}_m; \mathcal{G}) \ar[d] \\
\Hom_{\mathcal{S}}(T_1, \dots, T_m; U) \ar[r]^-{ \circ_i f} & \Hom_{\mathcal{S}}(T_1, \dots, T_{i-1}, S_1, \dots, S_n, T_{i+1}, \dots, T_m; U)
} \]
is 2-Cartesian. 

A 1-morphism
\[ \xi \in \Hom_f(\mathcal{E}_1, \dots, \mathcal{E}_n; \mathcal{F}) \] 
is called {\bf weakly coCartesian} w.r.t.\@ $p$, if 
\[ \xymatrix{
\Hom_{\id_T}(\mathcal{F}; \mathcal{G}) \ar[r]^-{\circ \xi} & \Hom_{f}(\mathcal{E}_1, \dots, \mathcal{E}_n; \mathcal{G})
} \]
 is an equivalence of categories for all $\mathcal{G} \in \mathcal{D}$ with $p(\mathcal{G}) = T$. 
\end{DEF}
If $p: \mathcal{D} \rightarrow \mathcal{S}$ is a 2-isofibration (cf.\@ Definition~\ref{DEF2FIB}) then a coCartesian 1-morphism is weakly coCartesian by the proof of Proposition~\ref{LEMMAWEAKLY} below. 

\begin{DEF}\label{DEFCART}
Let $p: \mathcal{D} \rightarrow \mathcal{S}$ be a strict functor of 2-multicategories. 
A 1-morphism
\[ \xi \in \Hom_f(\mathcal{E}_1, \dots, \mathcal{E}_n; \mathcal{F}) \] 
in $\mathcal{D}$ over $f \in \Hom(S_1, \dots, S_n; T)$
is called {\bf Cartesian} w.r.t.\@ $p$ and w.r.t.\@ the $i$-th slot, 
 if for all $\mathcal{G}_1,\dots,\mathcal{G}_m \in \mathcal{D}$ lying over $U_1, \dots, U_m \in \mathcal{S}$
 the diagram of categories
\[ \xymatrix{
\Hom_{\mathcal{D}}(\mathcal{G}_1, \dots, \mathcal{G}_m; \mathcal{E}_i) \ar[r]^-{ \xi \circ_i} \ar[d] & \Hom_{\mathcal{D}}(\mathcal{E}_1, \dots, \mathcal{E}_{i-1}, \mathcal{G}_1, \dots, \mathcal{G}_m, \mathcal{E}_{i+1}, \dots, \mathcal{E}_n; \mathcal{F}) \ar[d] \\
\Hom_{\mathcal{S}}(U_1, \dots, U_m; S_i) \ar[r]^-{ f \circ_i } & \Hom_{\mathcal{S}}(S_1, \dots, S_{i-1}, U_1, \dots, U_m, S_{i+1}, \dots, S_n; T)
} \]
is 2-Cartesian.  

A 1-morphism
\[ \xi \in \Hom_f(\mathcal{E}_1, \dots, \mathcal{E}_n; \mathcal{F}) \] 
 is called {\bf weakly Cartesian} w.r.t.\@ $p$ and the w.r.t.\@ $i$-th slot, if 
\[ \xymatrix{
\Hom_{\id_{S_i}}(\mathcal{G}; \mathcal{E}_i) \ar[r]^-{\xi \circ_i } & \Hom_{f}(\mathcal{E}_1, \dots, \mathcal{G}, \dots, \mathcal{E}_n; \mathcal{F})
} \]
is an equivalence of categories for all $\mathcal{G} \in \mathcal{D}$ with $p(\mathcal{G}) = S_i$. 
\end{DEF}

\begin{DEF}\label{DEF2FIB}
Let $p: \mathcal{D} \rightarrow \mathcal{S}$ be a strict functor of 2-multicategories. 
\begin{itemize}
\item $p$ is called a {\bf 1-opfibration of 2-multicategories} if for all 1-morphisms $f \in \Hom_{\mathcal{S}}(S_1, \dots, S_n; T)$ and all objects $\mathcal{E}_1, \dots, \mathcal{E}_n \in \mathcal{D}$ lying over $S_1, \dots, S_n \in \mathcal{S}$ there is an object $\mathcal{F} \in \mathcal{D}$ with $p(\mathcal{F}) = T$ and a coCartesian 1-morphism in $\Hom_f(\mathcal{E}_1, \dots, \mathcal{E}_n; \mathcal{F})$.

\item $p$ is called a {\bf 2-opfibration of 2-multicategories} if for $\mathcal{E}_1, \dots, \mathcal{E}_n; \mathcal{F} \in \mathcal{D}$ lying over $S_1, \dots, S_n; T \in \mathcal{S}$ the functors 
\[ \Hom_{\mathcal{D}}(\mathcal{E}_1, \dots, \mathcal{E}_n; \mathcal{F}) \rightarrow \Hom_{\mathcal{S}}(S_1, \dots, S_n; T) \]
are opfibrations, and the composition functors in $\mathcal{D}$ are morphisms of opfibrations, i.e.\@ if they map pairs of coCartesian 2-morphisms to coCartesian 2-morphisms.
 
\item $p$ is called a {\bf 1-fibration of 2-multicategories} if for all 1-morphisms $f \in \Hom_{\mathcal{S}}(S_1, \dots, S_n; T)$, for all $i=1, \dots, n$ and for all objects $\mathcal{E}_1, \overset{\widehat{i}}{\dots}, \mathcal{E}_n; \mathcal{F} \in \mathcal{D}$ lying over $S_1, \overset{\widehat{i}}{\dots}, S_n; T \in \mathcal{S}$ there is an object $\mathcal{E}_i \in \mathcal{D}$ with $p(\mathcal{E}_i)=S_i$ and a Cartesian 1-morphism w.r.t.\@ the $i$-th slot in $\Hom_f(\mathcal{E}_1, \dots, \mathcal{E}_n; \mathcal{F})$. 

\item $p$ is called a {\bf 2-fibration of 2-multicategories} if for $\mathcal{E}_1, \dots, \mathcal{E}_n; \mathcal{F} \in \mathcal{D}$ lying over $S_1, \dots, S_n; T \in \mathcal{S}$ the functors 
\[ \Hom_{\mathcal{D}}(\mathcal{E}_1, \dots, \mathcal{E}_n; \mathcal{F}) \rightarrow \Hom_{\mathcal{S}}(S_1, \dots, S_n; Y) \]
 are fibrations, and the composition functors in $\mathcal{D}$ are morphisms of fibrations, i.e.\@ if they map pairs of Cartesian 2-morphisms to Cartesian 2-morphisms.

\item Similarly we define the notions of {\bf 1-bifibration} and {\bf 2-bifibration}.

\item Let $S$ be an object in $\mathcal{S}$. The 2-category consisting of those objects, (1-ary) 1-morphisms, and 2-morphisms which $p$ maps to $S$, $\id_S$ and $\id_{\id_S}$ respectively is called the {\bf fiber} $\mathcal{D}_S$
of $p$ above $S$.

\item We say that $p$ has {\bf 1-categorical fibers}, if all fibers $\mathcal{D}_S$ are equivalent to 1-categories (this is also equivalent to all 2-morphism sets in the fibers being either empty or consisting of exactly one isomorphism). 

\item We say that $p$ has {\bf discrete fibers}, if all fibers $\mathcal{D}_S$ are equivalent to sets (this is also equivalent to all morphism categories in the fibers being either empty or
equivalent to the terminal category). 

\item $p$ is called a {\bf 2-isofibration} if $p$ induces a 2-fibration (or equivalently a 2-opfibration) when restricted to the strict 2-functor
\[ \mathcal{D}^{2-\sim} \rightarrow \mathcal{S}^{2-\sim} \]
where the 2-morphisms sets are the subsets of 2-isomorphisms in $\mathcal{D}$ and $\mathcal{S}$, respectively. 

Obviously every 2-fibration (or 2-opfibration) is a 2-isofibration. 
\end{itemize}
\end{DEF}

Note that $p$ is a {\bf 2-isofibration} precisely if the restriction $\mathcal{D}^{2-\sim} \rightarrow \mathcal{S}^{2-\sim}$ is full on 2-morphisms, i.e.\@ if
2-isomorphisms have a preimage under $p$. 



 For 2-isofibrations, by Lemma~\ref{LEMMACARTSQUAREISOFIB}, we could have defined (co)Cartesian 1-morphisms equivalently using the notion of Cartesian diagram instead of 2-Cartesian diagram. 

\begin{PROP}\label{LEMMAWEAKLY}
A 2-fibration or 2-opfibration of 2-multicategories
$p: \mathcal{D} \rightarrow \mathcal{S}$ is a 1-fibration if and only if the following two conditions hold:
\begin{enumerate}
\item For all 1-morphisms $f \in \Hom_{\mathcal{S}}(S_1, \dots, S_n; T)$ and all $i=1, \dots, n$ and all objects $\mathcal{E}_1, \overset{\widehat{i}}{\dots}, \mathcal{E}_n; \mathcal{F} \in \mathcal{D}$ with $p(\mathcal{E}_k)=S_k$ and $p(\mathcal{F})=T$ there is an object $\mathcal{E}_i$ with $p(\mathcal{E}_i)=S_i$ and a {\em weakly} Cartesian 1-morphism w.r.t.\@ the $i$-th slot in $\Hom_f(\mathcal{E}_1, \dots, \mathcal{E}_n; \mathcal{F})$;
\item The composition of weakly Cartesian 1-morphisms is weakly Cartesian. 
\end{enumerate}
\end{PROP}
A similar statement holds for 1-opfibrations where it is important that the Cartesian morphisms are composed w.r.t.\@ the correct slot (otherwise see \ref{LEMMACOMPOSCART}). 
\begin{proof}
Let $f \in \Hom_{\mathcal{S}}(S_1, \dots, S_n; T_i)$, and let $\xi \in \Hom_f(\mathcal{E}_1, \dots, \mathcal{E}_n; \mathcal{F}_i)$ be a weakly coCartesian morphism with $p(\xi) = f$. We have to show that $\xi$ is coCartesian.

By Lemma~\ref{LEMMACARTSQUAREISOFIB}, 3.\@, to prove that $p$ is a 1-fibration, it suffices to show that  
\[ \Hom_g(\mathcal{F}_1, \dots, \mathcal{F}_m; \mathcal{G}) \rightarrow \Hom_{g \circ_i f}(\mathcal{F}_1, \dots, \mathcal{F}_{i-1}, \mathcal{E}_1, \dots, \mathcal{E}_n, \mathcal{F}_{i+1}, \dots, \mathcal{F}_m; \mathcal{G}) \]
is an equivalence of categories for all $g \in \Hom(T_1, \dots, T_m; U)$. Now choose another weakly coCartesian 1-morphism
\[ \xi' \in \mathcal{F}_1, \dots, \mathcal{F}_m \rightarrow \mathcal{G}' \]
over $g$. 
We get the following sequence of functors
\[  \xymatrix{ \Hom_{\id_U}(\mathcal{G}'; \mathcal{G}) \ar[r]^-{\circ \xi'} &  \Hom_g(\mathcal{F}_1, \dots, \mathcal{F}_m; \mathcal{G})  \ar[r]^-{\circ \xi} &  \Hom_{g \circ f}(\mathcal{F}_1, \dots, \mathcal{F}_{i-1}, \mathcal{E}_1, \dots, \mathcal{E}_n, \mathcal{F}_{i+1}, \dots, \mathcal{F}_m; \mathcal{G}). } \]
Since the composition $\xi' \circ \xi$ is also weakly coCartesian the left functor {\em and} the composition are equivalences of categories. Hence
also the right functor is an equivalence. 

To show the converse, we show that coCartesian morphisms are weakly coCartesian. The following Lemma states that,
in general, coCartesian morphisms are stable under composition. 
Let $f \in \Hom_{\mathcal{S}}(S_1, \dots, S_n; T)$, and let $\xi \in \Hom_f(\mathcal{E}_1, \dots, \mathcal{E}_n; \mathcal{F})$ be a coCartesian morphism with $p(\xi) = f$.
In particular, the diagram 
\[ \xymatrix{
\Hom_{\mathcal{D}}(\mathcal{F}; \mathcal{G}) \ar[r]^-{\circ_i \xi} \ar[d] & \Hom_{\mathcal{D}}( \mathcal{E}_1, \dots, \mathcal{E}_n; \mathcal{G}) \ar[d] \\
\Hom_{\mathcal{S}}(U; U) \ar[r]^-{ \circ_i f} & \Hom_{\mathcal{S}}(S_1, \dots, S_n; U)
} \]
is 2-Cartesian and hence (this uses that we have a 2-isofibration) satisfies the statement of Lemma~\ref{LEMMACARTSQUAREISOFIB}, 2.\@ which implies that 
\[ \Hom_{\id_U}(\mathcal{F}; \mathcal{G}) \rightarrow \Hom_{f}( \mathcal{E}_1, \dots, \mathcal{E}_n; \mathcal{G}) \] 
is an equivalence. 
\end{proof}

\begin{LEMMA}\label{LEMMACOMPOSCART}
Let $p: \mathcal{D} \rightarrow \mathcal{S}$ be a strict functor between 2-multicategories. Then the composition of (co)Cartesian 1-morphisms (resp.\@ 2-morphisms) is (co)Cartesian.
For {\em Cartesian} 1-morphisms this holds true only if the slot used for the composition agrees with the slot at which the second morphism is Cartesian. Otherwise we have the following statement: 
If $\xi \in \Hom(\mathcal{E}_1, \dots, \mathcal{E}_n; \mathcal{F}_i)$ is a {\em coCartesian} 1-morphism and $\xi'  \in \Hom(\mathcal{F}_1, \dots, \mathcal{F}_m; \mathcal{G})$ is a {\em Cartesian} 1-morphism  w.r.t. the $j$-th slot ($i \not= j$)  then the composition
\[ \xi' \circ_j \xi \in \Hom_{\mathcal{D}}(\mathcal{F}_1, \dots, \mathcal{F}_{i-1}, \mathcal{E}_1, \dots, \mathcal{E}_n, \mathcal{F}_{i+1},  \dots, \mathcal{F}_m; \mathcal{G})  \]
is {\em Cartesian} w.r.t.\@ the $j$-th slot if $j<i$ and w.r.t.\@ the $j+n-1$-th slot if $j>i$. (This holds true in particular also in the case $n=0$). 
\end{LEMMA}

\begin{proof}
The 1-categorical statement is well-known, hence the composition of (co)Cartesian 2-morphisms is (co)Cartesian. 
We now show that the composition of coCartesian 1-morphisms is coCartesian. 
Let $f \in \Hom_{\mathcal{S}}(S_1, \dots, S_n, T_i)$ and $f' \in \Hom_{\mathcal{S}}(T_1, \dots, T_m, U_j)$ be arbitrary 1-morphisms in $\mathcal{S}$, and let
\[ \xi \in \Hom_f(\mathcal{E}_1, \dots, \mathcal{E}_n; \mathcal{F}_i) \] 
and
\[ \xi' \in \Hom_{f'}(\mathcal{F}_1, \dots, \mathcal{F}_m; \mathcal{G}_j) \]
be coCartesian morphisms.  
We want to show that their composition w.r.t.\@ the $i$-th-slot
\[ \xi' \circ_i \xi \in \Hom_{f' \circ_i f}(\mathcal{F}_1, \dots, \mathcal{F}_{i-1}, \mathcal{E}_1, \dots, \mathcal{E}_n, \mathcal{F}_{i+1}, \dots,  \mathcal{F}_m; \mathcal{G}_j) \]
is Cartesian.  

Let $\mathcal{G}_1, \overset{\widehat{j}}{\dots},  \mathcal{G}_k \in \mathcal{D}$ be objects lying over $U_1, \overset{\widehat{j}}{\dots}, U_k \in \mathcal{S}$, and let $\mathcal{H} \in \mathcal{D}$ an object over $V  \in \mathcal{S}$ (all arbitrary). 
Consider the diagram

\[ \xymatrix{
\Hom_{\mathcal{D}}\left({\begin{array}{c}\mathcal{G}_1\\ \dots\\ \mathcal{G}_k \end{array}}; \mathcal{H}\right)    \ar[r]^{\circ_j \xi'}  \ar[d] &
\Hom_{\mathcal{D}}\left({\begin{array}{c}\mathcal{G}_1\\ \dots\\ \mathcal{G}_{j-1}  \\ \mathcal{F}_1 \\ \dots \\ \mathcal{F}_m \\
\mathcal{G}_{j+1}\\ \dots\\ \mathcal{G}_{k} \end{array}}; \mathcal{H}\right) \ar[r]^{\circ_i \xi}  \ar[d]  &
\Hom_{\mathcal{D}}\left({\begin{array}{c}\mathcal{G}_1\\ \dots\\ \mathcal{G}_{j-1} \\ \mathcal{F}_1 \\ \dots \\ \mathcal{F}_{i-1} \\ \mathcal{E}_1 \\ \dots \\ \mathcal{E}_n \\ \mathcal{F}_{i+1} \\ \dots \\ 
 \mathcal{F}_m \\
\mathcal{G}_{j+1}\\ \dots\\ \mathcal{G}_{k} \end{array}}; \mathcal{H}\right)   \ar[d]  \\
\Hom_{\mathcal{S}}\left({\begin{array}{c}U_1\\ \dots\\ U_k \end{array}}; V \right)    \ar[r]^{\circ_j f'}   &
\Hom_{\mathcal{S}}\left({\begin{array}{c}U_1\\ \dots\\ U_{j-1}  \\ T_1 \\ \dots \\ T_m \\
U_{j+1}\\ \dots\\ U_{k} \end{array}}; V\right) \ar[r]^{\circ_i f}    &
\Hom_{\mathcal{S}}\left({\begin{array}{c}U_1\\ \dots\\ U_{j-1}  \\ T_1 \\ \dots \\ T_{i-1} \\ S_1 \\ \dots \\ S_n \\ T_{i+1} \\ \dots \\ 
 T_m \\
U_{j+1}\\ \dots\\ U_{k} \end{array}}; V\right)   
} 
\]
The right hand square is 2-Cartesian because $\xi$ is coCartesian, and the left square is 2-Cartesian because $\xi'$ is coCartesian. Hence also the composed square is 2-Cartesian, i.e. $\xi' \circ \xi$ is coCartesian as well. 

 The assertion about the composition of 1-Cartesian morphisms is proven in the same way. 
 For the additional statement, let $f \in \Hom_{\mathcal{S}}(S_1, \dots, S_n, T_i)$ and $f' \in \Hom_{\mathcal{S}}(T_1, \dots, T_m, U)$ be arbitrary 1-morphisms in $\mathcal{S}$, and let
\[ \xi \in \Hom_f(\mathcal{E}_1, \dots, \mathcal{E}_n; \mathcal{F}_i) \] 
be coCartesian (here $n=0$ is possible) and
\[ \xi' \in \Hom_{f'}(\mathcal{F}_1, \dots, \mathcal{F}_m; \mathcal{G}) \]
be Cartesian w.r.t.\@ to the slot $j \not= i$. To fix notation assume $i<j$. 

We want to show that their composition w.r.t.\@ the $i$-th-slot
\[ \xi' \circ_i \xi \in \Hom_{f' \circ_i f}(\mathcal{F}_1, \dots, \mathcal{F}_{i-1}, \mathcal{E}_1, \dots, \mathcal{E}_n, \mathcal{F}_{i+1}, \dots,  \mathcal{F}_m; \mathcal{G}) \]
is Cartesian w.r.t.\@ to the slot $j+n-1$. 

Let $\mathcal{E}_1', \dots,  \mathcal{E}_k' \in \mathcal{D}$ be objects lying over $S_1', \dots, S_k' \in \mathcal{S}$ (all arbitrary). 
Consider the diagram

\[ \xymatrix{
\Hom_{\mathcal{D}}\left({\begin{array}{c}\mathcal{E}_1'\\ \dots\\ \mathcal{E}_k' \end{array}}; \mathcal{F}_j\right)    \ar[r]^{\xi' \circ_j}  \ar[d] &
\Hom_{\mathcal{D}}\left({\begin{array}{c}\mathcal{F}_1\\ \dots\\ \mathcal{F}_{j-1}  \\ \mathcal{E}_1' \\ \dots \\ \mathcal{E}_k' \\
\mathcal{F}_{j+1}\\ \dots\\ \mathcal{F}_{m} \end{array}}; \mathcal{G}\right) \ar[r]^{\circ_i \xi}  \ar[d]  &
\Hom_{\mathcal{D}}\left({\begin{array}{c}\mathcal{F}_1\\ \dots\\ \mathcal{F}_{i-1} \\ \mathcal{E}_1 \\ \dots \\ \mathcal{E}_{n} \\ \mathcal{F}_{i+1} \\ \dots \\ \mathcal{F}_{j-1} \\ \mathcal{E}_1' \\ \dots \\ \mathcal{E}_k'  \\ \mathcal{F}_{j+1} \\ \dots \\ 
 \mathcal{F}_m \end{array}};  \mathcal{G} \right)   \ar[d]  \\
\Hom_{\mathcal{S}}\left({\begin{array}{c}S_1'\\ \dots\\ S_k' \end{array}}; T_j\right)    \ar[r]^{f' \circ_j}  &
\Hom_{\mathcal{S}}\left({\begin{array}{c}T_1\\ \dots\\ T_{j-1}  \\ S_1' \\ \dots \\ S_k' \\
T_{j+1} \\ \dots\\ T_{m} \end{array}}; U\right) \ar[r]^{\circ_i f}    &
\Hom_{\mathcal{S}}\left({\begin{array}{c}T_1\\ \dots\\ T_{i-1} \\ S_1 \\ \dots \\ S_{n} \\ T_{i+1} \\ \dots \\ T_{j-1} \\ S_1' \\ \dots \\ S_k'  \\ T_{j+1} \\ \dots \\ 
 T_m \end{array}};  U \right)   
} \]
Now note that the composed functor
\[ \rho \mapsto (\xi' \circ_j \rho) \circ_i \xi   \]
is the same as
\[ \rho \mapsto    (\xi' \circ_i \xi) \circ_{j+n-1}  \rho  \]
because of the independence of slots (analogously for the bottom line functors). The right hand square is 2-Cartesian because $\xi$ is coCartesian, and the left square is 2-Cartesian because $\xi'$ is Cartesian w.r.t.\@ the $i$-th slot. Hence also the composed square is 2-Cartesian, i.e. $\xi' \circ_j \xi$ is Cartesian w.r.t.\@ the slot $i+n-1$ as well.  
\end{proof}

\begin{PAR}\label{PARPSEUDONATEQUIVALENCE}
Recall the definition of pseudo-functor between strict 2-categories, pseudo-natural transformations, and modifications (Definitions~\ref{DEFPSEUDOFUN}--\ref{DEFMODIFICATION}). 
Let $F$, $G$ be pseudo-functors from a 2-category $\mathcal{D}$ to a 2-category $\mathcal{D}'$.
A pseudo-natural transformation $\xi: F \rightarrow G$ is called an {\bf equivalence} if there are a pseudo-natural transformation
$\eta: G \rightarrow F$, and modifications (isomorphisms) $\xi \circ \eta \cong \id_G$, and $\eta \circ \xi \cong \id_F$. 
\end{PAR}
\begin{LEMMA}
A pseudo-natural transformation $\xi: F \rightarrow G$ is an equivalence if and only if for all $\mathcal{E} \in \mathcal{D}$
\[ \xi_{\mathcal{E}}: F(\mathcal{E}) \rightarrow G(\mathcal{E}) \]
 is an equivalence in the target-2-category $\mathcal{D}'$. In other words,  choosing a point-wise inverse sets up automatically a pseudo-natural transformation as well, and the
point-wise natural transformations between the compositions constitute the required modifications. 
\end{LEMMA}

\begin{proof}
The ``only if'' implication is clear. For the ``if'' part choose a quasi-inverse $\xi'(\mathcal{E}): G(\mathcal{E}) \rightarrow F(\mathcal{E})$  to $\xi(\mathcal{E}): F(\mathcal{E}) \rightarrow G(\mathcal{E})$ for all objects $\mathcal{E} \in \mathcal{D}$. 
Hence, for all  $\mathcal{E} \in \mathcal{D}$, we can find 
isomorphisms $\id_{G(\mathcal{E})} \Rightarrow \xi(\mathcal{E}) \circ \xi'(\mathcal{E})$ and  $\xi'(\mathcal{E}) \circ \xi(\mathcal{E}) \Rightarrow \id_{F(\mathcal{E})}$ satisfying the unit-counit equations. 
Let $f: \mathcal{E} \rightarrow \mathcal{F}$ be a 1-morphism in $\mathcal{D}$. 
Define $\xi'_f$ to be the following composition: 
\[ \xi'(\mathcal{F}) \circ G(f) \Rightarrow \xi'(\mathcal{F}) \circ G(f) \circ \xi(\mathcal{E}) \circ \xi'(\mathcal{E})  \Leftarrow \xi'(\mathcal{F}) \circ \xi(\mathcal{F}) \circ F(f) \circ \xi'(\mathcal{E}) \Rightarrow  F(f) \circ \xi'(\mathcal{E}). \] 
We leave to reader to check that this defines indeed a pseudo-natural transformation. 
\end{proof}

\begin{DEF}
Recall that an object $\mathcal{E}$ in a strict 2-category defines a strict 2-functor
\begin{eqnarray*} 
\Hom(\mathcal{E},-): \mathcal{D} &\rightarrow &\mathcal{CAT}  \\
\mathcal{F}& \mapsto & \Hom(\mathcal{E},\mathcal{F})
\end{eqnarray*}

A pseudo-functor from a 2-category $\mathcal{D}$
\[ F: \mathcal{D} \rightarrow \mathcal{CAT} \]
is called {\bf representable} if there is an object $\mathcal{E}$ and a pseudo-natural transformation
\[ \nu: F \rightarrow \Hom(\mathcal{E}, -) \]
which is an equivalence, cf.\@ \ref{PARPSEUDONATEQUIVALENCE}.
\end{DEF}

\begin{LEMMA}\label{LEMMAREPREQUIV}
An object $\mathcal{E}$ which represents a functor $F$ is determined up to equivalence.  
\end{LEMMA}
\begin{proof}
We have to show that every pseudo-natural transformation
\[ \xi: \Hom(\mathcal{E}, -)  \rightarrow \Hom(\mathcal{E}', -) \]
which has an inverse up to modification, induces an equivalence $\mathcal{E} \rightarrow \mathcal{E}'$. Let $\eta$ be the quasi-inverse of $\xi$. We have a 2-commutative diagram 
\[ \xymatrix{
\Hom(\mathcal{E}, \mathcal{E}') \ar[r]^{\xi_{\mathcal{E}'}} \ar[d]_{\xi_{\mathcal{E}}(\id_{\mathcal{E}}) \circ -} 
\ar@{}[rd]|{\Downarrow^\sim} & \Hom(\mathcal{E}', \mathcal{E}') \ar[d]^{\xi_{\mathcal{E}}(\id_{\mathcal{E}}) \circ -} \\
\Hom(\mathcal{E}, \mathcal{E}) \ar[r]_{\xi_{\mathcal{E}}} & \Hom(\mathcal{E}', \mathcal{E})
} \]
by the definition of pseudo-natural transformation. Hence also a 2-commutative diagram: 
\[ \xymatrix{
\Hom(\mathcal{E}, \mathcal{E}')\ar[d]_{\xi_{\mathcal{E}}(\id_{\mathcal{E}}) \circ -}  
\ar@{}[rd]|{\Downarrow^\sim} & \Hom(\mathcal{E}', \mathcal{E}') \ar[d]^{\xi_{\mathcal{E}}(\id_{\mathcal{E}})  \circ -}  \ar[l]_{\eta_{\mathcal{E}'}}  \\
\Hom(\mathcal{E}, \mathcal{E})  & \Hom(\mathcal{E}', \mathcal{E}) \ar[l]^{\eta_{\mathcal{E}}}
} \]
In particular, we get 2-isomorphisms
\[ \xi_{\mathcal{E}}(\id_{\mathcal{E}})  \circ \eta_{\mathcal{E}'}(\id_{\mathcal{E}'}) \Rightarrow \eta_{\mathcal{E}}(\xi_{\mathcal{E}}(\id_{\mathcal{E}})) \Rightarrow \id_{\mathcal{E}}    \]
where the second one comes from the fact that $\eta$ and $\xi$ are inverse to each other up to 2-isomorphism. 
Similarly, there is a 2-isomorphism
\[  \eta_{\mathcal{E}'}(\id_{\mathcal{E}'}) \circ  \xi_{\mathcal{E}}(\id_{\mathcal{E}})  \Rightarrow  \id_{\mathcal{E}'}.    \]
Hence we get the required equivalence
\[ \xymatrix{ \mathcal{E} \ar@/^3pt/[rr]^{\eta_{\mathcal{E}'}(\id_{\mathcal{E}'})} & & \ar@/^3pt/[ll]^{ \xi_{\mathcal{E}}(\id_{\mathcal{E}})} \mathcal{E}'  } \]
\end{proof}

The previous lemma shows that the following definition makes sense:
\begin{DEF}\label{DEFPUSHFWD}
\begin{enumerate}
\item Let $p: \mathcal{D} \rightarrow \mathcal{S}$ be a strict functor of 2-multicategories which is a 1-opfibration and 2-isofibration. The target object $\mathcal{F}$ of a coCartesian 1-morphism (cf.\@ Definition~\ref{DEFCOCART}) starting from $\mathcal{E}_1, \dots, \mathcal{E}_n$ and lying over a 1-multimorphism $f \in \Hom(S_1, \dots, S_n; T)$ in $\mathcal{S}$
is denoted by $f_\bullet(\mathcal{E}_1, \dots, \mathcal{E}_n)$.
\item Let $p: \mathcal{D} \rightarrow \mathcal{S}$ be a strict functor of 2-multicategories which is a 1-fibration and 2-isofibration. The $i$-th source object $\mathcal{F}$ of a Cartesian 1-morphism w.r.t.\@ to the $i$-th slot (cf.\@ Definition~\ref{DEFCART}) starting from $\mathcal{E}_1, \overset{\widehat{i}}{\dots}, \mathcal{E}_n$ with target $\mathcal{F}$ and lying over a 1-multimorphism $f \in \Hom(S_1, \dots, S_n; T)$ in $\mathcal{S}$
is denoted by $f^{\bullet,i}(\mathcal{E}_1, \overset{\widehat{i}}{\dots}, \mathcal{E}_n; \mathcal{F})$.
\end{enumerate}
In both cases the objects are uniquely determined up to equivalence in $\mathcal{D}_T$. 
\end{DEF}
Note that for two different objects $f_\bullet(\mathcal{E}_1, \dots, \mathcal{E}_n)$ and $f_\circ(\mathcal{E}_1, \dots, \mathcal{E}_n)$ each representing the 2-functor
\[ \mathcal{F} \mapsto \Hom_f(\mathcal{E}_1, \dots, \mathcal{E}_n; \mathcal{F})  \]
on the 2-category $\mathcal{D}_T$, we get an equivalence $f_\bullet(\mathcal{E}_1, \dots, \mathcal{E}_n) \leftrightarrow f_\circ(\mathcal{E}_1, \dots, \mathcal{E}_n)$ by  Lemma~\ref{LEMMAREPREQUIV}.

\begin{PAR}\label{PARCATMULTI}
The 2-category
$\mathcal{CAT}$ has a natural structure of a symmetric 2-multicategory setting
\[ \Hom(\mathcal{C}_1, \dots, \mathcal{C}_n; \mathcal{D}) :=  \Fun(\mathcal{C}_1 \times \cdots \times \mathcal{C}_n, \mathcal{D}). \]
$\mathcal{CAT}$ is obviously opfibered over $\{\cdot\}$ with the monoidal product given by the product of categories and with the final category as neutral element. 
\end{PAR}

\begin{DEF}[2-categorical Grothendieck construction] \label{GROTHCONSTR}
For a pseudo-functor of 2-multicategories
\[ \Xi: \mathcal{S} \rightarrow \mathcal{CAT} \]
where $\mathcal{CAT}$ is equipped with the structure of 2-multicategory of \ref{PARCATMULTI},
we get a 2-multicategory $\int \Xi$ and a strict functor
\[ \int \Xi \rightarrow \mathcal{S} \]
which is 1-opfibered and 2-fibered and
whose fiber over $S \in \mathcal{S}$ is isomorphic to $\Xi(S)$ (hence it is a 1-category), as follows: 
The objects of  $\int \Xi$ are pairs 
\[ (\mathcal{E}, S) \]
where $S$ is an object of $\mathcal{S}$, and $\mathcal{E}$ is an object of $\Xi(S)$. The 1-morphisms in
\[ \Hom_{\int \Xi}( (\mathcal{E}_1, S_1), \dots, (\mathcal{E}_n, S_n);  (\mathcal{F}, T) ) \]
are pairs $(\alpha, f)$ where $f \in \Hom(S_1, \dots, S_n; T)$ is a 1-morphism in $\mathcal{S}$  and $\alpha: \Xi(f)(\mathcal{E}_1, \dots, \mathcal{E}_n) \rightarrow \mathcal{F}$ is a morphism in $\Xi(T)$. The
2-morphisms
\[ \nu: (\alpha, f) \Rightarrow (\beta, g) \] 
are those 2-morphisms $\nu: f \Rightarrow g$ such that $\beta \circ \Xi(\nu) = \alpha$.

Similarly there is a Grothendieck construction which starts from a pseudo-functor of 2-multicategories 
\[ \Xi: \mathcal{S}^{2-\op} \rightarrow \mathcal{CAT} \]
and produces a 1-opfibration and 2-opfibration.
\end{DEF}

\begin{PAR}\label{GROTHCONSTRDUAL}
There is also a Grothendieck construction which starts from a pseudo-functor of 2-categories (not 2-multicategories)
\[ \Xi: \mathcal{S}^{1-\op} \rightarrow \mathcal{CAT} \]
and produces a 1-fibration and 2-opfibration $\nabla \Xi \rightarrow \mathcal{S}$, or from a pseudo-functor
\[ \Xi: \mathcal{S}^{1-\op, 2-\op} \rightarrow \mathcal{CAT} \]
respectively, and produces a 1-fibration and 2-fibration $\nabla \Xi \rightarrow \mathcal{S}$. 
A 1-fibration of (2-){\em multi}\,categories cannot be so easily described by a pseudo-functor because one gets several pullback functors depending on the slot (e.g.\@ $\mathcal{HOM}_l, \mathcal{HOM}_r$).   
\end{PAR}

\begin{PROP}\label{PROPGROTHCONSTR}
For a strict functor between 2-multicategories $p: \mathcal{D}\rightarrow \mathcal{S}$ which is 1-opfibered and 2-fibered {\em with 1-categorical fibers}, we get an associated
pseudo-functor of 2-multicategores:
\begin{eqnarray*} 
\Xi_{\mathcal{D}}: \mathcal{S}& \rightarrow& \mathcal{CAT}  \\
 S &\mapsto& \mathcal{D}_S 
\end{eqnarray*}
The construction is inverse (up to isomorphism of pseudo-functors, resp.\@ 1-opfibrations/2-fibrations) to the one given in the previous definition. 
\end{PROP}

An analogous proposition is true for 1-(op)fibrations and 2-(op)fibration, with the restriction that for 1-fibrations the multi-aspect has to be neglected. 

\begin{proof}[Proof (Sketch).]
The pseudo-functor $\Xi_{\mathcal{D}}$ maps a 1-morphism $f: S_1, \dots, S_n \rightarrow T$ to the functor (cf.\@ Definition~\ref{DEFPUSHFWD})
\[ f_\bullet(-, \dots, -):  \mathcal{D}_{S_1} \times \dots \times   \mathcal{D}_{S_n} \rightarrow  \mathcal{D}_{T}. \]
A 2-morphism $\nu: f \Rightarrow g$ is mapped to the following natural transformation between $f_\bullet(-, \dots, -)$ and $g_\bullet(-, \dots, -)$.
With the definition (or characterization) of $f_\bullet(-, \dots, -)$ there comes a natural equivalence of {\em discrete} categories
\begin{equation}\label{eqnat} 
\Hom_f(\mathcal{E}_1, \dots, \mathcal{E}_n; \mathcal{F}) \rightarrow \Hom_{\mathcal{D}_T}(f_\bullet(\mathcal{E}_1, \dots, \mathcal{E}_n); \mathcal{F}). 
\end{equation}
Because $p$ is 2-fibered and any 2-isomorphism is Cartesian, $\nu$ induces a well-defined isomorphism
\[ \Hom_f(\mathcal{E}_1, \dots, \mathcal{E}_n; \mathcal{F}) \cong \Hom_g(\mathcal{E}_1, \dots, \mathcal{E}_n; \mathcal{F}). \]
Since this is true for any $\mathcal{F}$, using the natural equivalences (\ref{eqnat}) for $f$ and $g$, we get a morphism in $\mathcal{D}_T$
\[ f_\bullet(\mathcal{E}_1, \dots, \mathcal{E}_n) \rightarrow g_\bullet(\mathcal{E}_1, \dots, \mathcal{E}_n). \]
One checks that this defines a natural transformation and that the whole construction $\Xi$ is indeed a pseudo-functor of 2-multicategories. 
\end{proof}

\begin{KOR}
The concept of functor between 1-multicategories $p: \mathcal{D} \rightarrow \{\cdot\}$ which are (1\nobreakdash-)opfibered is equivalent to the concept of a monoidal category. 
The functor is, in addition, (1-)fibered if the corresponding monoidal category is closed. 
\end{KOR}

\begin{PAR}
For a strict functor between 2-multicategories $p: \mathcal{D}\rightarrow \mathcal{S}$ which is 1-opfibered and 2-fibered but {\em with arbitrary 2-categorical fibers}, and every
$f \in \Hom_{\mathcal{S}}(S_1, \dots, S_n; T)$ and $\mathcal{E}_1, \dots, \mathcal{E}_n$ we get still an object
\[ f_\bullet(\mathcal{E}_1, \dots, \mathcal{E}_n)  \]
which is well-defined up to equivalence. This defines a certain kind of pseudo-3-functor
\[  \mathcal{S} \rightarrow 2-\mathcal{CAT}. \]
Since this becomes confusing and we will not need it, we will not go into any details of this. 
For example, if $\mathcal{S} = \{ \cdot\}$ then a 2-multicategory $\mathcal{D}$ which  is 1-opfibered and 2-fibered over $\{ \cdot\}$ is the same
as a {\bf monoidal 2-category} in the sense of \cite{KV94, Shu10, Gur06, GPS95}. The (symmetric) prototype here is $\mathcal{CAT}$ with the structure
of 2-multicategory considered above. 
\end{PAR}

\begin{BEISPIEL}\label{EXMULTICATS}
Let $\mathcal{S}$ be a usual category. Then $\mathcal{S}$ may be turned into a symmetric
multicategory by setting 
\[ \Hom(S_1, \dots, S_n; T) := \Hom(S_1; T) \times \cdots \times \Hom(S_n; T). \]
If $\mathcal{S}$ has coproducts, then $\mathcal{S}$ (with this multicategory structure) is opfibered over $\{\cdot\}$.
Let $p: \mathcal{D} \rightarrow \mathcal{S}$ be an opfibered (usual) category. 
Then a multicategory structure on $\mathcal{D}$ which turns $p$ into an opfibration w.r.t.\@ this multicategory structure on  $\mathcal{S}$, is equivalent  to a monoidal structure on the fibers of $p$ such that the push-forwards $f_\bullet$ are monoidal functors and such that the compatibility morphisms between them are morphisms of monoidal functors. This is called a {\em covariant monoidal pseudo-functor} in \cite[(3.6.7)]{LH09}.
\end{BEISPIEL}
\begin{BEISPIEL}\label{EXMULTICATSOP}
Let $\mathcal{S}$ be a usual category.  Then $\mathcal{S}^{\mathrm{op}}$ may be turned into a symmetric 
multicategory (or equivalently $\mathcal{S}$ into a symmetric opmulticategory) by setting 
\[ \Hom(S_1, \dots, S_n; T) := \Hom(T; S_1) \times \cdots \times \Hom(T; S_n). \]
If $\mathcal{S}$ has products then $\mathcal{S}^{\mathrm{op}}$ (with this multicategory structure)  is opfibered over $\{\cdot\}$.
Let $p: \mathcal{D} \rightarrow \mathcal{S}^{\mathrm{op}}$ be an opfibered (usual) category. 
Then a multicategory structure on $\mathcal{D}$ which turns $p$ into an opfibration w.r.t.\@ this multicategory structure on $\mathcal{S}^{\op}$, is equivalent to a monoidal structure on the fibers of $p$ such that the pull-backs $f^*$ (along morphisms in $\mathcal{S}$) are monoidal functors and such that the compatibility morphisms between them are morphisms of monoidal functors. This is called a {\em contravariant monoidal pseudo-functor} in \cite[(3.6.7)]{LH09}.
\end{BEISPIEL}

\begin{LEMMA}\label{LEMMAEQUIVALENCECARTCOCART}
Let $p: \mathcal{D} \rightarrow \mathcal{S}$ be a strict functor of 2-multicategories. Any equivalence in $\mathcal{D}$ is a Cartesian and coCartesian 1-morphism. 
\end{LEMMA}
\begin{proof}
An equivalence $\mathcal{F} \rightarrow \mathcal{F'}$ has the property that the composition
\[ \Hom_{\mathcal{D}}(\mathcal{E}_1, \dots, \mathcal{E}_n; \mathcal{F}) \rightarrow \Hom_{\mathcal{D}}(\mathcal{E}_1, \dots, \mathcal{E}_n; \mathcal{F}')  \]
is an equivalence of categories for all objects $\mathcal{E}_1, \dots, \mathcal{E}_n$ of $\mathcal{D}$. We hence get a commutative diagram of categories
\[ \xymatrix{
\Hom_{\mathcal{D}}(\mathcal{E}_1, \dots, \mathcal{E}_n; \mathcal{F}) \ar[r] \ar[d] & \Hom_{\mathcal{D}}(\mathcal{E}_1, \dots, \mathcal{E}_n; \mathcal{F}') \ar[d] \\
\Hom_{\mathcal{S}}(S_1, \dots, S_n; T) \ar[r] & \Hom_{\mathcal{S}}(S_1, \dots, S_n; T')
} \]
where the two horizontal morphisms are equivalences. It is automatically 2-Cartesian. 
\end{proof}

\begin{LEMMA}\label{LEMMA1COCART2ISO}
Let $p: \mathcal{D} \rightarrow \mathcal{S}$ be a strict functor of 2-multicategories. If $\xi \in \Hom_{\mathcal{D}}(\mathcal{E}_1, \dots, \mathcal{E}_n; \mathcal{F})$ is a (co)Cartesian 1-morphism and
$\alpha: \xi \Rightarrow \xi'$ is a 2-isomorphism in $\mathcal{D}$, then $\xi'$ is (co)Cartesian as well.
\end{LEMMA}
\begin{proof}
The 2-isomorphism $\alpha$ induces a natural isomorphism between the functor `composition with $\xi$' and the functor `composition with $\xi'$'. 
And $p(\alpha)$ induces a natural isomorphism between the functor `composition with $p(\xi)$' and the functor `composition with $p(\xi')$'. 
Therefore the
diagram expressing the coCartesianity of $\xi$ is 2-Cartesian if and only if the corresponding diagram for $\xi'$ is 2-Cartesian. 
\end{proof}

\begin{PAR}\label{DEFPULLBACK}
Consider 2-multicategories $\mathcal{D}$, $\mathcal{S}$, $\mathcal{S}'$ and a diagram
\[ \xymatrix{
 & \mathcal{D} \ar[d]^p \\
\mathcal{S'} \ar[r]_F & \mathcal{S}
} \]
where $p$ is a strict 2-functor and $F$ is a pseudo-functor. 
We define the {\bf pull-back} of $p$ along $F$ as the following 2-multicategory $F^*\mathcal{D}$:
\begin{enumerate}
\item The objects of $F^*\mathcal{D}$ are pairs of objects $\mathcal{F} \in \mathcal{D}$ and $S \in \mathcal{S}'$ such that $p(\mathcal{F}) = F(S)$.
\item The 1-morphisms $(S_1, \mathcal{F}_1), \dots, (S_n, \mathcal{F}_n) \rightarrow (T, \mathcal{G})$  are pairs consisting of a 1-morphism $\alpha \in \Hom_{\mathcal{D}}(\mathcal{F}_1, \dots, \mathcal{F}_n; \mathcal{G})$ and a 1-morphism $\beta \in \Hom_{\mathcal{S}'}(S_1, \dots, S_n;  T)$ and a 2-isomorphism
\[ \xymatrix{
(p(\mathcal{F}_1), \dots, p(\mathcal{F}_n)) \ar[rr]^-{p(\alpha)} \ar@{=}[d] \ar@{}[rrd]|-{\Downarrow^\gamma} && p(\mathcal{G}) \ar@{=}[d]  \\
(F(S_1), \dots, F(S_n)) \ar[rr]_-{F(\beta)} && F(T) 
} \]

\item The 2-morphisms $(\alpha, \beta, \gamma) \Rightarrow (\alpha', \beta', \gamma')$ are 2-morphisms $\mu: \alpha \Rightarrow \alpha'$ and $\nu: \beta \Rightarrow \beta'$ such that $\gamma' p(\mu) = F(\nu) \gamma$. 
\item Composition for the $\gamma$'s is given by the following pasting (here depicted for 1-ary morphisms):
\[ \xymatrix{
p(\mathcal{F}) \ar[r]^-{p(\alpha_1)} \ar@/^20pt/[rr]^{p(\alpha_2 \alpha_1)} \ar@{=}[d] \ar@{}[rd]|-{\Downarrow^{\gamma_1}} & p(\mathcal{F}') \ar[r]^-{p(\alpha_2)}  \ar@{=}[d] \ar@{}[rd]|-{\Downarrow^{\gamma_2}} & p(\mathcal{F}'') \ar@{=}[d]  \\
F(S) \ar[r]^-{F(\beta_1)} \ar@/_20pt/[rr]_{F(\beta_2 \beta_1)}^{\Downarrow^{F_{\beta_2,\beta_1}}}  & F(S')  \ar[r]^-{F(\beta_2)} & F(S'') 
} \]
Here $F_{\beta_2,\beta_1}$ is the 2-isomorphism given by the pseudo-functoriality of $F$ (cf.\@ Definition~\ref{DEFPSEUDOFUN}). Associativity follows from the axioms of a pseudo-functor. 
\end{enumerate}

We get a commutative diagram of 2-multicategories in which the vertical 2-functors are strict:
\[ \xymatrix{
F^* \mathcal{D} \ar[d]_-{F^*p} \ar[r] & \mathcal{D} \ar[d]^-p \\
\mathcal{S}' \ar[r]_F & \mathcal{S}
} \]
\end{PAR}

\begin{PROP}\label{PROPPULLBACK}
If $p$ is a 1-fibration (resp.\@ 1-opfibration, resp.\@ 2-fibration, resp.\@ 2-opfibration) then $F^*p$ is a 1-fibration (resp.\@ 1-opfibration, resp.\@ 2-fibration, resp.\@ 2-opfibration). 
\end{PROP}
\begin{proof}
We show the proposition for 1-opfibrations and 2-opfibrations. The other assertions are shown similarly.
Consider the diagram 
\[ \xymatrix{
\Hom_{F^*\mathcal{D}}((S_1,\mathcal{F}_1), \dots, (S_m,\mathcal{F}_m); (T,\mathcal{G})) \ar[r] \ar[d]_{F^*p} & \Hom_{\mathcal{D}}(\mathcal{F}_1, \dots, \mathcal{F}_m; \mathcal{G}) \ar[d]^p \\
\Hom_{\mathcal{S}'}(S_1, \dots, S_m; T) \ar[r]^-F & \Hom_{\mathcal{S}}(F(S_1), \dots, F(S_m); F(T))
} \]
where $S_1, \dots, S_m, T$ are objects of $\mathcal{S}'$ and $\mathcal{F}_1, \dots, \mathcal{F}_m, \mathcal{G}$ are objects of $\mathcal{D}$ such that $F(S_i) = p(\mathcal{F}_i)$ and $F(T) = p(\mathcal{G})$. By definition of pull-back this diagram is 2-Cartesian. 

Hence if $p$ is an opfibration then so is $F^*p$. Furthermore a 2-morphism in $F^*\mathcal{D}$, i.e.\@ a morphism in the category $\Hom_{F^*\mathcal{D}}((T_1,\mathcal{F}_1), \dots, (T_m,\mathcal{F}_m); (U,\mathcal{G}))$
is coCartesian for $F^*p$ if and only if its image in $\Hom_{\mathcal{D}}(\mathcal{F}_1, \dots, \mathcal{F}_m; \mathcal{G})$ is coCartesian for $p$. 

Let $f \in \Hom(S_1, \dots, S_n; T_i)$ be a 1-morphism in $\mathcal{S}'$ and $\mathcal{E}_1, \dots, \mathcal{E}_n$ be objects of $\mathcal{D}$ such that $F(S_i) = p(\mathcal{E}_i)$.
Choose a coCartesian 1-morphism $\xi \in \Hom_{\mathcal{D}}(\mathcal{E}_1, \dots, \mathcal{E}_n; \mathcal{F}_i)$ over $F(f)$ and consider the corresponding morphism
\[ (\xi,f) \in \Hom_{F^*\mathcal{D}}((S_1,\mathcal{E}_1), \dots, (S_m,\mathcal{E}_m); (T_i,\mathcal{F}_i)) \]
over $f$. We will show that the 1-morphism $(\xi,f)$ is coCartesian for $F^*\mathcal{D} \rightarrow \mathcal{S}'$.
Consider the following 2-commutative diagram of categories (we omitted the natural isomorphisms which occur in the left, right, bottom and top faces): 
{\tiny
\[ \xymatrix{
&\Hom_{\mathcal{D}}\left({\begin{array}{c}\mathcal{F}_1\\ \dots\\ \mathcal{F}_m\end{array}}; \mathcal{G}\right) \ar[rr]^-{\circ_i \xi} \ar[dd] & & \Hom_{\mathcal{D}}\left({\begin{array}{c}\mathcal{F}_1\\ \dots\\ \mathcal{F}_{i-1}\\ \mathcal{E}_1\\ \dots\\ \mathcal{E}_n\\ \mathcal{F}_{i+1}\\ \dots\\ \mathcal{F}_m\end{array}}; \mathcal{G}\right) \ar[dd] \\
\Hom_{F^*\mathcal{D}}\left({\begin{array}{c}(T_1,\mathcal{F}_1)\\ \dots\\ (T_m,\mathcal{F}_m)\end{array}}; (U,\mathcal{G})\right) \ar[ru] \ar[dd] \ar[rr]^-(.3){\circ_i (\xi,f)} &&\Hom_{F^*\mathcal{D}}\left({\begin{array}{c}(T_1,\mathcal{F}_1)\\ \dots\\ (T_{i-1}, \mathcal{F}_{i-1})\\ (S_1,\mathcal{E}_1)\\ \dots\\ (S_n, \mathcal{E}_n)\\ (T_{i+1}, \mathcal{F}_{i+1})\\ \dots\\ (T_m, \mathcal{F}_m)\end{array}}; (U, \mathcal{G})\right) \ar[dd] \ar[ru]   \\
&\Hom_{\mathcal{S}}\left({\begin{array}{c}F(T_1)\\ \dots\\ F(T_m)\end{array}}; F(U)\right) \ar[rr]^-(.3){ \circ_i F(f)} & & \Hom_{\mathcal{S}}\left({\begin{array}{c}F(T_1)\\ \dots\\ F(T_{i-1})\\ F(S_1)\\ \dots\\ F(S_n)\\ F(T_{i+1})\\ \dots\\ F(T_m)\end{array}}; F(U)\right) \\
\Hom_{\mathcal{S}'}\left({\begin{array}{c}T_1\\ \dots\\ T_m\end{array}}; U\right) \ar[ru]  \ar[rr]^-{ \circ_i f} && \Hom_{\mathcal{S}'}\left({\begin{array}{c}T_1\\ \dots\\ T_{i-1}\\ S_1\\ \dots\\ S_n\\ T_{i+1}\\ \dots\\ T_m\end{array}}; U\right) \ar[ru] 
} \]}

The back face of the cube is 2-Cartesian by the definition of coCartesian for $\xi$. The left and right face of the cube are 2-Cartesian by the definition of pull-back. Therefore also the front face is 2-Cartesian, and hence $(\xi,f)$ is a Cartesian 1-morphism. 

Furthermore, for the composition with any (not necessarily coCartesian) 1-morphism we may draw a similar diagram and have to show that if the top horizonal functor in the back face is a morphism of opfibrations then the front face is a morphism of opfibrations. This follows from the characterization of coCartesian 2-morphisms given in the beginning of the proof. 
\end{proof}

\begin{PROP}\label{PROPBIFIBTRANSITIVITY}
If $p_1: \mathcal{E} \rightarrow \mathcal{D}$ and $p_2: \mathcal{D} \rightarrow \mathcal{S}$ are 1-fibrations (resp.\@ 1-opfibrations, resp.\@ 2-fibrations, resp.\@ 2-opfibrations) 
of 2-multicategories then the composition
$p_2 \circ p_1: \mathcal{E} \rightarrow \mathcal{S}$ is a 1-fibration, (resp.\@ 1-opfibration, resp.\@ 2-fibration, resp.\@ 2-opfibration) of 2-multicategories.
An $i$-morphism $\xi$ is (co)Cartesian w.r.t.\@ $p_2 \circ p_1$ if and only if it is $i$-(co)Cartesian w.r.t.\@ 
$p_1$ and $p_1(\xi)$ is $i$-(co)Cartesian w.r.t.\@ $p_2$.
\end{PROP}

\begin{proof}
Let $\xi \in \Hom_{\mathcal{E}}(\Sigma_1, \dots, \Sigma_n; \Xi_i)$ be a 1-morphism which is coCartesian for $p_1$ and such that $p_1(\xi)$ is coCartesian for $p_2$. 
Then we have the following diagram
\begin{equation}\label{diatrans}
\vcenter{ \xymatrix{
\Hom_{\mathcal{E}}(\Xi_1, \dots, \Xi_m; \Pi) \ar[rr]^-{\circ_i \xi} \ar[d] && \Hom_{\mathcal{E}}(\Xi_1, \dots, \Xi_{i-1}, \Sigma_1, \dots, \Sigma_n, \Xi_{i+1}, \dots, \Xi_m; \Pi) \ar[d] \\
\Hom_{\mathcal{D}}(\mathcal{F}_1, \dots, \mathcal{F}_m; \mathcal{G}) \ar[rr]^-{\circ_i p_1(\xi)} \ar[d] && \Hom_{\mathcal{D}}(\mathcal{F}_1, \dots, \mathcal{F}_{i-1}, \mathcal{E}_1, \dots, \mathcal{E}_n, \mathcal{F}_{i+1}, \dots, \mathcal{F}_m; \mathcal{G}) \ar[d] \\
\Hom_{\mathcal{S}}(T_1, \dots, T_m; U) \ar[rr]^-{ \circ_i p_2(p_1(\xi))} && \Hom_{\mathcal{S}}(T_1, \dots, T_{i-1}, S_1, \dots, S_n, T_{i+1}, \dots, T_m; U)
} }
\end{equation}
in which both small squares commute and are 2-Cartesian. Hence also the composite square is 2-Cartesian, that is, $\xi$ is coCartesian for $p_2 \circ p_1$.

Let $f \in \Hom_{\mathcal{S}}(S_1, \dots, S_n; T)$ be a 1-morphism and $\Sigma_1, \dots, \Sigma_n$ be objects of $\mathcal{E}$ over $S_1, \dots, S_n$. Choose a coCartesian 1-morphism 
$\mu \in \Hom_{\mathcal{D}}(p_1(\Sigma_1), \dots, p_1(\Sigma_n); \mathcal{F}_i)$ in $\mathcal{D}$ over $f$. Choose a coCartesian 1-morphism (for $p_1$)  
$\xi \in \Hom_{\mathcal{E}}(\Sigma_1, \dots, \Sigma_n; \Xi_i)$ over $\mu$. We have seen before that $\xi$ is coCartesian for $p_2 \circ p_1$ as well. 

Let $\xi' \in \Hom_{\mathcal{E}}(\Sigma_1, \dots, \Sigma_n; \Xi_i')$ be a different coCartesian 1-morphism for $p_2 \circ p_1$ over $f$. 
We still have to prove the implication that $\xi'$ is coCartesian for $p_1$ and that $p_1(\xi')$ is coCartesian for $p_2$.

By Lemma~\ref{LEMMAREPREQUIV} there is an equivalence $\alpha: \Xi_i' \rightarrow \Xi_i$ such that $\xi'$ is isomorphic to $\alpha \circ \xi$. Then $p_1(\xi')$ is isomorphic to $p_1(\alpha) \circ \mu$. 
The 1-morphism $\alpha \circ \xi$ is coCartesian for $p_1$, being a composition of
coCartesian 1-morphisms for $p_1$  (cf.\@ Lemma~\ref{LEMMACOMPOSCART} and Lemma~\ref{LEMMAEQUIVALENCECARTCOCART}). Therefore, by Lemma~\ref{LEMMA1COCART2ISO}, also $\xi'$ is coCartesian for $p_1$, and hence
$p_1(\alpha) \circ \mu$ is a composition of coCartesian morphisms for $p_2$. Therefore, by Lemma~\ref{LEMMA1COCART2ISO}, also $p_1(\xi')$ is coCartesian for $p_2$. 
\end{proof}

There is a certain converse to the previous proposition: 
\begin{PROP}\label{PROPBIFIBTRANSITIVITYCONVERSE}
Let $p_1: \mathcal{E} \rightarrow \mathcal{D}$ and $p_2: \mathcal{D} \rightarrow \mathcal{S}$ be 2-isofibrations of 2-multicategories.
Then $p_1: \mathcal{E} \rightarrow \mathcal{D}$ is a 1-fibration (resp.\@ 1-opfibration), if the following conditions hold: 
\begin{enumerate}
\item $p_2 \circ p_1$ is a 1-fibration (resp.\@ 1-opfibration);
\item $p_1$ maps (co)Cartesian 1-morphisms w.r.t.\@ $p_2 \circ p_1$ to (co)Cartesian 1-morphisms w.r.t.\@ $p_2$;
\item $p_1$ induces a 1-fibration (resp.\@ 1-opfibration) between fibers\footnote{Note that these fibers are usual 2-categories, not 2-multicategories.} $\mathcal{E}_S \rightarrow \mathcal{D}_S$ for any $S \in \mathcal{S}$ and (co)Cartesianity of 1-morphisms in the fibers of $p_1$ is stable under pull-back (resp.\@ push-forward) w.r.t.\@ $p_2 \circ p_1$.

More precisely (here for the opfibered case, the other case is similar): For a morphism $f \in \Hom(S_1, \dots, S_n; T)$, for objects $\mathcal{E}_i$  over $S_i$, and morphisms
$\tau_i: \mathcal{E}_i \rightarrow \mathcal{F}_i$ over $\id_{S_i}$, consider a diagram in $\mathcal{D}$
\[ \xymatrix{
\mathcal{E}_1, \dots, \mathcal{E}_n \ar[rr]^{(\tau_1, \dots, \tau_n)} \ar[d]_{\xi} \ar@{}[drr]|{\Swarrow^\sim}  && \mathcal{F}_1, \dots, \mathcal{F}_n \ar[d]^{\xi'}  \\
\mathcal{G} \ar[rr] && \mathcal{H}
} \]
over the diagram in $\mathcal{S}$
\[ \xymatrix{
S_1, \dots, S_n \ar@{=}[r] \ar[d]_f  & S_1, \dots, S_n \ar[d]^f  \\
T \ar@{=}[r] & T 
} \]
where $\xi$ and $\xi'$ are coCartesian 1-morphisms (in particular the 1-morphism $\mathcal{G} \rightarrow \mathcal{H}$ is uniquely determined up to 2-isomorphism).
Given a diagram in $\mathcal{E}$ 
\[ \xymatrix{
\Xi_1, \dots, \Xi_n \ar[rr]^{(\mu_1, \dots, \mu_n)} \ar[d]_{\kappa} \ar@{}[drr]|{\Swarrow^\sim}  && \Phi_1, \dots, \Phi_n \ar[d]^{\kappa'}  \\
\Pi \ar[rr]^\nu && \Sigma
} \]
over the other two, the following holds true: If $\kappa$ and $\kappa'$ are coCartesian 1-morphisms w.r.t.\@ $p_2 \circ p_1$ and if $\mu_1, \dots, \mu_n$ are coCartesian 1-morphisms w.r.t.\@ $p_1$ (restriction to the respective fiber) then
also $\nu$ is a coCartesian 1-morphism w.r.t.\@ $p_1$ (restriction to the fiber over $T$). 
\end{enumerate}
\end{PROP}

\begin{proof}
We have to show that coCartesian 1-morphisms w.r.t.\@ $p_1$ exist. To ease notation we will neglect the multi-aspect.

Let $\tau: \mathcal{E} \rightarrow \mathcal{F}$ be a 1-morphism over $f: S \rightarrow T$ and let $\Xi$ be an object over $\mathcal{E}$. 
Choose a coCartesian 1-morphism $\xi: \Xi \rightarrow \Xi'$ over $f$ w.r.t.\@ $p_2 \circ p_1$ which exists by property 1. By property 2.\@ we have that
$p_1(\xi): \mathcal{E} \rightarrow \mathcal{E}'$ is a coCartesian 1-morphism over $f$ w.r.t.\@ $p_2$. We therefore have an induced 1-morphism $\widetilde{\tau}: \mathcal{E}' \rightarrow \mathcal{F}$ over $\id_T$ and a 2-isomorphism
\[ \eta: \widetilde{\tau} \circ  p_1(\xi) \Rightarrow \tau. \]
Now choose a coCartesian 1-morphism $\xi': \Xi' \rightarrow \Xi''$ w.r.t.\@ $p_{1,T}: \mathcal{E}_T \rightarrow \mathcal{D}_T$  over $\widetilde{\tau}$. 
We claim that 
\[ \eta_*(\xi' \circ \xi): \Xi \rightarrow \Xi'' \]
is a coCartesian 1-morphism over $\tau$. Using Lemma~\ref{LEMMA1COCART2ISO} this is equivalent to $\xi' \circ \xi$ being a coCartesian 1-morphism over $\widetilde{\tau} \circ p_1(\xi)$. 
Using diagram (\ref{diatrans}) from the proof of the previous proposition we see that $\xi$ is a coCartesian 1-morphism for $p_1$ as well. Since the composition of coCartesian 
1-morphisms is coCartesian we are left to show that $\xi'$  is coCartesian for $p_1$.  
Let $f: T \rightarrow U$ be a morphism in $\mathcal{S}$ and $\Sigma$ an object over $\mathcal{G}$ over $U$. We have to show that
\begin{equation} \label{diaPROPBIFIBTRANSITIVITYCONVERSE}
\vcenter{ \xymatrix{
\Hom_{\mathcal{E},f}(\Xi'', \Sigma) \ar[r]^{\circ \xi' } \ar[d] &  \Hom_{\mathcal{E},f}(\Xi', \Sigma) \ar[d] \\
\Hom_{\mathcal{D},f}(\mathcal{F}, \mathcal{G}) \ar[r]^{\circ \widetilde{\tau} } & \Hom_{\mathcal{D},f}(\mathcal{E}', \mathcal{G})
} } \end{equation}
is 2-Cartesian (or Cartesian, which amounts to the same). We can form a 2-commutative diagram
\[ \xymatrix{
  \Xi' \ar[r]^{\xi'}  \ar[d] \ar@{}[dr]|{\Swarrow^\sim} & \Xi''  \ar[rd]_{\Swarrow^\sim}  \ar[d] \\
 \widetilde{\Xi}' \ar[r]^{\widetilde{\xi}'} & \widetilde{\Xi}'' \ar[r] &  \Sigma
} \]
in which the vertical morphisms are coCartesian 1-morphisms w.r.t.\@ $p_2 \circ p_1$ over $f$. 
The diagram (\ref{diaPROPBIFIBTRANSITIVITYCONVERSE}) is point-wise equivalent to the diagram
\[ \xymatrix{
\Hom_{\mathcal{E},\id_U}(\widetilde{\Xi}'', \Sigma) \ar[r]^{\widetilde{\xi}' } \ar[d] &  \Hom_{\mathcal{E},\id_U}(\widetilde{\Xi}', \Sigma) \ar[d] \\
\Hom_{\mathcal{D},\id_U}(p_1(\widetilde{\Xi}''), \mathcal{G}) \ar[r]^{\circ p_1(\widetilde{\xi}') } & \Hom_{\mathcal{D},\id_U}(p(\widetilde{\Xi}'), \mathcal{G})
} \]
which is 2-Cartesian because $\widetilde{\xi}'$ is coCartesian w.r.t.\@ $p_{1,U}$ by property 3.
\end{proof}

\section{Correspondences in a category and abstract six-functor formalisms}\label{SECTIONCOR}

Let $\mathcal{S}$ be a usual 1-category with fiber products and final object and assume that strictly associative fiber products have been chosen in $\mathcal{S}$.

\begin{DEF}\label{DEFSCOR}
We define the {\bf 2-multicategory $\mathcal{S}^{\cor}$ of correspondences in $\mathcal{S}$}  to be the following 2-multicategory. 
\begin{enumerate}
\item The objects are the objects of $\mathcal{S}$.
\item The 1-morphisms $\Hom(S_1, \dots, S_n; T)$ are the (multi-)correspondences\footnote{as usual, $n=0$ is allowed.}
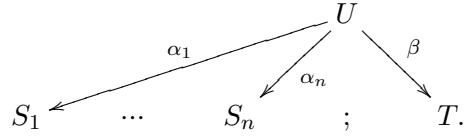
\[ \xymatrix{ &&&U \ar[dlll]_{\alpha_1} \ar[dl]^{\alpha_n}  \ar[dr]^{\beta} \\ 
S_1 & \cdots & S_n &;& T.
} \]

\item The 2-morphisms $(U, \alpha_1, \dots, \alpha_n, \beta) \Rightarrow (U', \alpha_1', \dots, \alpha_n', \beta')$ are the {\em iso}morphisms $\gamma: U \rightarrow U'$ such that in
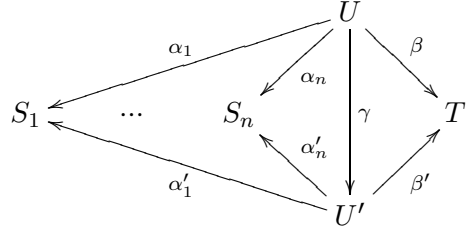
\[ \xymatrix{ 
&&&U \ar[dlll]_{\alpha_1} \ar[dl]^{\alpha_n} \ar[dd]^\gamma \ar[dr]^{\beta} \\ 
S_1 & \cdots & S_n && T \\
&&&U' \ar[ulll]^{\alpha_1'} \ar[ul]_{\alpha_n'}  \ar[ur]_{\beta'}
} \]
all triangles are commutative. 
\item The composition is given by the fiber product in the following way: the correspondence 
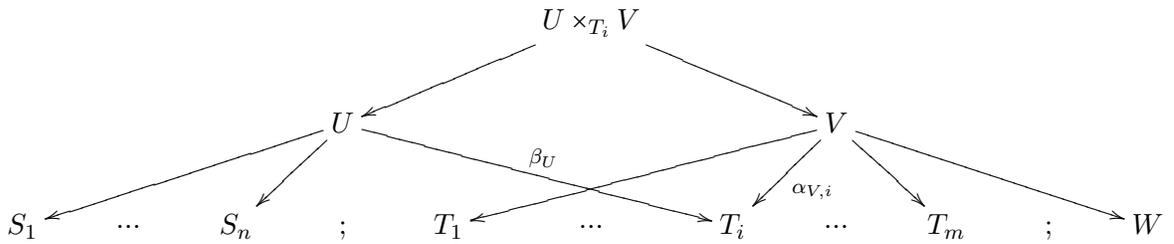
\[ \xymatrix{
&&&&& U \times_{T_i} V  \ar[lld] \ar[rrd] \\
&&& U \ar[llld] \ar[ld] \ar[rrrd]^{\beta_U} &&&&  V \ar[llld] \ar[ld]^{\alpha_{V,i}}  \ar[rd] \ar[rrrd] \\
S_1 & \cdots & S_n & ; & T_1 & \cdots & T_i & \cdots & T_m & ; & W
} \]
in $\Hom(T_1, \dots, T_{i-1}, S_1, \dots, S_n, T_{i+1},  \dots, T_m; W)$
is the composition w.r.t.\@ the $i$-th slot of the left correspondence in  $\Hom(S_1, \dots, S_n; T_i)$ and
the right correspondence in $\Hom(T_1, \dots, T_m; W)$. 
\end{enumerate}
\end{DEF}

The 2-multicategory $\mathcal{S}^{\cor}$ is symmetric, representable
(i.e.\@ opfibered over $\{\cdot\}$), closed (i.e.\@ fibered over $\{\cdot\}$) and self-dual, with tensor product and internal hom {\em both} given by the product $\times$ in $\mathcal{S}$ and having as unit
the final object of $\mathcal{S}$. 

\begin{DEF}
We define also the larger category $\mathcal{S}^{\cor,G}$ where in addition every morphism $\gamma: U \rightarrow U'$ such that the above diagrams commute is a 2-morphism (i.e.\@ $\gamma$ does not necessarily have to be an isomorphism). 
\end{DEF}

\begin{PAR}
The previous definition can be generalized to the case of a general opmulticategory (\ref{OPMULTI}) $\mathcal{S}$ which has multipullbacks: Given a multimorphism $T \rightarrow S_1, \dots, S_n$  and a morphism $S_i' \rightarrow S_i$ for some $1 \le i \le n$,  
a {\bf multipullback}  is a universal square of the form
\[ \xymatrix{
T' \ar[r] \ar[d] & S_1, \dots, S_i', \dots, S_n \ar[d] \\
T \ar[r] & S_1, \dots, S_n.
} \]

A usual category  $\mathcal{S}$ becomes an opmulticategory setting 
\begin{equation}\label{opmulti} \Hom(T; S_1, \dots, S_n) := \Hom(T, S_1) \times \cdots \times \Hom(T, S_n). \end{equation}
In case that a usual category $\mathcal{S}$ has pullbacks it automatically has multipullbacks w.r.t.\@ opmulticategory structure given by (\ref{opmulti}). Those are given by
 Cartesian squares
\[ \xymatrix{
T' \ar[r] \ar[d] &  S_i'  \ar[d] \\
T \ar[r] & S_i.
} \]
For any opmulticategory $\mathcal{S}$ with multipullbacks we define $\mathcal{S}^{\cor}$ to be the 2-category whose objects are the objects of $\mathcal{S}$, whose 1-morphisms
are the multicorrespondences of the form
\[ \xymatrix{
& U \ar[rd] \ar[ld]  \\
S_1, \dots, S_n &  & T \\
} \]
and whose 2-morphisms are commutative diagrams of multimorphisms
\[ \xymatrix{
& U \ar[rd] \ar[ld] \ar[dd] \\
S_1, \dots, S_n &  & T. \\
& U' \ar[ru] \ar[lu]  
} \]
The composition is given by forming the multipullback. 
The reader may check that if the opmulticategory structure on $\mathcal{S}$ is given by (\ref{opmulti}) we reobtain the 2-multicategory $\mathcal{S}^{\cor}$ defined in \ref{DEFSCOR}.
\end{PAR}

\begin{PAR}
We now extend \cite[Definition A.2.16]{Hor15} (cf.\@ Section~\ref{GROTHWIRTH} for an explanation of the terminology).
\end{PAR}
\begin{DEF}\label{DEF6FU}
Let $\mathcal{S}$ be a opmulticategory with multipushouts. A {\bf (symmetric) Grothendieck six-functor-formalism} on $\mathcal{S}$ is a 1-bifibered and 2-bifibered (symmetric) 2-multicategory with 1-categorical fibers
\[ p: \mathcal{D} \rightarrow \mathcal{S}^{\mathrm{cor}}. \]
A {\bf (symmetric) Grothendieck context} on $\mathcal{S}$ is a 1-bifibered and 2-opfibered (symmetric) 2-multicategory with 1-categorical fibers
\[ p: \mathcal{D} \rightarrow \mathcal{S}^{\mathrm{cor},G}. \]
A {\bf (symmetric) Wirthm\"uller context} on $\mathcal{S}$ is a 1-bifibered and 2-fibered (symmetric) 2-multicategory with 1-categorical fibers
\[ p: \mathcal{D} \rightarrow \mathcal{S}^{\mathrm{cor},G}. \]
\end{DEF}

\begin{PAR}\label{PARS0}
If we are given a class of ``proper'' (resp.\@ ``etale'') 1-ary morphisms $\mathcal{S}_0$ in $\mathcal{S}$, it is convenient to define $\mathcal{S}^{\cor, 0}$ to be the category where the morphisms $\gamma: U \rightarrow U'$ entering the definition of 2-morphism are the morphisms in $\mathcal{S}_0$. Then we would consider a 1-bifibration
\[ p: \mathcal{D} \rightarrow \mathcal{S}^{\mathrm{cor}, 0} \]
which is a 2-opfibration in the proper case and a 2-fibration in the etale case. We call this respectively a {\bf (symmetric) proper six-functor-formalism} and a {\bf (symmetric) etale six-functor-formalism}. 
\end{PAR}

\begin{PAR}\label{CONCRETE6FU}
We have a morphism of opfibered (over $\{\cdot\}$) symmetric multicategories $\mathcal{S}^{\mathrm{op}} \rightarrow \mathcal{S}^{\mathrm{cor}}$.
However, if $\mathcal{S}$ has the opmulticategory structure (\ref{opmulti}),  i.e.\@ if $\mathcal{S}^{\cor}$ is as defined in \ref{DEFSCOR},  there is no reasonable morphism of opfibered multicategories $\mathcal{S} \rightarrow \mathcal{S}^{\mathrm{cor}}$ where $\mathcal{S}$ is equipped with the symmetric multicategory structure as in \ref{EXMULTICATS}\footnote{There is though a morphism of multicategories $\mathcal{S} \rightarrow \mathcal{S}^{\mathrm{cor}}$, where $\mathcal{S}$ is equipped with the multicategory structure $\Hom_{\mathcal{S}}(S_1, \dots, S_n; T):=\Hom(S_1 \times \cdots \times S_n; T)$.}. This reflects the fact that, in the classical formulation of the six functors, there is no compatibility involving only `$\otimes$' and `$!$'. From a Grothendieck six-functor-formalism over $\mathcal{S}$ equipped with the opmulticategory structure (\ref{opmulti}) we get operations $g_*$, $g^*$ as the pull-back and the push-forward along the correspondence
\[ \xymatrix{ 
 & \ar[ld]_g S \ar@{=}[rd] &\\
 T & ; & S.   }\]
 We get $f^!$ and $f_!$ as the pull-back and the push-forward along the correspondence
\[ \xymatrix{ 
 & \ar@{=}[ld] S \ar[rd]^f &\\
 S & ; & T.   }\]
 
We get the monoidal product $\mathcal{E} \otimes \mathcal{F}$ for objects $\mathcal{E}, \mathcal{F}$ above $S$ as the target of any Cartesian morphism 
$\otimes$ over the correspondence
\[ \xi_S = 
 \left(\vcenter{\xymatrix{
& & S \ar@{=}[lld] \ar@{=}[ld] \ar@{=}[rd] &  \\
S & S & ; & S }}\right).
\]

Alternatively, we have
\[ \mathcal{E} \otimes \mathcal{F} = \Delta^* (\mathcal{E} \boxtimes \mathcal{F}) \]
where $\Delta^*$ is the push-forward along the correspondence
\[ \left(\vcenter{\xymatrix{
& S \ar[ld]_\Delta \ar@{=}[rd] & \\
S \times S & ; & S }}\right)
 \]
induced by the canonical 1-morphism $\xi_S \in \Hom(S,S; S)$, and where $\boxtimes$ is the absolute monoidal product which exists because by Proposition~\ref{PROPBIFIBTRANSITIVITY} the composition $\mathcal{D} \rightarrow \{\cdot\}$ is opfibered as well, i.e.\@ $\mathcal{D}$ is monoidal.
\end{PAR}

\begin{PAR}
It is easy to derive from the definition of bifibered multicategory over $\mathcal{S}^{\mathrm{cor}}$ that
the absolute monoidal product $\mathcal{E} \boxtimes \mathcal{F}$ can be reconstructed from the fiber-wise product as $\pr_1^*\mathcal{E}  \otimes \pr_2^*\mathcal{F}$ on $S \times T$, whereas the
absolute $\mathbf{HOM}(\mathcal{E}, \mathcal{F})$ is given by $\mathcal{HOM}(\pr_1^*\mathcal{E}, \pr_2^! \mathcal{F})$ on $S \times T$.
In particular, for an object $\mathcal{E}$ of $\mathcal{D}$ lying over an object $S$ in $\mathcal{S}$, we can define the absolute duality by $D\mathcal{E}:=\mathbf{HOM}(\mathcal{E},1)$. It is then equal to $\mathcal{HOM}(\mathcal{E}, \pi^! 1)$ for $\pi: S \rightarrow \cdot$ being the final morphism. Here $1$ is the unit object w.r.t.\@ to the monodal structure on $\mathcal{D}_{\cdot}$, i.e. an object representing $\Hom_{\mathcal{D}_{\cdot}}(;-)$. The unit object $1$ seen as an object in $\mathcal{D}$ is also the unit w.r.t.\@ the absolute monoidal structure. 
We will discuss this more thoroughly in Section~\ref{ABSRELMONOIDAL}.
\end{PAR}

\begin{PROP}\label{LEMMA6FU}
Given a Grothendieck six-functor-formalism on $\mathcal{S}$ 
\[ p: \mathcal{D} \rightarrow \mathcal{S}^{\mathrm{cor}} \]
where $\mathcal{S}$ is a usual category equipped with the opmulticategory structure (\ref{opmulti})
for the six functors as extracted in \ref{CONCRETE6FU} there exist naturally the following compatibility isomophisms:
\begin{center}
\begin{tabular}{rlll}
& left adjoints & right adjoints \\
\hline
$(*,*)$ & $(fg)^* \iso g^* f^*$ & $(fg)_* \iso f_* g_*$ &\\
$(!,!)$ & $(fg)_! \iso f_! g_!$ & $(fg)^! \iso g^! f^!$ &\\ 
$(!,*)$ 
& $g^* f_! \iso F_! G^*$ & $G_* F^! \iso f^! g_*$ & \\
$(\otimes,*)$ & $f^*(- \otimes -) \iso f^*- \otimes f^* -$ & $f_* \mathcal{HOM}(f^*-, -) \iso \mathcal{HOM}(-, f_*-)$  & \\
$(\otimes,!)$ & $f_!(- \otimes f^* -) \iso  (f_! -) \otimes -$ & $f_* \mathcal{HOM}(-, f^!-) \iso \mathcal{HOM}(f_! -, -)$ & \\ 
& & $f^!\mathcal{HOM}(-, -) \iso \mathcal{HOM}(f^* -, f^!-)$ & \\
$(\otimes, \otimes)$ &  $(- \otimes -) \otimes - \iso - \otimes (- \otimes -)$ &  $\mathcal{HOM}(- \otimes -, -) \iso \mathcal{HOM}(-, \mathcal{HOM}(-, -))$ & 
\end{tabular}
\end{center}

Here $f, g, F, G$ are morphisms in $\mathcal{S}$ which, in the $(!,*)$-row, are related by a  {\em Cartesian} diagram
\[ \xymatrix{ \cdot \ar[r]^G \ar[d]_F  & \cdot \ar[d]^f \\ \cdot \ar[r]_g & \cdot } \]
\end{PROP}
\begin{proof} See \cite[Lemma~A.2.19]{Hor15}.
\end{proof}

\begin{BEM}
This raises the question about to what extent a converse of Proposition~\ref{LEMMA6FU} holds true. In the literature a six-functor-formalism is often introduced
merely as a collection of functors such that the isomorphisms of Proposition~\ref{LEMMA6FU} exist, without specifying explicitly their compatibilities. In view
of the theory developed in this section the question becomes: how can the 1- and 2-morphisms in the 2-multicategory $\mathcal{S}^{\cor}$ be presented by generators and relations? We will not try to answer this question because all compatibilities, if needed, can be easily derived from the definition of $\mathcal{S}^{\cor}$. As an illustration, we prove that the diagram of isomorphisms
\begin{equation} \label{diatocommute}
\vcenter{
\xymatrix{
G_! F^* (A \otimes g^*B)  \ar[d]_{(\otimes, *)}  && f^* g_! (A \otimes g^*B) \ar[ll]_{(!,*)}  \ar[d]^{(\otimes, !)} \\
G_! ((F^* A) \otimes F^*g^*B)   && f^* (g_!A \otimes B)  \ar[dd]^{(\otimes, *)} \\
G_! ((F^* A) \otimes (gF)^*) \ar[d]_{(*,*)} \ar[u]^{(*,*)}  \\
G_! ((F^* A) \otimes G^*f^*B)  \ar[dr]_{(\otimes, !)}  &&  (f^* g_!A)  \otimes f^*B  \ar[dl]^{(*, !)} \\  
& (G_! F^* A) \otimes f^*B  &
} }
\end{equation}
commutes.
For this we only have to check that the two chains of obvious 2-isomorphisms in $\mathcal{S}^{\cor}$ given in Figure~\ref{fig1} and Figure~\ref{fig2} are equal. 

\begin{figure}[!ht]
{ \tiny
\[
 \begin{array}{c}
  \left( \vcenter{ 
\xymatrix{ &  X  \ar[ld]_{F} \ar[rd]^{G} \\
 Z & ; & Y } } \right) \left( \vcenter{ 
\xymatrix{ & &  Z \ar[lld]_{g} \ar@{=}[ld] \ar@{=}[rd] \\
 W & Z & ; & Z  } } \right)
 
 \\  
\downarrow \\

 \left( \vcenter{ 
\xymatrix{ & &  X \ar[lld]_{F} \ar[ld]^F \ar[rd]^G \\
 Z & Z & ; & Y  } } \right)
 \left( \vcenter{ 
\xymatrix{ &  Z  \ar[ld]^{g} \ar@{=}[rd] \\
 W & ; & Z } } \right)

\\  
\uparrow  \\
  
 \left( \vcenter{ 
\xymatrix{ & &  X \ar[lld]_{gF} \ar[ld]^F \ar[rd]^G \\
 W & Z & ; & Y  } } \right)

\\ 
\downarrow \\

 \left( \vcenter{ 
\xymatrix{ & &  X \ar[lld]_{G} \ar[ld]^F \ar[rd]^G \\
 Y & Z & ; & Y  } } \right)
 \left( \vcenter{ 
\xymatrix{ &  Y  \ar[ld]^{f} \ar@{=}[rd] \\
 W & ; & Y } } \right)
\\ 
\downarrow \\

 \left( \vcenter{ 
\xymatrix{ & &  Y \ar@{=}[lld] \ar@{=}[ld] \ar@{=}[rd] \\
 Y & Y & ; & Y  } } \right)   \left( \vcenter{ 
\xymatrix{ &  X  \ar[ld]_{F} \ar[rd]^{G} \\
 Z & ; & Y } } \right) 
 \left( \vcenter{ 
\xymatrix{ &  Y  \ar[ld]^{f} \ar@{=}[rd] \\
 W & ; & Y } } \right) 
\end{array} \]
}
\caption{The first composition}\label{fig1}
\end{figure}
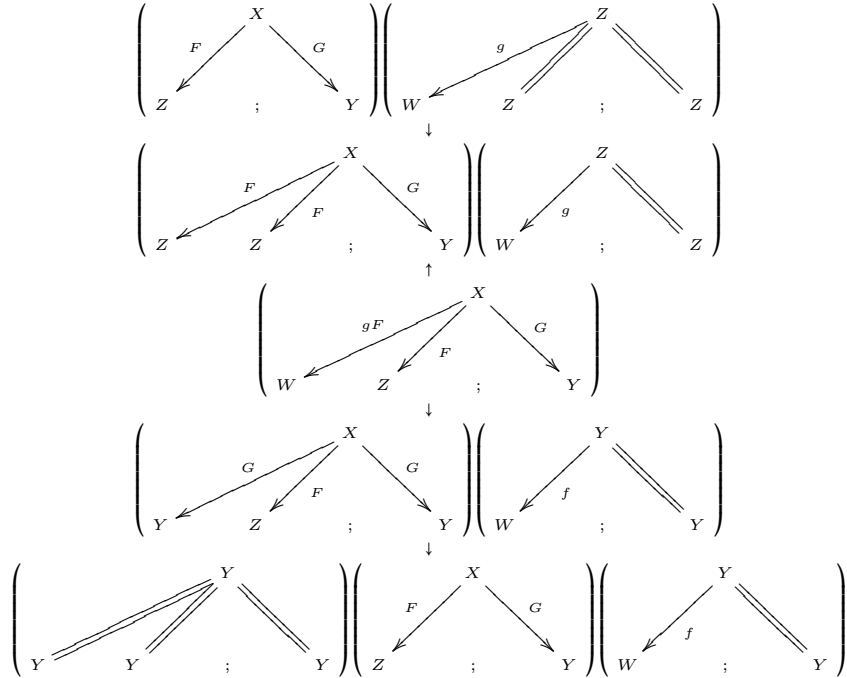

\begin{figure}[!ht]
{ \tiny
\[
 \begin{array}{c}
  \left( \vcenter{ 
\xymatrix{ &  X  \ar[ld]_{F} \ar[rd]^{G} \\
 Z & ; & Y } } \right) \left( \vcenter{ 
\xymatrix{ & &  Z \ar[lld]_{g} \ar@{=}[ld] \ar@{=}[rd] \\
 W & Z & ; & Z  } } \right)
 
 \\  
\uparrow \\

  \left( \vcenter{ 
\xymatrix{ &  Y  \ar[ld]_{f} \ar@{=}[rd] \\
 W & ; & Y } } \right)
  \left( \vcenter{ 
\xymatrix{ &  Z  \ar@{=}[ld] \ar[rd]^g \\
 Z & ; & W } } \right)
  \left( \vcenter{ 
\xymatrix{ & &  Z \ar[lld]_{g} \ar@{=}[ld] \ar@{=}[rd] \\
 W & Z & ; & Z  } } \right)

 \\  
\downarrow \\

  \left( \vcenter{ 
\xymatrix{ &  Y  \ar[ld]_{f} \ar@{=}[rd] \\
 W & ; & Y } } \right)
  \left( \vcenter{ 
\xymatrix{ & &  W \ar@{=}[lld] \ar@{=}[ld] \ar@{=}[rd] \\
 W & W & ; & W  } } \right)
  \left( \vcenter{ 
\xymatrix{ &  Z  \ar@{=}[ld] \ar[rd]^g \\
 Z & ; & W } } \right)

 \\  
\downarrow \\

  \left( \vcenter{ 
\xymatrix{ & &  Y \ar[lld]_f \ar[ld]^f \ar@{=}[rd] \\
 W & W & ; & Y  } } \right)
  \left( \vcenter{ 
\xymatrix{ &  Z  \ar@{=}[ld] \ar[rd]^g \\
 Z & ; & W } } \right)

 \\  
\downarrow \\

 \left( \vcenter{ 
\xymatrix{ & &  Y \ar@{=}[lld] \ar@{=}[ld] \ar@{=}[rd] \\
 Y & Y & ; & Y  } } \right)   \left( \vcenter{ 
\xymatrix{ &  X  \ar[ld]_{F} \ar[rd]^{G} \\
 Z & ; & Y } } \right) 
 \left( \vcenter{ 
\xymatrix{ &  Y  \ar[ld]^{f} \ar@{=}[rd] \\
 W & ; & Y } } \right) 
\end{array} \]
}
\caption{The second composition}\label{fig2}
\end{figure}
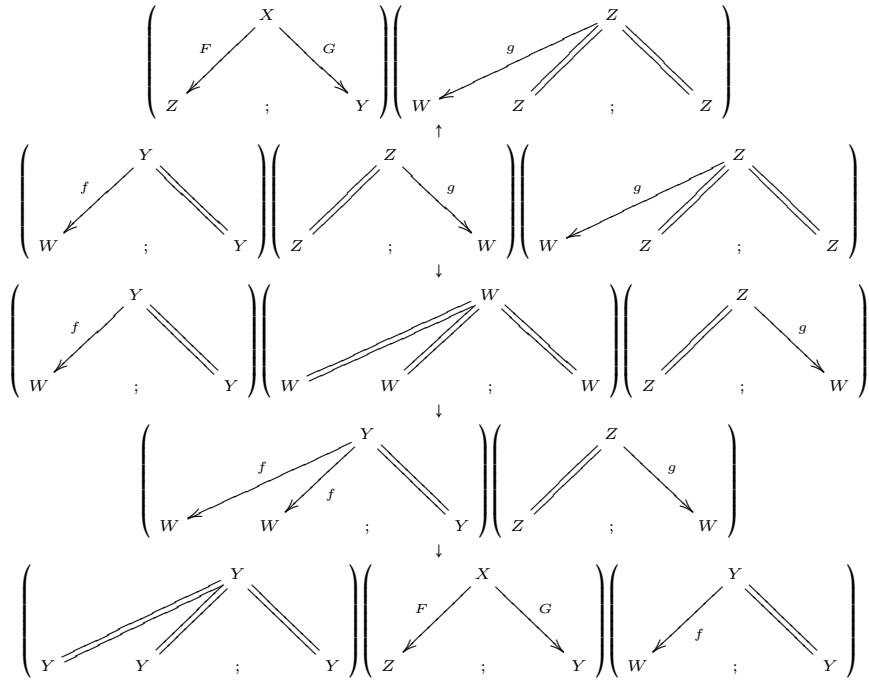
To see this, observe that the multicorrespondences in the lines are all 2-isomorphic to the multicorrespondence
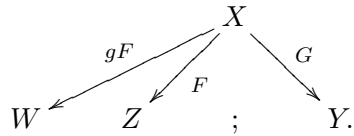
\[
\xymatrix{ & &  X \ar[lld]_{gF} \ar[ld]^F \ar[rd]^G \\
 W & Z & ; & Y.  } 
\]
and that all the 2-isomorphisms in the chains (which induce the isomorphisms in Lemma~\ref{LEMMA6FU} used in the diagram (\ref{diatocommute})) respect these 2-isomorphisms. 

See \ref{EXPROPERETALE} for a similar calculation involving also an (iso-)morphism $f_! \rightarrow f_*$, i.e.\@ involving a proper six-functor-formalism. 
\end{BEM}

\begin{PAR}
Canonical Grothendieck contexts:
Let $\mathcal{S}$ be a 1-opmulticategory with multipullbacks and let 
$p: \mathcal{D} \rightarrow \mathcal{S}^{\op}$ be an ordinary bifibration of 1-multicategories. Let $\mathcal{S}_0$ be a subcategory of ``proper'' morphisms 
for which projection formula and base change formula hold true. This means that for every multipullback with $f_i \in \mathcal{S}_0$
\[ \xymatrix{
T' \ar[r]^-{G} \ar[d]_{F_i} & S_1, \dots, S_i', \dots, S_n \ar[d]^{(\id_{S_1}, \dots, f_i, \dots, \id_{S_n})} \\
T \ar[r]_-{g} & S_1, \dots, S_i, \dots, S_n
} \] 
the canonical exchange natural transformation
\begin{equation}\label{eqbasechange} 
g_\bullet  \circ_i f_i^{\bullet} \rightarrow F_i^\bullet \circ  G_\bullet    
\end{equation}
is an isomorphism. Note that the morphisms are morphisms in $\mathcal{S}$ (and not in $\mathcal{S}^{\op}$), e.g.\@ $F_i^\bullet: \mathcal{D}_{T'} \rightarrow \mathcal{D}_{T}$ denotes a right-adjoint {\em push-forward} along the corresponding morphism in $\mathcal{S}$. 

Assume that $\mathcal{S}_0$ is stable under multipullback, i.e.\@ for any multipullback diagram as above, $F_i$ is in $\mathcal{S}_0$ as well. 
\end{PAR}

\begin{DEF}\label{DEFCANGROTHENDIECK}
Define a category $\widetilde{\mathcal{D}}^{\mathrm{proper}}$ which has the same objects as $\mathcal{D}$ and whose 1-morphisms $\Hom(\mathcal{E}_1, \dots, \mathcal{E}_n; \mathcal{F})$, where $p(\mathcal{E}_i)=S_i$ and $p(\mathcal{F}) = T$, are the 1-morphisms 
\[ \xymatrix{ &U  \ar[dl]_{g}  \ar[dr]^{f} \\ 
(S_1, \dots, S_n) && T
} \]
in $\mathcal{S}^{\cor, 0}$ (cf.\@ \ref{PARS0}) {\em such that $f \in \mathcal{S}_0$}, together with a 1-morphism
\[ \rho \in \Hom_{T}( f^\bullet g_\bullet (\mathcal{E}_1, \dots, \mathcal{E}_n);  \mathcal{F})   \]
in $\mathcal{D}_T$.
A 2-morphism $(U, g, f, \rho) \Rightarrow (U', g', f', \rho')$ is a morphism $h: U \rightarrow U'$ in $\mathcal{S}_0$ making the obvious diagrams commute and such that the diagram
\[ \xymatrix{
(f')^\bullet g_\bullet (\mathcal{E}_1, \dots, \mathcal{E}_n) \ar[r]^-{\rho'} \ar[d]_{\mathrm{unit}_h} & \mathcal{F} \\
(f')^\bullet h^\bullet h_\bullet g'_\bullet (\mathcal{E}_1, \dots, \mathcal{E}_n) \ar[r]^-\sim & \ar[u]_{\rho}  f^\bullet g_\bullet (\mathcal{E}_1, \dots, \mathcal{E}_n)
} \]
also commutes. 
\end{DEF}

\begin{PROP}\label{PROPCANGROTHENDIECK}
Definition~\ref{DEFCANGROTHENDIECK} is reasonable, i.e.\@ the composition induced by projection and base change formula, i.e.\@ by the natural isomorphism (\ref{eqbasechange}), is associative.

The obvious projection
\[ \widetilde{p}: \widetilde{\mathcal{D}}^{\mathrm{proper}} \rightarrow \mathcal{S}^{\cor, \mathrm{proper}, 0} \]
where $\mathcal{S}^{\cor, \mathrm{proper}, 0}$ is the subcategory of $\mathcal{S}^{\cor, 0}$ in whose multicorrespondences the morphism $f$ is in $\mathcal{S}^0$, 
is a 1-opfibration and 2-opfibration of 2-multicategories with 1-categorical fibers.  
\end{PROP}
\begin{proof}This is a straight-forward check that we leave to the reader. For the second assertion note that the category $\widetilde{\mathcal{D}}^{\mathrm{proper}}$ is obviously 2-opfibered over $\mathcal{S}^{\cor, \mathrm{proper}, 0}$, the 2-push-forward given by $\rho \mapsto \mathrm{unit}_h \circ \rho$.
\end{proof}


In particular, if (\ref{eqbasechange}) holds true for {\em all} multipullbacks in $\mathcal{S}$, and all $f^\bullet$ have right adjoints, we obtain the {\bf canonical Grothendieck context} associated with $p: \mathcal{D} \rightarrow \mathcal{S}$:
\[ \widetilde{p}: \widetilde{\mathcal{D}} \rightarrow \mathcal{S}^{\cor,G}. \]

If (\ref{eqbasechange}) holds true only for a proper subclass of morphisms, it is possible under additional hypothesis to extend the so constructed partial six-functor-formalism to a 1-opfibration (which is still 2-opfibered with 1-categorical fibers) over the whole $\mathcal{S}^{\cor, 0}$:
\[ \xymatrix{
\widetilde{\mathcal{D}}^{\mathrm{proper}} \ar@{^{(}->}[r] \ar[d] &  \widetilde{\mathcal{D}}  \ar[d] \\
\mathcal{S}^{\cor, \mathrm{proper}, 0}  \ar@{^{(}->}[r] & \mathcal{S}^{\cor, 0} 
} \]
That is, if right adjoints exist, even to a Grothendieck six-functor-formalism.
The right hand side 1-opfibration and 2-opfibration encodes also morphisms $f_! \rightarrow f_*$ for the corresponding operations and all their compatibilities (cf.\@ \ref{PROPPROPERETALE}). 
Its construction will be explained more generally in the derivator context in forthcoming articles and parallels the classical construction using compactifications.

\section{Correspondences of diagrams}\label{SECTIONDIACOR}

Let $\Dia$ be a diagram category \cite[\S 1.1]{Hor15}, i.e.\@ a full 2-subcategory of the 2-category of small categories satisfying some basic closure properties.
Assume that strictly associative fiber products have been chosen in $\Dia$. Assume also for the rest of this article that $\Dia$ permits arbitrary Grothendieck constructions, i.e.\@ if $I$ is in $\Dia$ and $F: I \rightarrow \Dia$ is a pseudo-functor, then $\int F$ is in $\Dia$.

In this section we will define a category $\Dia^{\cor}$  of correspondences in $\Dia$ similarly to the category of correspondences in a usual category considered in the last section. A Wirthm\"uller context over $\Dia^{\cor}$ in a similar way as defined in the last section will be essentially equivalent to a closed monoidal derivator with domain $\Dia$ (without the axioms (Der1) and (Der2)). Also the more general notion of fibered multiderivator developed in \cite{Hor15} can be easily encoded as a certain (op)fibration of 2-multicategories. 
Since $\Dia$ is a 2-category, the definition of $\Dia^{\cor}$ is a bit more involved. 

\begin{DEF}\label{DEFSPAN}
Let $I_1, \dots, I_n, J$ be diagrams in $\Dia$. Define $\Span(I_1, \dots, I_n; J)$ to be the following strict 2-category:
\begin{enumerate}
\item
The objects are diagrams of the form
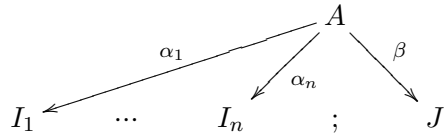
\[ \xymatrix{ &&&A \ar[dlll]_{\alpha_1} \ar[dl]^{\alpha_n}  \ar[dr]^{\beta} \\ 
I_1 & \cdots & I_n &;& J
} \]
with $A \in \Dia$.
\item The 1-morphisms $(A, \alpha_1, \dots, \alpha_n, \beta) \Rightarrow (A', \alpha_1', \dots, \alpha_n', \beta')$ are functors $\gamma: A \rightarrow A'$ and natural transformations $\nu_1, \dots, \nu_n, \mu$:
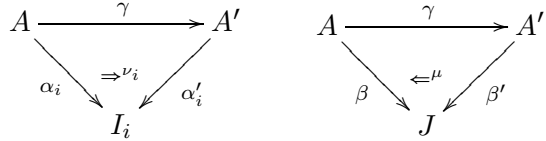
\[ \vcenter{ \xymatrix{
A \ar[rd]_{\alpha_i} \ar[rr]^{\gamma} & \ar@{}[d]|{\Rightarrow^{\nu_i}} & A' \ar[ld]^{\alpha_i'}  \\
&I_i
} } \qquad \vcenter{ \xymatrix{
A \ar[rd]_{\beta} \ar[rr]^{\gamma} & \ar@{}[d]|{\Leftarrow^\mu}  & A' \ar[ld]^{\beta'}  \\
&J
} } \]
\item The 2-morphisms are natural transformations
$\eta: \gamma \Rightarrow \gamma'$ such that $(\alpha_i' \ast \eta) \circ \nu_i  = \nu_i'$ and $(\beta' \ast \eta) \circ \mu' = \mu$ hold.
\end{enumerate}
\end{DEF}
We define also the full subcategory $\Span^{F}(I_1, \dots, I_n; J)$ of those objects for which $\alpha_1 \times \cdots \times \alpha_n: A \rightarrow I_1 \times \cdots \times I_n$ is a fibration and $\beta$ is an opfibration.
The $\gamma$'s do not need to be morphisms of fibrations, respectively of opfibrations.

\begin{PAR}\label{1TRUNC}
For a 2-category $\mathcal{C}$, denote by $\tau_1(\mathcal{C})$ the 1-category in which the morphism sets or classes are the $\pi_0$ (sets or classes of connected components) of the respective categories of 1-morphisms in $\mathcal{C}$. 
\end{PAR}

\begin{DEF}\label{DEFDIACOR}
We define the {\bf 2-multicategory of correspondences of diagrams} $\Dia^{\cor}$ as the following 2-multicategory:
\begin{enumerate}
\item The objects are diagrams $I \in \Dia$.
\item For every $I_1, \dots, I_n, J$ diagrams in $\Dia$, the category $\Hom_{\Dia^{\cor}}(I_1, \dots, I_n; J)$ of 1-morphisms of $\Dia^{\cor}$ is the truncated category $\tau_1(\Span^F(I_1, \dots, I_n; J))$. 
\end{enumerate}
\end{DEF}


Composition is defined by taking fiber products. The diagram (forgetting the functor to $J_i$)
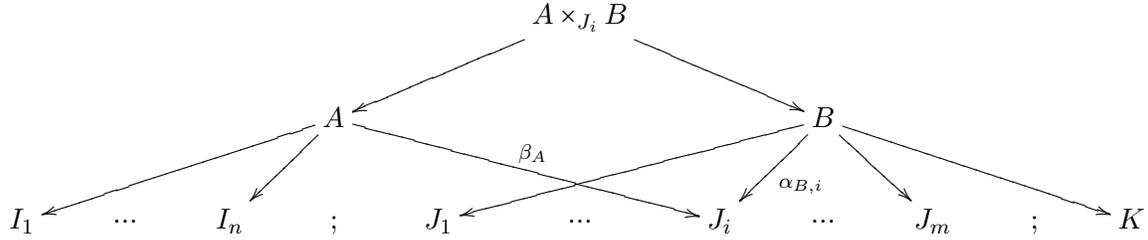
\[ \xymatrix{
&&&&& A \times_{J_i} B  \ar[lld] \ar[rrd] \\
&&& A \ar[llld] \ar[ld] \ar[rrrd]^{\beta_A} &&&&  B \ar[llld] \ar[ld]^{\alpha_{B,i}}  \ar[rd] \ar[rrrd] \\
I_1 & \cdots & I_n & ; & J_1 & \cdots & J_i & \cdots & J_m & ; & K
} \]
is defined to be the composition of the left hand side correspondence in $\Hom(I_1, \dots, I_n; J_i)$ with the right hand side correspondence in $\Hom(J_1, \dots, J_m; K)$. 
One checks that $A \times_{J_i} B \rightarrow J_1 \times \cdots \times J_{i-1} \times I_1 \times \cdots \times I_n \times J_{i+1} \times \cdots \times J_m$ is again a opfibration and that $A \times_{J_i} B \rightarrow K$ is again a fibration. 

It remains to be seen that the composition is functorial in 2-morphisms and that the relations in $\pi_0$ are respected. 
This is obviously the case if we have a functor $\gamma_A: A \rightarrow A'$ such that $\beta_A' \gamma_A = \beta_A'$ holds and a 
functor $\gamma_B: B \rightarrow B'$ such that $\alpha_{B,1} \gamma_B = \alpha_{B,1}'$ holds. Hence the statement follows from the following lemma.

\begin{LEMMA}
Any 2-morphism in $\Dia^{\cor}$
\[ (\gamma, \nu_1, \dots, \nu_n, \mu): (A, \alpha_1, \dots, \alpha_n, \beta) \rightarrow  (A', \alpha_1', \dots, \alpha_n', \beta') \] 
is equivalent (in the sense of $\pi_0$) to a 2-morphism in which
all $\nu_k$ are identities {\em or} to a 2-morphism where $\mu$ is an identity. 
\end{LEMMA}
\begin{proof}
For each $a \in A$, the natural transformation $\mu$ induces a morphism $\beta' \gamma(a) \rightarrow \beta(a)$. Since $\beta'$ is an opfibration, we have a coCartesian morphism $\xi_a: \gamma(a) \rightarrow a'$ in $A'$ mapping via $\beta'$ to 
$\beta' \gamma(a) \rightarrow \beta(a)$. By the universal property of coCartesian arrows, we see that the association $\gamma': a \rightarrow a'$ is actually functorial and that the coCartesian morphisms assemble to 
a natural transformation $\xi: \gamma \rightarrow \gamma'$.
By construction we have $\beta' \gamma' = \beta$ and hence $(\gamma', (\alpha_1' \ast \xi) \nu_1, \dots, (\alpha_n' \ast \xi_n) \nu_n,  \id)$ is an equivalent morphism in which $\mu$ is an identity. 
Similarly the other case is shown. 
\end{proof}

To prove that relations in $\pi_0$ are respected, we can now concentrate on relations between morphisms, where all $\nu_i$ are identities, resp.\@ $\mu$ is an identity. These 
obviously define  a relation for morphisms between the fiber products. 

We could also have used  $\tau_1(\Span(I_1, \dots, I_n; J))$ (without the restriction ${}^F$) in the definition of $\Dia^{\cor}$ and defined composition involving the comma category. This leads only to a bicategory which, however, is equivalent to the present strict one (cf.\@ Corollary~\ref{KOREQSPAN} and the discussion thereafter).

\begin{PAR}
Recall the procedure from \cite[\S 1.3.1]{Cis03} to associate with a pseudo-functor $F: I^{\op} \times J \rightarrow \Dia$, a category 
\[ \xymatrix{
& \int \nabla F = \nabla \int F \ar[dl]_{\alpha} \ar[dr]^\beta \\
I & & J
} \]
such that $\alpha$ is a fibration and $\beta$ is an opfibration. This is done by applying the Grothendieck construction, and its dual, respectively, to the two variables separately (cf.\@ \ref{GROTHCONSTR}, \ref{GROTHCONSTRDUAL}). Explicitly, the category $\int \nabla F$ has the objects $(i, j, X \in F(i,j))$ and the morphisms $(i, j, X \in F(i,j)) \rightarrow (i', j', X' \in F(i',j'))$ are
triples consisting of morphisms $a: i \rightarrow i'$ and $b: j \rightarrow j'$ and a morphism $F(\id_i, b)X \rightarrow F(a, \id_j)X'$. 
The pseudo-functors $F: I^{\op} \times J \rightarrow \Dia$ form a 2-category $\Hom(I^{\op} \times J, \Dia)$ consisting of pseudo-functors, pseudo-natural transformations and modifications. 
\end{PAR}

\begin{PROP}
There is a pair of pseudo-functors
\[ \xymatrix{  \Fun(I_1^{\op} \times \cdots \times I_n^{\op} \times J, \Dia) \ar@<3pt>[rr]^-\Xi && \ar@<3pt>[ll]^-\Pi \Span(I_1, \dots, I_n; J)  } \]
such that there are morphisms in the 2-category of endofunctors of $\Span(I_1, \dots, I_n; J)$
\[ \xymatrix{ \Xi \circ  \Pi  \ar@<3pt>[rr] && \ar@<3pt>[ll] \id_{\Span(I_1, \dots, I_n; J)} } \]
which are inverse to each other up to chains of 2-morphisms, 
and such that there are morphisms in the 2-category of endofunctors of $\Fun(I_1^{\op} \times \cdots \times I_n^{\op} \times J, \Dia)$
\[ \xymatrix{ \Pi \circ  \Xi  \ar@<3pt>[rr] && \ar@<3pt>[ll] \id_{\Fun(I_1^{\op} \times \cdots \times I_n^{\op} \times J, \Dia)} } \]
which are inverse to each other up to chains of 2-morphisms. 
\end{PROP}

\begin{proof}
The pseudo-functor $\Xi$ is defined as follows: 
A pseudo-functor $F \in \Fun(I_1^{\op} \times \cdots \times I_n^{\op} \times J, \Dia)$ is sent to the category $\int \nabla F$ defined above, which comes equipped with a fibration to $I_1 \times \dots \times I_n$ and an opfibration to $J$. 
The fact that these are a fibration, and an opfibration, respectively, does not play any role for this proposition, however.  
A natural transformation $\mu: F \rightarrow G$ is sent to the obvious functor $\widetilde{\mu}: \int \nabla F \rightarrow \int \nabla G$.
A modification $\mu \Rightarrow \mu'$ induces a natural transformation $\widetilde{\mu} \Rightarrow \widetilde{\mu}'$ which whiskered with any of the projections to the $I_k$ or to $J$ gives an identity.

$\Pi$ is defined as follows: 
A correspondence $(A, \alpha_1, \dots, \alpha_n, \beta)$ in $\Span(I_1, \dots, I_n; J)$ is sent to the following functor:
\begin{eqnarray*} 
I_1^{\op} \times \cdots \times I_n^{\op} \times J &\rightarrow & \Dia \\
 (i_1, \dots, i_n, j) &\mapsto& \{(i_1, \dots, i_n)\} \times_{/(I_1 \times \dots \times I_n)} A \times_{/J} \{j\} 
\end{eqnarray*} 
A 1-morphism given by $\gamma: A \rightarrow A'$ and $\nu_1, \dots, \nu_n, \mu$, respectively,  induces functors 
\[ \widetilde{\gamma}(i_1, \dots, i_n; j): \{(i_1, \dots, i_n)\} \times_{/(I_1 \times \dots \times I_n)} A \times_{/J} \{j\}  \rightarrow \{(i_1, \dots, i_n)\} \times_{/(I_1 \times \dots \times I_n)} A' \times_{/J} \{j\} \]
which assemble to a pseudo-natural transformation. A 2-morphism $\mu: \gamma \Rightarrow \gamma'$ induces a natural transformation between the corresponding functors
$\widetilde{\gamma}(i_1, \dots, i_n; j) \Rightarrow \widetilde{\gamma}(i_1, \dots, i_n; j)$ which assemble to a modification. 

We now proceed to construct the required 1-morphisms: 
$\Pi \circ \Xi$ maps a functor $F$ to the functor 
\[ F: (i_1, \dots, i_n, j) \mapsto \{(i_1, \dots, i_n)\} \times_{/(I_1 \times \dots \times I_n)} ( \int \nabla F ) \times_{/J} \{j\}. \]

Pointwise the required natural transformation $\id \rightarrow \Pi \circ \Xi$ is given by
sending an object $X$ of $F(i_1, \dots, i_n; j)$ to the object $(i_1, \dots, i_n, j, X)$ of $\int \nabla F$ together with the various identities $\id_{i_1}, \dots, \id_{i_n}, \id_{j}$. 
Pointwise the required natural transformation $\Pi \circ \Xi \rightarrow \id$ is given by
sending an object $(i_1', \dots, i_n', j', X \in F(i_1', \dots, i_n', j'))$ of $ \int \nabla F $ together with $\alpha_k: i_k \rightarrow i_k'$ and $\beta: j' \rightarrow j$ to
$F(\alpha_1, \dots, \alpha_n; \beta)X \in F(i_1, \dots, i_n; j)$. One easily checks that these natural transformations even constitute an adjunction in the 2-category of endofunctors of $\Fun(I_1^{\op}\times \dots \times I_n^{\op} \times J, \Dia)$.

The pseudo-functor $\Xi \circ \Pi$ is given by 
\[ (A, \alpha_1, \dots, \alpha_n, \beta) \mapsto (I_1 \times \cdots \times I_n \times_{/(I_1 \times \cdots \times I_n)} A \times_{/J} J; \pr_{I_1}, \dots, \pr_{I_n}, \pr_J) \]
together with the various projections. 
First we will construct an adjunction of $\Xi \circ \Pi$ with the pseudo-functor
\[ (A, \alpha_1, \dots, \alpha_n, \beta) \mapsto (I_1 \times \cdots \times I_n \times_{/(I_1 \times \cdots \times I_n)} A ; \pr_{I_1}, \dots, \pr_{I_n}, \beta). \]
In one direction we have the functor which complements an object $(a, \dots)$ by the identity $\id_{\beta(a)}$.
In the other direction we have the forgetful functor, forgetting $\beta(a) \rightarrow j$. 
Those two functors form an adjunction in the 2-category of endofunctors of $\Span(I_1, \dots, I_n; J)$.

Similarly we have an adjunction between
\[ (A, \alpha_1, \dots, \alpha_n, \beta) \mapsto (I_1 \times \cdots \times I_n \times_{/(I_1 \times \cdots \times I_n)} A ; \pr_{I_1}, \dots, \pr_{I_n}, \beta) \]
and the identity
\[ (A, \alpha_1, \dots, \alpha_n, \beta) \mapsto (A, \alpha_1, \dots, \alpha_n, \beta). \]
\end{proof}

Observe that the functor $\Xi$ actually has values in the full subcategory $\Span^F(I_1, \dots, I_n; J)$. 

\begin{KOR}\label{KOREQSPAN}
We have equivalences of categories (cf.\@ \ref{1TRUNC}):
\[ \tau_1(\Fun(I_1^{\op} \times \cdots \times I_n^{\op} \times J, \Dia)) \cong \tau_1(\Span(I_1, \dots, I_n; J)) \cong  \tau_1(\Span^F(I_1, \dots, I_n; J)),  \]
\end{KOR}

Hence we could have defined the 2-multicategory $\Dia^{\mathrm{cor}}$ (as a bimulticategory) using any of these three models for the categories of 1-morphisms. 
The composition of 1-morphisms looks as follows in these three models:

\begin{enumerate}
\item Using $\tau_1(\Span^F(I_1, \dots, I_n; J))$ we get the composition as defined before: 
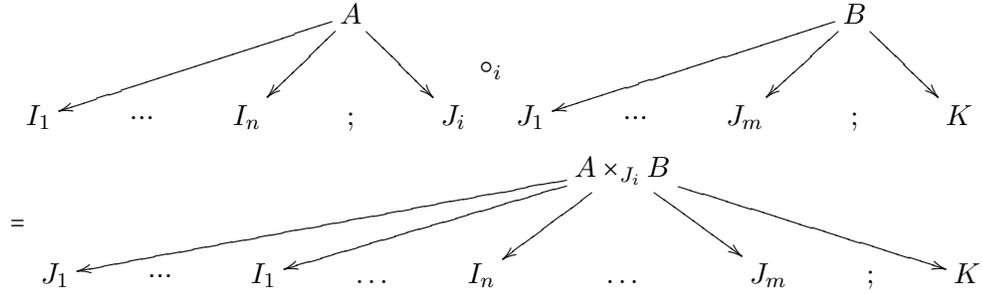
\begin{gather*}
\vcenter{\xymatrix{ &&&A \ar[dlll] \ar[dl]  \ar[dr] \\ 
I_1 & \cdots & I_n &;& J_i
} }
\circ_i
\vcenter{ \xymatrix{ &&&B \ar[dlll] \ar[dl]  \ar[dr] \\ 
J_1 & \cdots & J_m &;& K
}  } \\
=  \vcenter{ \xymatrix{ &&&&&A \times_{J_i} B \ar[dl] \ar[dlll] \ar[dlllll] \ar[dr]  \ar[drrr] \\ 
J_1 & \cdots & I_1 & \dots & I_n & \dots & J_m &;& K
} }
\end{gather*}
\item Using $\tau_1(\Span(I_1, \dots, I_n; J))$ the composition involves the comma category: 
\begin{gather*}
\vcenter{\xymatrix{ &&&A \ar[dlll] \ar[dl]  \ar[dr] \\ 
I_1 & \cdots & I_n &;& J_i
} }
\circ_i
\vcenter{ \xymatrix{ &&&B \ar[dlll] \ar[dl]  \ar[dr] \\ 
J_1 & \cdots & J_m &;& K
}  } \\
= \vcenter{  \xymatrix{ &&&&&A \times_{/J_i} B \ar[dl] \ar[dlll] \ar[dlllll] \ar[dr]  \ar[drrr] \\ 
J_1 & \cdots & I_1 & \dots & I_n & \dots & J_m &;& K
} }
\end{gather*}
\item Using $\tau_1(\Fun(I_1^{\op} \times \cdots \times I_n^{\op} \times J, \Dia))$ we get for pseudo-functors
\[ F : I_1^{\op} \times \cdots \times I_n^{\op} \times J_i \rightarrow \Dia 
\qquad G : J_1^{\op} \times \cdots \times J_m^{\op} \times K \rightarrow \Dia \]
that
\[ G \circ_i F = \mathrm{hocoend}_{J_i} G \times F, \]
where $\mathrm{hocoend}$ is defined in Definition~\ref{DEFHOCOEND} below. 
\end{enumerate}

All these compositions are compatible with the equivalences of Corollary~\ref{KOREQSPAN}.
However, only using the model $\Span^F(I_1, \dots, I_n; J)$ we get strict associativity and the existence of identities.

\begin{DEF}\label{DEFHOCOEND}
 Let $J$ be a diagram and let $F: J^{\op} \times J \rightarrow \Dia$ be a pseudo-functor. 
We define the diagram $\mathrm{hocoend}_J F$ as the category whose objects are the pairs $(j, x)$ with $j \in J$ and $x \in F(j, j)$ and
whose morphisms $(\alpha, \gamma); (j, x) \rightarrow (j', x')$ are the pairs consisting of a morphisms $\alpha: j \rightarrow j'$ and a morphism $\gamma: F(\id_j, \alpha)x \rightarrow F(\alpha, \id_{j'})x'$. The composition of two morphisms $(\alpha, \gamma); (j, x) \rightarrow (j', x')$ and $(\alpha', \gamma'); (j', x') \rightarrow (j'', x'')$ is defined by $(\alpha', \gamma') \circ (\alpha, \gamma) = (\alpha' \alpha, (F(\alpha, \id_{j''}) \gamma' )  \circ (F(\id_j, \alpha') \gamma))$.
\end{DEF}

\begin{PROP}\label{PROPPSEUDOFUNCTDIA}
\begin{enumerate}
\item There is a pseudo-functor of 2-multicategories
\[ \Dia^{2-\op} \rightarrow \Dia^{\cor} \]
where $\Dia^{2-\op}$ is turned into a 2-multicategory by setting 
\[ \Hom_{\Dia^{2-\op}}(I_1, \dots, I_n; J):= \Hom(I_1 \times \cdots \times I_n, J)^{\op}. \]
\item There is a pseudo-functor of 2-multicategories
\[ \Dia^{1-\op} \rightarrow \Dia^{\cor} \]
where $\Dia^{1-\op}$ is turned into a 2-multicategory by setting 
\[ \Hom_{\Dia^{1-\op}}(I_1, \dots, I_n; J):= \Hom(J, I_1) \times \cdots \times \Hom(J, I_n). \]
In particular for any $I \in \Dia$ there is a natural pseudofunctor of 2-multicategories 
\[ \{\cdot\}  \rightarrow \Dia^{\cor} \]
with value $I$. 
\end{enumerate}
\end{PROP}
\begin{proof}
The functor
\[ \Dia^{2-\op} \rightarrow \Dia^{\cor} \]
is the identity on objects. A 1-morphism $\alpha \in \Hom(I_1 \times \cdots \times I_n, J)$ is mapped to the correspondence
\[ \xymatrix{
&&  I_1 \times \cdots \times I_n \times_{/J} J \ar[lld] \ar[d] \ar[rrd] \\
I_1 & \cdots & I_n  &; & J
} \]
and a 2-morphism $\mu: \alpha \rightarrow \alpha'$ to the morphism $(I_1 \times \cdots \times I_n) \times_{/\alpha', J} J \rightarrow (I_1 \times \cdots \times  I_n) \times_{/\alpha,J} J$ induced by $\mu$. Note that the projections from $I_1 \times \cdots \times I_n \times_{/J} J$ to $I_1 \times \cdots \times I_n$, and to $J$, are respectively an 
opfibration, and a fibration. 

To establish the pseudo-functoriality, we have to show that there is a natural isomorphism of correspondences between
\[ (I_1 \times \cdots \times I_n) \times_{/J_i} J_i \times_{J_i} (J_1 \times \cdots \times J_m) \times_{/K} K  \]
\[ = (I_1 \times \cdots \times I_n) \times_{/J_i} (J_1 \times \cdots \times J_m) \times_{/K} K  \]
and
\[ (J_1 \times \cdots \times J_{i-1} \times I_1 \times \cdots \times I_n \times J_{i+1} \times \cdots \times J_m) \times_{/K} K \]
in $\tau_1(\Span^F(J_1, \dots, J_{i-1}, I_1, \dots, I_n, J_{i+1}, \dots, J_m; K))$. One checks that there is even an adjunction between the two categories which establishes this isomorphism.

The pseudo-functor
\[ \Dia^{1-\op} \rightarrow \Dia^{\cor} \]
sends a multimorphism given by $\{ \alpha_k: J \rightarrow I_k \}$ to the correspondence 
\[ \xymatrix{
&&  I_1 \times \cdots \times I_n \times_{/(I_1 \times \cdots \times I_n)} J \ar[lld] \ar[d] \ar[rrd] \\
I_1 & \cdots & I_n  &; & J
} \]

To establish the pseudo-functoriality, we have to show that there is a natural isomorphism of correspondences between
\[ (I_1 \times \cdots \times I_n) \times_{/(I_1 \times \cdots \times I_n)} J_i \times_{J_i} (J_1 \times \cdots \times J_m) \times_{/(J_1 \times \cdots \times J_m)} K  \]
\[ = (I_1 \times \cdots \times I_n) \times_{/(I_1 \times \cdots \times I_n)} (J_1 \times \cdots \times J_m) \times_{/(J_1 \times \cdots \times J_m)} K  \]
and
\[ (J_1 \times \cdots \times J_{i-1} \times I_1 \times \cdots \times I_n \times J_{i+1} \times \cdots \times J_m) \times_{(J_1 \times \cdots \times J_{i-1} \times I_1 \times \cdots \times I_n \times J_{i+1} \times \cdots \times J_m)} K \]
in $\tau_1(\Span^F(J_1, \dots, J_{i-1}, I_1, \dots, I_n, J_{i+1}, \dots, J_m; K))$. One checks that there is even an adjunction between the two categories which establishes this isomorphism. 

The requested pseudo-functor
\[ \{\cdot\}  \rightarrow \Dia^{2-\op} \]
with value $I$ is given by the composition of the obvious pseudo-functor $\{\cdot\}  \rightarrow \Dia^{1-\op}$, sending the unique multimorphism in $\Hom(\cdot, \dots, \cdot\,; \cdot)$ to $\{\id_I\}_{i=1..n}$, with the previous pseudo-functor
$\Dia^{1-\op} \rightarrow \Dia^{\cor}$.
\end{proof}

\begin{PROP}\label{PROPDIACORMONOIDAL}
The 2-multicategory $\Dia^{\cor}$ is (strictly symmetric) 1-bifibered and (trivially) 2-bifibered over $\{\cdot\}$ hence it is a (strictly symmetric) monoidal 2-category with monoidal structure represented by
\[ I \otimes J = I \times J \]
and internal hom
\[ \mathcal{HOM}(I,J) = I^{\op} \times J \]
with unit given by the final diagram $\{\cdot\}$. 
In particular every object is dualizable w.r.t.\@ the final diagram and the duality functor is $I \mapsto I^{\op}$ on the objects, while on 1-morphisms it is given by the composition of equivalences:
\[  \Hom_{\Dia^{\cor}}(J^{\op}, I^{\op}) \cong \tau_1(\Fun(J \times I^{\op}, \Dia)) = \tau_1(\Fun(I^{\op} \times J, \Dia)) \cong \Hom_{\Dia^{\cor}}(I, J).  \]
\end{PROP}
\begin{proof}
By Corollary~\ref{KOREQSPAN} we have equivalences 
\begin{equation}\label{eqequiv1} 
\tau_1(\Span(I_1, I_2; J)) \cong \tau_1(\Fun(I_1^{\op} \times I_2^{\op} \times J, \Dia)) 
\end{equation}
and also
\begin{equation}\label{eqequiv3} 
 \tau_1(\Span(I_1 \times I_2; J)) \cong \tau_1(\Fun(I_1^{\op} \times I_2^{\op} \times J, \Dia)). 
\end{equation}
Obviously the composition of (\ref{eqequiv1}) with the inverse of (\ref{eqequiv3}) is isomorphic to the canonical equivalence
\[ \tau_1(\Span(I_1, I_2; J)) \rightarrow \tau_1(\Span(I_1 \times I_2; J))  \]
given by 
\[ \left(\vcenter{\xymatrix{
& & A \ar[lld]_{\alpha_1} \ar[ld]^{\alpha_2} \ar[rd]^\beta &  \\
I_1 & I_2 & ; & J }} \right)
\mapsto \left( \vcenter{ \xymatrix{
&  A \ar[ld]_{(\alpha_1, \alpha_2)} \ar[rd]^\beta &  \\
I_1 \times  I_2 & ; & J }} \right).
\]
 
Furthermore this canonical equivalence preserves the $\Span^F$-subcategories and is compatible with composition, by definition of the composition by fiber products.

Similarly, by Corollary~\ref{KOREQSPAN} again, we have an equivalence 
\begin{equation}\label{eqequiv2}
 \tau_1(\Hom(I_1; I_2^{\op} \times J)) \cong \tau_1(\Fun(I_1^{\op} \times I_2^{\op} \times J, \Dia)). 
 \end{equation}
Explicitly the  equivalence (\ref{eqequiv1}) maps a correspondence
\[ \xymatrix{
& & A \ar[lld]_{\alpha_1} \ar[ld]^{\alpha_2} \ar[rd]^\beta &  \\
I_1 & I_2 & ; & J }
\]
to the functor
\begin{eqnarray*} 
F_\xi: I_1^{\op} \times I_2^{\op} \times J &\rightarrow & \Dia \\
(i_1, i_2, j) & \mapsto &   (i_1, i_2) \times_{/ I_1\times I_2} A  \times_{/J} j 
\end{eqnarray*}
and the  inverse of (\ref{eqequiv2}) maps this to
\[ \xymatrix{
&  \int \nabla F_\xi  \ar[ld]_{} \ar[rd] &  \\
I_1  & ; & I_2^{\op} \times J }
\]
Explicitly the category
\[  \int \nabla F_\xi \]
has objects $(i_1, i_2, j, a, \mu_1, \mu_2, \nu)$ where $\mu_1: i_1 \rightarrow \alpha_1(a) , \mu_2:  i_2 \rightarrow \alpha_2(a) , \nu: \beta(a) \rightarrow j$.
Morphisms $(i_1, i_2, j, a, \mu_1, \mu_2, \nu) \rightarrow (i_1', i_2', j', a', \mu_1', \mu_2', \nu')$ are morphisms
$i_1 \rightarrow i_1', i_2' \rightarrow i_2, j \rightarrow j', a \rightarrow a'$ such that the obvious diagrams commute. 
This again preserves the $\Span^F$-subcategories and is compatible with composition. 
\end{proof}

\begin{PAR}
We can also investigate how the corresponding Cartesian resp.\@ coCartesian morphisms look like: The trivial correspondence
\[ \xymatrix{
&I_1 \times I_2   \ar@{=}[ld]_{} \ar@{=}[rd]_{}  \\
I_1 \times I_2 & ;  & I_1 \times I_2
} \]
corresponds, by the explicit description given in the proof, to the morphism
\[ \xymatrix{
& I_1 \times I_2   \ar[ld]_{} \ar[d]_{} \ar@{=}[rrd] \\
I_1 & I_2 & ; & I_1 \times I_2
} \]
which therefore constitutes the corresponding coCartesian morphism.

The trivial correspondence
\[ \xymatrix{
&I_1^{\op} \times I_2   \ar@{=}[ld]_{} \ar@{=}[rd]_{}  \\
I_1^{\op} \times I_2 & ;   & I_1^{\op} \times I_2
} \]
corresponds (up to 2-isomorphism) to the functor
\begin{eqnarray*} 
I_1 \times I_2^{\op} \times I_1^{\op} \times I_2 & \rightarrow & \Dia \\
 (i_1, i_2', i_1', i_2) & \mapsto & \Hom(i_1', i_1) \times \Hom(i_2', i_2) 
\end{eqnarray*} 
where the image consists of discrete categories.
It corresponds (up to 2-isomorphism) to the 1-morphism
\[ \xymatrix{
& \tw(I_1^{\op}) \times I_2 \times_{/I_2} I_2 \ar[ld]_{} \ar[d]_{} \ar[rrd] \\
I_1 & I_2 \times I_1^{\op} & ; & I_2
} \]
or simply to the 1-morphism
\[ \xymatrix{
& \tw(I_1^{\op}) \times I_2  \ar[ld]_{} \ar[d]_{} \ar[rrd] \\
I_1 & I_2 \times I_1^{\op} & ; & I_2
} \]
which therefore is the corresponding Cartesian  morphism. Here for a category $I$, the category $\tw(I) = \int \Hom_I(-,-)$ is the twisted arrow category. 
In particular, the duality morphism in $\Hom(I, I^{\op}; \cdot)$ is given by the multicorrespondence of diagrams:
\[ \xymatrix{
& \tw(I_1^{\op})  \ar[ld]_{} \ar[d]_{} \ar[rrd] \\
I_1 & I_1^{\op} & ; & \{\cdot\}
} \]
\end{PAR}

\section{The canonical Wirthm\"uller context of a fibered multiderivator}\label{SECTIONWIRTHFIBDER}

In the next two sections it is proven that the axioms of a fibered multiderivator can be encoded as a fibration over the category $\Dia^{\cor}$ defined in Section \ref{SECTIONDIACOR}. 

\begin{PAR}

Recall Definition~\ref{DEFSPAN}, where $\Span(I_1, \dots, I_n; J)$ was defined. 
Let $\SSS$ be a pre-multiderivator (cf.\@ \cite[Definiton~1.2.1.]{Hor15}). 
Such a pre-multiderivator defines, for each tupel of diagrams $I_1, \dots, I_n; J$ in  $\Dia$ and objects $S_i \in \SSS(I_i)$ and $T \in \SSS(J)$, a 2-functor
\[ \Span_{\SSS}: \Span(I_1, \dots, I_n; J)^{1-\op} \rightarrow \mathcal{SET}. \]
where $\mathcal{SET}$ is considered a 2-category with only identities as 2-morphisms. 
$\Span_{\SSS}$ maps a multicorrespondence of diagrams in $\Dia$
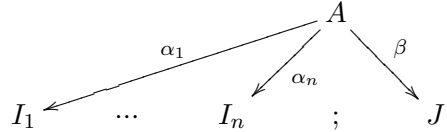
\[ \xymatrix{ &&&A \ar[dlll]_{\alpha_1} \ar[dl]^{\alpha_n}  \ar[dr]^{\beta} \\ 
I_1 & \cdots & I_n &;& J
} \]
to the set
\[ \Hom_{\SSS(A)}(\alpha_1^*S_1, \dots, \alpha_n^*S_n; \beta^* T), \]
and maps a 1-morphism $(\gamma, \nu_1, \dots, \nu_n, \mu)$ to the map
\[ \rho \mapsto \SSS(\mu)(T) \circ (\gamma^*\rho) \circ (\SSS(\nu_1)(S_1), \dots, \SSS(\nu_n)(S_n)). \]
It is immediate from the axioms of a pre-multiderivator that this defines a 2-functor, in particular that it sends 1-morphisms which are connected by a 2-morphism to the same map.
\end{PAR}

\begin{DEF}
Let $S_i \in \SSS(I_i)$, $T \in \SSS(J)$ be objects and define 
$\Span_{\SSS}((I_1, S_1), \dots, (I_n, S_n); (J, T))$ be the strict 2-category obtained from the pseudo-functor $\Span_{\SSS}$ via the Grothendieck construction \ref{GROTHCONSTR}.
Explicitly: 
\begin{enumerate}
\item The objects are correspondences
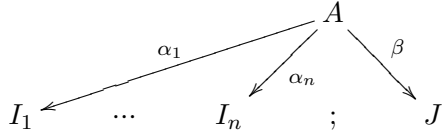
\[ \xymatrix{ &&&A \ar[dlll]_{\alpha_1} \ar[dl]^{\alpha_n}  \ar[dr]^{\beta} \\ 
I_1 & \cdots & I_n &; & J
} \]
together with a 1-morphism 
\[ \rho \in \Hom(\alpha_1^*S_1, \dots, \alpha_n^*S_n; \beta^*T) \]
in $\SSS(A)$. 

\item The 1-morphisms $(A, \alpha_1, \dots, \alpha_n, \beta, \rho) \rightarrow (A', \alpha_1', \dots, \alpha_n', \beta', \rho')$ are tuples
$(\gamma, \nu_1, \dots, \nu_n, \mu)$, where 
 $\gamma: A \rightarrow A'$ is a functor,  $\nu_i$ is a natural transformation 
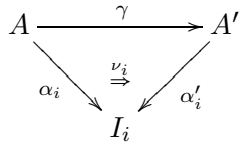
\[ \xymatrix{
A \ar[rd]_{\alpha_i} \ar[rr]^{\gamma} & \ar@{}[d]|-{\overset{\nu_i}{\Rightarrow}} & A' \ar[ld]^{\alpha_i'}  \\
&I_i
} \]
and $\mu$ is a natural transformation 
\[ \xymatrix{
A \ar[rd]_{\beta} \ar[rr]^{\gamma} & \ar@{}[d]|-{\overset{\mu}{\Leftarrow}}  & A' \ar[ld]^{\beta'}  \\
&J
} \]
such that the diagram
\begin{equation} \label{eqcommdiaspand}
\vcenter{ \xymatrix{
(\gamma^*(\alpha_1')^*S_1, \dots, \gamma^*(\alpha_n')^*S_n) \ar[r]^-{\gamma^*\rho'}& \gamma^*(\beta')^*T \ar[d]^-{\SSS(\mu)} \\
(\alpha_1^*S_1, \dots, \alpha_n^*S_n) \ar[r]^-{\rho}  \ar[u]^-{(\SSS(\nu_1), \dots, \SSS(\nu_n))}  & \beta^*T
} }
\end{equation}
commutes. 

\item The 2-morphisms are the natural transformations
$\eta: \gamma \Rightarrow \gamma'$ such that $(\alpha_i' \ast \eta) \circ \nu_i  = \nu_i'$ and $\mu' \circ (\beta' \ast \eta)   = \mu$.
\end{enumerate}
We again define the full subcategory $\Span^F_{\SSS}$ insisting that $\alpha_1 \times \cdots \times \alpha_n: A \rightarrow I_1 \times \cdots \times I_n$ is a fibration and $\beta$ is an opfibration.
\end{DEF}

\begin{LEMMA}\label{LEMMASPAN1}
Let $p: \DD \rightarrow \SSS$ be a strict morphism of pre-multiderivators (cf.\@ \cite[Defintion 1.2.1]{Hor15}), and let $I_1, \dots, I_n, J$ be diagrams in $\Dia$, let $\mathcal{E}_i$ be objects in $\DD(I_i)$ lying over $S_i$ and
$\mathcal{F}$ an object in $\DD(J)$ lying over $T$. 

Consider the strictly commuting diagram of 2-categories and strict 2-functors
\[ \xymatrix{
\Span^F_{\DD}((I_1, \mathcal{E}_1), \dots, (I_n, \mathcal{E}_n); (J,\mathcal{F})) \ar@{^{(}->}[r] \ar[d] & \Span_{\DD}((I_1, \mathcal{E}_1), \dots, (I_n, \mathcal{E}_n); (J,\mathcal{F})) \ar[d] \\
\Span^F_{\SSS}((I_1, S_1), \dots, (I_n, S_n); (J,T)) \ar@{^{(}->}[r] \ar[d] & \Span_{\SSS}((I_1, S_1), \dots, (I_n, S_n); (J,T)) \ar[d] \\
\Span^F(I_1, \dots, I_n; J) \ar@{^{(}->}[r] &  \Span(I_1, \dots, I_n; J)
} \]
\begin{enumerate}
\item All vertical 2-functors are 1-fibrations and 2-bifibrations with discrete fibers. 
\item Every object in a 2-category on the right hand side is in the image of the corresponding horizontal 2-functor {\em up to a chain of adjunctions}. 
\end{enumerate}
\end{LEMMA}

\begin{proof}
1.\@ follows directly from the definition of the corresponding categories by a Grothendieck construction.
2.\@ is a refinement of \ref{KOREQSPAN} proved as follows. 
We  first embed the left hand side category, say $\Span_{\SSS}^F((I_1, S_1), \dots, (I_n, S_n); (J,T))$, into the full subcategory of $\Span_{\SSS}((I_1, S_1), \dots, (I_n, S_n); (J,T))$ consisting of objects $(A, \alpha_1, \dots, \alpha_n, \beta, \rho)$, in which $\beta$ is an opfibration but the $\alpha_i$ are arbitrary. 
We will show that every object is connected by an adjunction with an object of this bigger subcategory. By a similar argument one shows that this holds also for the second inclusion. 

Consider an arbitrary correspondence $\xi'$ of diagrams in $\Dia$
 \[ \xymatrix{
& & & A \ar[rd]^{\beta} \ar[ld]^{\alpha_n}  \ar[llld]_{\alpha_1} \\
I_1 & \cdots & I_n & ; & J
} \]
and the 1-morphisms in $\Span(I_1, \dots, I_n; J)$
 
 \[ 
 \xymatrix{
& & & A  \times_{/J} J \ar[rd]^{\pr_2} \ar[ld]^{\alpha_n\pr_1} \ar[llld]_{\alpha_1\pr_1} \ar[dd]_{\pr_1} \\
I_1 & \cdots & I_n & \ar@{}[r]|-{\overset{\mu}{\Rightarrow}} & J \\
& & & A  \ar[ru]_{\beta} \ar[lu]_{\alpha_n} \ar[lllu]^{\alpha_1} 
}  \quad  \xymatrix{
& & & A  \ar[rd]^{\beta} \ar[ld]^{\alpha_n} \ar[llld]_{\alpha_1} \ar[dd]^{\Delta}   \\
I_1 & \cdots & I_n &  & J \\
& & & A  \times_{/J} J \ar[ru]_{\pr_2} \ar[lu]_{\alpha_n\pr_1} \ar[lllu]^{\alpha_1\pr_1} 
} \]

One easily checks that $\pr_1 \circ \Delta = \id_A$ and that the obvious 2-morphism  $\Delta \circ \pr_1 \Rightarrow \id_{A\times_{/J} J}$ induced by $\mu$
define an adjunction in the 2-category $\Span(I_1, \dots, I_n; J)$. 
Using Lemma~\ref{LEMMAAUX} below, we get a corresponding adjunction also in the 2-category $\Span_{\SSS}((I_1, S_1), \dots, (I_n, S_n); (J, T))$.
\end{proof}

\begin{LEMMA}\label{LEMMAAUX}
Let $\mathcal{D} \rightarrow \mathcal{S}$ be a 1-fibration and 2-(op)fibration of 2-categories with 1-categorical fibers.
Given  an adjunction in $\mathcal{S}$
\[ \xymatrix{ S \ar@/^10pt/[rr]^F & & \ar@/^10pt/[ll]^G T  } \]
with counit $G \circ F = \id_S$ being the identity and unit $F \circ G \Rightarrow \id_T$, 
for any object $\mathcal{E} \in \mathcal{D}_S$ there is an adjunction 
\[ \xymatrix{  \mathcal{E} \ar@/^10pt/[rr]^{\widetilde{F}} & & \ar@/^10pt/[ll]^{\widetilde{G}} \mathcal{F}  } \]

in $\mathcal{D}$, lying over the previous one, where $\widetilde{F}$ and $\widetilde{G}$ are Cartesian. 
\end{LEMMA}
\begin{proof}
We concentrate on the 2-opfibered case and may assume by Proposition~\ref{PROPGROTHCONSTR} that $\mathcal{D}$ is equal to the Grothendieck construction applied to a pseudo-functor $\Psi: \mathcal{S}^{1-\op} \rightarrow \mathcal{CAT}$. 
We then have corresponding pullback functors $F^\bullet:=\Psi(F)$, $G^\bullet:=\Psi(G)$ and a 2-isomorphism $\eta:  \id_{\Psi(S)} \cong F^\bullet \circ G^\bullet$ and
a 2-morphism $\mu:  G^\bullet \circ F^\bullet \Rightarrow \id_{\Psi(T)} $ given by the pseudo-functoriality and the contravariant functoriality on 2-morphisms. 

We define $\widetilde{G}:=(G, \id_{G^\bullet \mathcal{E}}): G^\bullet \mathcal{E} \rightarrow \mathcal{E}$, the canonical Cartesian morphism, and
 $\widetilde{F}:=(F, \eta(\mathcal{E})): \mathcal{E} \rightarrow G^\bullet \mathcal{E}$, which is Cartesian as well, $\eta(\mathcal{E})$ being an isomorphism. 
 There is a 2-isomorphism $\widetilde{G} \circ \widetilde{F}  \cong \id_{\mathcal{E}}$, and a 2-morphism $ \widetilde{F} \circ \widetilde{G} \rightarrow \id_{G^\bullet \mathcal{E}}$ given by $\mu(G^\bullet \mathcal{E})$. One checks that those define unit and counit of an adjunction again. 
 
 In the 2-fibered case we set $\widetilde{F}:=(F, \eta(\mathcal{E})^{-1}): \mathcal{E} \rightarrow G^\bullet \mathcal{E}$ and may reason analogously. 
 \end{proof}

\begin{LEMMA}\label{PAREQSPANFIBERED}
Let $p: \DD \rightarrow \SSS$ be a morphism of (lax/oplax) 2-pre-multiderivators. Consider the following strictly commuting diagram of functors obtained from the one of Lemma~\ref{LEMMASPAN1} by 1-truncation (\ref{1TRUNC}):  
\[ \xymatrix{
\tau_1(\Span^F_{\DD}((I_1, \mathcal{E}_1), \dots, (I_n, \mathcal{E}_n); (J,\mathcal{F}))) \ar@{^{(}->}[r] \ar[d] & \tau_1(\Span_{\DD}((I_1, \mathcal{E}_1), \dots, (I_n, \mathcal{E}_n); (J,\mathcal{F}))) \ar[d] \\
\tau_1(\Span^F_{\SSS}((I_1, S_1), \dots, (I_n, S_n); (J,T))) \ar@{^{(}->}[r] \ar[d] & \tau_1(\Span_{\SSS}((I_1, S_1), \dots, (I_n, S_n); (J,T))) \ar[d] \\
\tau_1( \Span^F(I_1, \dots, I_n; J)) \ar@{^{(}->}[r] & \tau_1( \Span(I_1, \dots, I_n; J))
} \]
\begin{enumerate}
\item The horizontal functors are equivalences. 
\item All vertical morphisms are fibrations with discrete fibers. Furthermore the horizontal functors map Cartesian morphisms to Cartesian morphisms. 
\end{enumerate}

\end{LEMMA}
\begin{proof}
The horizontal morphisms are equivalences by definition of the truncation and Lemma~\ref{LEMMASPAN1}, 2. 
If we have a 1-fibration and 2-isofibration of 2-categories $\mathcal{D} \rightarrow \mathcal{C}$ with
{\em discrete} fibers  then the truncation $\tau_1(\mathcal{D}) \rightarrow \tau_1(\mathcal{C})$ is again fibered (in the 1-categorical sense). Hence the second assertion follows from Lemma~\ref{LEMMASPAN1}, 1.
\end{proof}

\begin{DEF}\label{PARDEFDIACORD}
We define a 2-multicategory $\Dia^{\cor}(\SSS)$ with a functor (with 1-categorical fibers) 
\[ \Dia^{\cor}(\SSS) \rightarrow \Dia^{\cor} \] 
as follows:
\begin{enumerate}
\item The objects of $\Dia^{\cor}(\SSS)$ are the pairs $(I, S)$ with $I \in \Dia$ and $S \in \SSS(I)$.
\item The  category $\Hom_{\Dia^{\cor}(\SSS)}((I_1, S_1), \dots, (I_n, S_n); (J, T))$ of 1-morphisms of $\Dia^{\cor}(\SSS)$ is the truncated category
$\tau_1(\Span_{\SSS}^F((I_1, S_1), \dots, (I_n, S_n); (J, T)))$. 
Composition is defined by composition in $\Dia^{cor}$, i.e.\@ by the fiber product
 
\[ \xymatrix{
&&&&& A \times_{J_i} B  \ar[lld]_{\pr_1} \ar[rrd]^{\pr_2} \\
&&& A \ar[llld] \ar[ld] \ar[rrrd]^{\beta_A} &&&&  B \ar[llld] \ar[ld]^{\alpha_{B,i}}  \ar[rd] \ar[rrrd] \\
I_1 & \cdots & I_n & ; & J_1 & \cdots & J_i & \cdots & J_m & ; & K
} \]
and $\rho_A: \alpha_{A,1}^* S_1, \dots, \alpha_{A,n}^* S_n \rightarrow \beta_A^* T_i$ is composed with
$\rho_B: \alpha_{B,1}^* T_1, \dots, \alpha_{B,m}^* T_m \rightarrow \beta_B^* U$ to the morphism
 \[ (\pr_2^* \rho_B) \circ_i (\pr_1^* \rho_A). \]
\end{enumerate}
\end{DEF}

\begin{BEM}\label{SSS2FIB}
By definition $\Dia^{\cor}(\SSS) \rightarrow \Dia^{\cor}$ has 1-categorical fibers and, by Lemma~\ref{PAREQSPANFIBERED}, it is 2-fibered over $\Dia^{\cor}$. 
In a subsequent article \cite[Section 4]{Hor16} we generalize this definition to pre-2-multiderivators. 
\end{BEM}

\begin{PAR}\label{PAREMBEDDING1}
Let $\SSS$ be a pre-multiderivator. 
Recall the definition of $\Dia(\SSS)$ from \cite[\S 1.6]{Hor15}:
\begin{enumerate}
\item The objects of $\Dia(\SSS)$ are the pairs $(I, S)$ where $I \in \Dia$ and $S \in \SSS(S)$. 
\item The 1-morphisms in $\Hom_{\Dia(\SSS)}((I, S); (J, T))$ are pairs $(\alpha, f)$, where $\alpha: I \rightarrow J$ is a functor in $\Dia$ together with a morphism
\[ f:  S \rightarrow  \alpha^*T. \]
\item The 2-morphisms $(\alpha, f) \Rightarrow (\alpha', f')$ are given by natural transformations $\delta: \alpha \rightarrow \alpha'$ such that the diagram
\[ \xymatrix{
\alpha^* S  \ar[r]^f  \ar[d]_{\SSS(\delta)} & T   \\
(\alpha')^* S \ar[ru]_{f'}  & 
} \]
commutes. 
\end{enumerate}
This category is 1-fibered and 2-fibered over $\Dia$.
There is a commutative diagram of pseudo-functors of 2-categories (not of 2-multicategories)
\[ \xymatrix{
\Dia(\SSS)^{2-\op} \ar[r] \ar[d] & \Dia^{\cor}(\SSS) \ar[d] \\
\Dia^{2-\op} \ar[r] & \Dia^{\cor}
} \]
where the bottom horizontal pseudo-functor is the one of  Proposition~\ref{PROPPSEUDOFUNCTDIA}, 1.
\end{PAR}

\begin{PAR}\label{PAREMBEDDING2}
Let $\SSS$ be a pre-multiderivator. 
Recall the definition of $\Dia^{\op}(\SSS)$ from  \cite[\S 1.6]{Hor15}. We define here the category $\Dia^{\op}(\SSS)^{1-\op}$ even as a 2-multicategory:
\begin{enumerate}
\item The objects of $\Dia^{\op}(\SSS)^{1-\op}$ are the pairs $(I, S)$ where $I \in \Dia$ and $S \in \SSS(S)$. 
\item The 1-morphisms in $\Hom_{\Dia^{\op}(\SSS)^{1-\op}}((I_1, S_1), \dots, (I_n, S_n); (J, T))$ are collections $\{ \alpha_i: J \rightarrow I_i \}$ together with a morphism
\[ f \in \Hom_{\SSS(J)}( \alpha_1^*S_1, \dots, \alpha_n^*S_n;  T). \]
\item The 2-morphisms are given by collections $\{ \delta_i: \alpha_i \rightarrow \alpha_i' \}$ such that the diagram
\[ \xymatrix{
(\alpha_1^* S_1, \dots,  \alpha_n^*S_n )  \ar[r] \ar[d] & T   \\
((\alpha_1')^* S_1, \dots,   (\alpha_n')^*S_n)  \ar[ur] & 
} \]
commutes. 
\end{enumerate}
There is a commutative diagram of pseudo-functors of 2-multicategories
\[ \xymatrix{
\Dia^{\op}(\SSS)^{1-\op} \ar[r] \ar[d] & \Dia^{\cor}(\SSS) \ar[d] \\
\Dia^{1-\op} \ar[r] & \Dia^{\cor}
} \]
where the bottom horizontal pseudo-functor is the one of  Proposition~\ref{PROPPSEUDOFUNCTDIA}, 2.
\end{PAR}

\section{Yoga of correspondences of diagrams in a pre-multiderivator}\label{YOGA}

Let $\SSS$ be a pre-multiderivator.
This section contains a discussion which will improve our understanding of the category $\Dia^{\cor}(\SSS)$.

\begin{PAR}\label{DEF3MORPH}
We will define three types of generating\footnote{``Generating'' in the sense that any 1-morphism in $\Dia^{\cor}(\SSS)$ is 2-isomorphic to a composition of these (cf.\@ Corollary~\ref{KORFIBDER2_2}).} 1-morphisms in $\Dia^{\cor}(\SSS)$. We first define them as objects in the categories
$\Span_{\SSS}(\dots)$ (without the restriction ${}^F$). 

\begin{enumerate}
\item[${[\beta^{(S)}]}$] for a functor $\beta: I \rightarrow J$ in $\Dia$ and an object $S \in \SSS(J)$, 
consists of the correspondence of diagrams
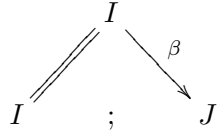
\[ \xymatrix{
& I \ar[rd]^{\beta}\ar@{=}[ld] \\
I & ; & J
} \]
and over it in $\tau_1(\Span_{\SSS}((I, \beta^*S ); (J, S)))$ the canonical correspondence given by the identity $\id_{\beta^*S}$.

\item[${[\alpha^{(S)}]'}$] for a functor $\alpha: I \rightarrow J$ in $\Dia$ and an object $S \in \SSS(J)$,
consists of the correspondence of diagrams 
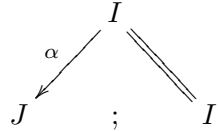
\[ \xymatrix{
& I \ar[ld]_{\alpha}\ar@{=}[rd] \\
J & ; & I
} \]
and over it in $\tau_1(\Span_{\SSS}((J, S); (I, \alpha^*S )))$ the canonical correspondence given by the identity $\id_{\alpha^*S}$.

\item[${[f]}$] for a morphism $f \in \Hom_{\SSS(A)}(S_1, \dots, S_n; T)$, where $A$ is any diagram in $\Dia$, and $S_1, \dots, S_n, T$ are objects in $\SSS(A)$, is defined by the trivial correspondence of diagrams 
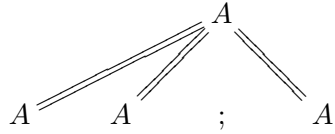
\[ \xymatrix{
& & A \ar@{=}[ld] \ar@{=}[lld] \ar@{=}[rd] \\
A & A &;  & A
} \]
together with $f$.  
\end{enumerate}
\end{PAR}

\begin{PAR}\label{PARDIACORSCAN1MOR}
Note that the correspondences of the last paragraph do not define 1-morphisms in $\Dia^{\cor}(\SSS)$ yet, as we defined it, because they are not always objects in the $\Span^F$ subcategory ($[\alpha^{(S)}]'$ is already, if $\alpha$ is a fibration; $[\beta^{(S)}]$ is, if $\beta$ is an opfibration; and $[f]$ is, if $n=0,1$, respectively). 

From now on, we denote by the same symbols $[\alpha^{(S)}], [\beta^{(S)}]', [f]$ chosen 1-morphisms in $\Dia^{\cor}(\SSS)$ which are isomorphic to those defined above in the $\tau_1$-categories (cf.\@ Lemma~\ref{PAREQSPANFIBERED}). Those are determined only up to 2-isomorphism in $\Dia^{\cor}(\SSS)$. 

For definiteness, we choose $[\beta^{(S)}]$ to be the correspondence 
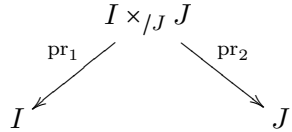
\[ \xymatrix{
& I \times_{/J} J \ar[rd]^{\pr_2}\ar[ld]_{\pr_1} \\
I & & J
} \]
and over it in $\tau_1(\Span_{\SSS}((I, \beta^*S ); (J, S)))$ the morphism $\pr_1^*\beta^*S \rightarrow \pr_2^*S$ 
given by the natural transformation $\mu_\beta: \beta \circ \pr_1 \Rightarrow \pr_2$.
Similarly, we choose $[\alpha^{(S)}]'$ to be the correspondence 
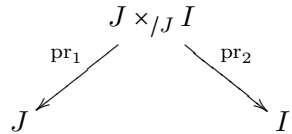
\[ \xymatrix{
& J \times_{/J} I \ar[rd]^{\pr_2}\ar[ld]_{\pr_1} \\
J & & I
} \]
and over it in $\tau_1(\Span_{\SSS}((J, S), (I, \alpha^*S )))$ the morphism $\pr_1^* S \rightarrow \pr_2^* \alpha^*S$ 
given by the natural transformation $\mu_\alpha: \pr_1 \Rightarrow \alpha \circ  \pr_2$.
\end{PAR}

\begin{PAR}\label{DIACORSADJUNCTION}
For any $\alpha: I \rightarrow J$, and an object $S \in \SSS(J)$, we define a 2-morphism
\[ \epsilon: \id \Rightarrow [\alpha^{(S)}] \circ  [\alpha^{(S)}]' \]
given by the commutative diagrams
\[ \vcenter{\xymatrix{
& I \ar@{=}[rd]^{} \ar@{=}[ld] \ar[dd]^{\Delta} \\
I & & I \\
& I \times_{/J} J \times_{/J} I \ar[ru]_{\pr_3}\ar[lu]^{\pr_1} 
} } \qquad \vcenter{ \xymatrix{
\Delta^* \pr_1^*  \alpha^* S \ar@{=}[rrrr]^{\Delta^*(\SSS( \mu_2 \circ \mu_1)(S)) = \id_{\alpha^*S} } &&&& \Delta^* \pr_3^* \alpha^* S  \ar@{=}[d] \\
\alpha^* S \ar@{=}[rrrr] \ar@{=}[u] &&&& \alpha^* S
} } \]

and we define a 2-morphism
\[ \mu:  [\alpha^{(S)}]' \circ  [\alpha^{(S)}] \Rightarrow \id \]
given by the commutative diagrams
\[ \vcenter{ \xymatrix{
& J \times_{/J} I \times_{/J} J  \ar[dd]|{\alpha \pr_2}  \ar[rd]^{\pr_3} \ar[ld]_{\pr_1}  \\
J \ar@{}[r]|{\overset{\mu_2}{\Rightarrow}} &  \ar@{}[r]|{\overset{\mu_1}{\Rightarrow}}  & J \\
& J \ar@{=}[ru]^{} \ar@{=}[lu]
} } \qquad  \vcenter{ \xymatrix{
\pr_2^* \alpha^*  S   \ar@{=}[rrr] &&&  \pr_2^* \alpha^*  S  \ar[d]^{\SSS(\mu_1)(S)} \\
\pr_1^* S \ar[rrr]_{\SSS(\mu_2 \circ \mu_1)(S)} \ar[u]^{\SSS(\mu_2)(S)}  &&& \pr_3^* S
}  } \]
\end{PAR}
\begin{PAR}\label{DIACORS2FUNCT}
A natural transformation $\nu: \alpha \Rightarrow \beta$ establishes a morphism
\[ [\nu]: [\SSS(\nu)(S)] \circ [\alpha^{(S)}] \Rightarrow [\beta^{(S)}]  \]
given by the commutative diagrams: 
\[ \vcenter{ \xymatrix{
& J \times_{/J,\beta} I \ar[rd]^{\pr_2'} \ar[ld]_{\pr_1'} \ar[dd]^{\widetilde{\nu}} \\
J & & I \\
& J \times_{/J,\alpha} I  \ar[ru]_{\pr_2}\ar[lu]^{\pr_1} 
} } \qquad  \vcenter{ \xymatrix{
\widetilde{\nu}^* \pr_1^*  S  \ar[rrr]^-{ \widetilde{\nu}^* \SSS(\mu_\alpha)} &&& \widetilde{\nu}^* \pr_2^* \alpha^* S  \ar[rr]^-{ \widetilde{\nu}^* \pr_2^* \SSS(\nu)} && \widetilde{\nu}^* \pr_2^* \beta^*S  \ar@{=}[d]  \\
(\pr_1')^* S \ar[rrrrr]_-{\SSS(\mu_\beta)(S)} \ar@{=}[u]  & & & &  & (\pr_2')^* \beta^* S
} } \]

Note that we have the equation of natural transformations $(\nu \ast \pr_2') \circ (\mu_\alpha \ast \widetilde{\nu}) = \mu_\beta$. 
Here $\mu_\alpha$ and $\mu_\beta$ are as in \ref{PARDIACORSCAN1MOR}.

Similarly, a natural transformation $\nu: \alpha \Rightarrow \beta$ establishes a morphism
\[ [\nu]:    [\beta^{(S)}]' \circ [\SSS(\nu)(S)] \Rightarrow  [\alpha^{(S)}]'. \]
\end{PAR}

\begin{PAR}\label{DIACORS2COMMSQUARE}
Consider the diagrams from axiom (FDer3 left/right)
\[ \xymatrix{
I \times_{/J} {j} \ar[r]^-\iota \ar[d]_p \ar@{}[rd]|{\Swarrow^\mu}  & I \ar[d]^\alpha \\
j \ar@{^{(}->}[r] & J
} \quad \xymatrix{
{j} \times_{/J} I  \ar[r]^-\iota \ar[d]_p \ar@{}[rd]|{\Nearrow^\mu}  & I \ar[d]^\alpha \\
j \ar@{^{(}->}[r] & J
} \]
By the construction in \ref{DIACORS2FUNCT}, we get a canonical 2-morphism
\begin{equation}\label{2commsquare_l}
 [\SSS(\mu)(S)]   \circ [\iota^{(\alpha^*S)}] \circ [\alpha^{(S)}] \Rightarrow   [p^{(S_j)}] \circ [j^{(S)}].
\end{equation}
and a canonical 2-morphism
\begin{equation}\label{2commsquare_r}
  [\alpha^{(S)}]' \circ   [\iota^{(\alpha^*S)}]' \circ [\SSS(\mu)(S)]  \Rightarrow  [j^{(S)}]' \circ [p^{(S_j)}]' .
\end{equation}
respectively. Here $S_j$ denotes $j^*S$ where $j$, by abuse of notation, also denotes the inclusion of the one-element category $j$ into $J$. 
\end{PAR}

\begin{PAR}
Let  $\xi$ be any 1-morphism $\Dia^{\cor}(\SSS)$ given by
\[ \xymatrix{
& && A \ar[rd]^{\beta} \ar[ld]^{\alpha_n} \ar[llld]_{\alpha_1}   \\
I_1 & \cdots &  I_n & ;  & J
} \]
and a morphism 
\[ f_\xi  \in \Hom_{\SSS(A)}(\alpha_1^*S_1, \dots, \alpha_n^*S_n; \beta^* T). \]
We define a 1-morphism $\xi \times K$ in $\Dia^{\cor}(\SSS)$ by
\[ \xymatrix{
& && A \times K \ar[rd]^{\beta \times \id} \ar[ld]^{\alpha_n \times \id} \ar[llld]_{\alpha_1 \times \id}   \\
I_1 \times K & \cdots &  I_n \times K & & J \times K
} \]
and 
\[ f_{\xi \times K} := \pr_1^* f_\xi  \in \Hom_{\SSS(A)}( \pr_1^* \alpha_1^*S_1, \dots, \pr_1^*  \alpha_n^*S_n; \pr_1^* \beta^* T). \]
Note that the here defined $\xi \times K$ does not necessarily lie in the category $\Span^F_\SSS(\dots)$. Hence we denote by $\xi \times K$ any isomorphic (in the $\tau_1$-truncation) correspondence which does
lie in $\Span^F_\SSS(\dots)$. We also define a correspondence 
 $\xi \times_j K$ in $\Dia^{\cor}(\SSS)$ by
\[ \xymatrix{
& && &&A \times K \ar[rd]^{\beta \times \id} \ar[ld]^{\alpha_n \pr_1} \ar[llld]^{\alpha_j \times \id}  \ar[llllld]_{\alpha_1 \pr_1}   \\
I_1  & \cdots&  I_j \times K & \cdots &  I_n  & & J \times K
} \]
and 
\[ f_{\xi \times_j K} := \pr_1^* \xi  \in \Hom_{\SSS(A)}( \pr_1^* \alpha_1^*S_1, \dots, \pr_1^*  \alpha_n^*S_n; \pr_1^* \beta^* T). \]
The here defined $\xi \times_j K$ does already lie in the category $\Span^F_\SSS(\dots)$. 
\end{PAR}

\begin{LEMMA}\label{LEMMAPROPDIACORS}
With the notation of \ref{DEF3MORPH}:
\begin{enumerate}
\item The 2-morphisms of \ref{DIACORSADJUNCTION}
\[ \epsilon: \id \Rightarrow [\alpha^{(S)}] \circ  [\alpha^{(S)}]' \qquad
 \mu:  [\alpha^{(S)}]' \circ  [\alpha^{(S)}] \Rightarrow \id \]
 establish an adjunction between $[\alpha^{(S)}]$ and $[\alpha^{(S)}]'$ in the 2-category $\Dia^{\cor}(\SSS)$. 
 \item The exchange 2-morphisms of (\ref{2commsquare_l}) and and of (\ref{2commsquare_r}) w.r.t.\@ the adjunction of 1.\@, namely
\[  [p^{(S_j)}]' \circ [\SSS(\mu)(S)] \circ  [\iota^{(\alpha^*S)}] \Rightarrow  [j^{(S)}] \circ [\alpha^{(S)}]'   \]
and
\[    [\iota^{(\alpha^*S)}]' \circ [\SSS(\mu)(S)] \circ [p^{(S_j)}]  \Rightarrow  [\alpha^{(S)}] \circ  [j^{(S)}]'   \]
are 2-isomorphisms. 
\item For any $\alpha: K \rightarrow L$ there are natural isomorphisms
\begin{equation} \label{commtensorpullback1}
[\alpha^{(\pr_1^*T)}]  \circ (\xi \times L) \cong (\xi \times K) \circ ([\alpha^{(\pr_1^*S_1)}], \dots, [\alpha^{(\pr_1^*S_n)}]) 
\end{equation}
and
\begin{equation}  \label{commtensorpullback2}
 [\alpha^{(\pr_1^*T)}] \circ (\xi \times_j L) \cong (\xi \times_j K) \circ_j [\alpha^{(\pr_1^*S_j)}]. 
\end{equation} 
\item 
The exchange of (\ref{commtensorpullback1}) w.r.t.\@ the adjunction of 1.\@, namely
\[ [\alpha^{(\pr_1^*T)}]' \circ (\xi \times K)   \circ ([\alpha^{(\pr_1^*S_1)}], \dots, \id, \dots,  [\alpha^{(\pr_1^*S_n)}])  \cong (\xi \times L) \circ_j [\alpha^{(\pr_1^*S_j)}]' \]
is an isomorphism if $\alpha$ is an opfibration. The exchange of (\ref{commtensorpullback2}) w.r.t.\@ the adjunction of 1.\@, namely
\[ [\alpha^{(\pr_1^*T)}]' \circ (\xi \times_j K)  \cong (\xi \times_j L) \circ_j [\alpha^{(\pr_1^*S_j)}]' \]
is an isomorphism for any $\alpha$. 
\item 
For any $f \in \Hom_{\SSS(J)}(S_1, \dots, S_n; T)$ and $\alpha: I \rightarrow J$ there is a natural isomorphism
\begin{equation} \label{commfpullback}
 [\alpha^{(T)}] \circ [ f ]  \cong [ \alpha^* f ]  \circ ([\alpha^{(S_1)}], \dots, [\alpha^{(S_n)}]).
\end{equation}
\item The exchange of (\ref{commfpullback}) w.r.t.\@ the adjunction of 1.\@, and w.r.t.\@ the $j$-th slot, namely
\[  [\alpha^{(T)}]' \circ [ \alpha^* f ] \circ ([\alpha^{(S_1)}], \dots, \id, \dots,  [\alpha^{(S_n)}])  \cong  [ f ] \circ_j  [\alpha^{(S_j)}]' \]
is an isomorphism if $\alpha$ is an opfibration.
\end{enumerate}
\end{LEMMA}
\begin{proof} A purely algebraic manipulation that we leave to the reader. \end{proof}

\begin{PAR}\label{PARPREPMORPHDIASCOR2}
Let $\DD \rightarrow \SSS$ be a morphism of (lax/oplax) pre-multiderivators satisfying (Der1) and (Der2).  
Consider the strict 2-functor
\[ \Dia^{\cor}(\DD) \rightarrow  \Dia^{\cor}(\SSS)   \]
{\em and assume that it is a 1-opfibration, and 2-bifibration with 1-categorical fibers. }
The fiber over a pair $(I, S)$ is just the fiber $\DD(I)_S$ of the usual functor $\DD(I) \rightarrow \SSS(I)$. 
The 1-opfibration and 2-fibration can be seen (via the construction of Proposition~\ref{PROPGROTHCONSTR}) as a pseudo-functor of 2-multicategories
\[ \Psi: \Dia^{\cor}(\SSS) \rightarrow \mathcal{CAT}. \]
\end{PAR}
\begin{PAR}\label{PARPREPMORPHDIASCOR2R}
If
\[ \Dia^{\cor}(\DD) \rightarrow  \Dia^{\cor}(\SSS)  \]
is a {\em 1-fibration}, and 2-fibration with 1-categorical fibers there is still an associated pseudo-functor of 2-categories (not of 2-multicategories)
\[ \Psi': \Dia^{\cor}(\SSS)^{1-\op, 2-\op} \rightarrow \mathcal{CAT}. \]
\end{PAR}

\begin{PROP}\label{PROPBASICDIACORS}With the notation of \ref{DEF3MORPH}:
\begin{enumerate}
\item 
Assume that
\[ \Dia^{\cor}(\DD) \rightarrow  \Dia^{\cor}(\SSS)   \]
is a {\em 1-opfibration}, and 2-fibration with 1-categorical fibers. 
Then the functor $\Psi$ of \ref{PARPREPMORPHDIASCOR2} maps (up to isomorphism of functors)
\begin{eqnarray*}
{ [\alpha^{(S)}]  } &\mapsto &(\alpha^{(S)})^*  \\
{ [\beta^{(S)}]' } &\mapsto& \beta_!^{(S)}  \\
{ [f] } &  \mapsto  & f_\bullet 
\end{eqnarray*}
where $\beta_!^{(S)}$ is a left adjoint of $\beta^*$ and $f_\bullet$ is a functor determined by $\Hom_{\DD(I), f}(\mathcal{E}_1, \dots, \mathcal{E}_n; \mathcal{F}) \cong \Hom_{\DD(I)_T}(f_\bullet(\mathcal{E}_1, \dots, \mathcal{E}_n), \mathcal{F})$.

\item Assume that
\[ \Dia^{\cor}(\DD) \rightarrow  \Dia^{\cor}(\SSS) \]
is a {\em 1-fibration}, and 2-fibration with 1-categorical fibers. 

Then pullback functors\footnote{In the case of $[\alpha^{(S)}]$ and $[\beta^{(S)}]'$ these are $\Psi'([\alpha^{(S)}])$ and $\Psi'([\beta^{(S)}]')$.} w.r.t.\@ the following 1-morphisms in $\Dia^{\cor}(\SSS)$ are given by 
\begin{eqnarray*}
{ [\alpha^{(S)}]  } &\mapsto &   \alpha_*^{(S)}  \\
{ [\beta^{(S)}]' } &\mapsto& (\beta^{(S)})^* \\
{ [f] } &  \mapsto  & f^{\bullet, j}  \qquad \text{pullback w.r.t.\@ the $j$-th slot.}
\end{eqnarray*}
where $\alpha_*^{(S)}$ is a right adjoint of $\alpha^*$ and $f^{\bullet, j}$ is a functor determined by $\Hom_{\DD(I), f}(\mathcal{E}_1, \dots, \mathcal{E}_n; \mathcal{F}) \cong \Hom_{\DD(I)_T}(\mathcal{E}_j, f^{\bullet, j}(\mathcal{E}_1, \overset{\widehat{j}}{\dots}, \mathcal{E}_n; \mathcal{F}))$. 
\end{enumerate}
\end{PROP}

\begin{proof}
1. We have an isomorphism of sets\footnote{We identify a small discrete category with its set of isomorphism classes.}
\[ \Hom_{\Dia^{\cor}(\DD), [\alpha^{(S)}]}((J, \mathcal{E}), (I, \mathcal{F}))  \cong \Hom_{\Dia^{\cor}(\DD)_{(I,S)}}(\Psi([\alpha^{(S)}]) \mathcal{E}, \mathcal{F}).  \]
On the other hand, by definition and by Lemma~\ref{PAREQSPANFIBERED}, the left hand side is isomorphic to the set
\[ \Hom_{\DD(I)_S}(\alpha^* \mathcal{E},  \mathcal{F}). \]
The first assertion follows from the fact that $\Dia^{\cor}(\DD)_{(I,S)} = \DD(I)_S$. 

The second assertion follows from the first because by Lemma~\ref{LEMMAPROPDIACORS}, 1.\@ the 1-morphisms $[\alpha^{(S)}]$ and $[\alpha^{(S)}]'$ are adjoint in the 2-category $\Dia^{\cor}(\SSS)$. Note that a pseudo-functor like $\Psi$
preserves adjunctions.

We have an isomorphism of sets
\[ \Hom_{\Dia^{\cor}(\DD), [f]}((A, \mathcal{E}_1), \dots, (A, \mathcal{E}_n); (A, \mathcal{F}))  \cong \Hom_{\Dia^{\cor}(\DD)_{(A,T)}}(\Psi([f])(\mathcal{E}_1, \dots, \mathcal{E}_n),  \mathcal{F}).  \]
On the other hand, by definition and by Lemma~\ref{PAREQSPANFIBERED}, the left hand side is isomorphic to the set
\[ \Hom_{\DD(I), f}(\mathcal{E}_1, \dots, \mathcal{E}_n; \mathcal{F})  \]
and the third assertion follows from the fact that $\Dia^{\cor}(\DD)_{(A,T)} = \DD(A)_T$. 

The proof of 2.\@ is completely analogous. 
\end{proof}

\begin{KOR}\label{KORFIBDER2_2}
Assuming the conditions of \ref{PARPREPMORPHDIASCOR2},
consider any correspondence 
\[ \xi'   \in \Span_{\SSS}((I_1, S_1), \dots, (I_n, S_n); (J, T)) \] consisting of 
\[ \xymatrix{ &&&A \ar[dlll]_{\alpha_1} \ar[dl]^{\alpha_n}  \ar[dr]^{\beta} \\ 
I_1 & \cdots & I_n &; & J
} \]
and a morphism 
\[ f \in \Hom(\alpha_1^*S_1, \dots, \alpha_n^*S_n; \beta^*T) \]
in $\SSS(A)$.  
\begin{enumerate}
\item Over any 1-morphism  $\xi$ in $\Dia^{\cor}(\SSS)$, which is isomorphic to $\xi'$, a corresponding push-forward functor between
fibers (which is $\Psi(\xi')$ in the discussion \ref{PARPREPMORPHDIASCOR2}) is given (up to natural isomorphism) by the composition: 
\[ \beta_!^{(T)} \circ f_{\bullet} \circ (\alpha_1^*, \dots, \alpha_n^*). \]
\item Over any 1-morphism  $\xi$ in $\Dia^{\cor}(\SSS)$, which is isomorphic to $\xi'$, a pull-back functor w.r.t.\@ any slot $j$ between
fibers (which is $\Psi'(\xi')$ in the discussion \ref{PARPREPMORPHDIASCOR2R} if $\xi$ is a 1-ary 1-morphism) is given (up to natural isomorphism) by the composition: 
\[ \alpha_{j,*}^{(S_j)} \circ f^{\bullet,j} \circ (\alpha_1^*, \overset{\widehat{j}}{\dots},  \alpha_n^*; \beta^*). \]
\end{enumerate}
\end{KOR}
\begin{proof}
Because of Proposition~\ref{PROPBASICDIACORS}, in both cases, we only have to show that there is a 2-isomorphism 
\[ \xi \cong [\beta^{(T)}]' \circ [f] \circ ([\alpha_1^{(S_1)}] , \dots, [\alpha_n^{(S_n)}]) \]
in $\Dia^{\cor}(\SSS)$, which is an easy and purely algebraic manipulation. 
\end{proof}

The ``if'' part of the following main theorem should be seen as an analogue of Proposition~\ref{LEMMA6FU}.

\begin{HAUPTSATZ}\label{MAINTHEOREMFIBDER}
Let $\DD$ and $\SSS$ be pre-multiderivators satisfying (Der1) and (Der2) (cf.\@ \cite[Definition~1.3.5.]{Hor15}).
 A strict morphism of pre-multiderivators $\DD \rightarrow \SSS$ is a left (resp.\@ right) fibered multiderivator if and only if the associated strict 2-functor $\Dia^{\cor}(\DD) \rightarrow \Dia^{\cor}(\SSS)$ is a 1-opfibration (resp.\@ 1-fibration) of 2-multicategories. 
 \end{HAUPTSATZ}
 \begin{proof}
 We first show that $\Dia^{\cor}(\DD) \rightarrow \Dia^{\cor}(\SSS)$ is a 1-opfibration, if $\DD \rightarrow \SSS$ is a left fibered multiderivator. 
 Let $x=(A; \alpha_{A,1}, \dots, \alpha_{A,n}; \beta_A)$ be a correspondence in $\Span^F(I_1, \dots, I_n; J)$ and let 
 \[ f \in \Hom_{\Dia^{\cor}(\SSS)}(\alpha_{A,1}^*S_1, \dots, \alpha_{A,n}^*S_n; \beta_A^*T)  \]
 be a 1-morphism in $\Dia^{\cor}(\SSS)$ lying over $x$. In $\Dia^{\cor}(\DD)$ we have the following composition of isomorphisms of sets:
\begin{eqnarray*} 
&& \Hom_{\Dia^{\cor}(\DD),f}((I_1, \mathcal{E}_1), \dots, (I_n, \mathcal{E}_n); (J, \mathcal{F})) \\
&\cong& \Hom_{\DD(A),f}(\alpha_{A,1}^*\mathcal{E}_1, \dots, \alpha_{A,n}^*\mathcal{E}_n; \beta_A^* \mathcal{F}) \\
&\cong& \Hom_{\DD(A),\id_{\beta_A^*T}}( f_\bullet(\alpha_{A,1}^*\mathcal{E}_1, \dots, \alpha_{A,n}^*\mathcal{E}_n); \beta_A^* \mathcal{F}) \\
&\cong& \Hom_{\DD(A),\id_{T}}( \beta_{A,!} f_\bullet(\alpha_{A,1}^*\mathcal{E}_1, \dots, \alpha_{A,n}^*\mathcal{E}_n);  \mathcal{F}) \\
&\cong& \Hom_{\Dia^{\cor}(\DD),\id_{(J, T)}}( (J, \beta_{A,!} f_\bullet(\alpha_{A,1}^*\mathcal{E}_1, \dots, \alpha_{A,n}^*\mathcal{E}_n));  (J, \mathcal{F}))
\end{eqnarray*}
using (FDer0 left) and (FDer3 left).
One checks that this composition is induced by the composition in $\Dia^{\cor}(\DD)$ with a 1-morphism in
\[  \Hom_f((I_1, \mathcal{E}_1), \dots, (I_n, \mathcal{E}_n); (J, \beta_{A,!} f_\bullet(\alpha_{A,1}^*\mathcal{E}_1, \dots, \alpha_{A,n}^*\mathcal{E}_n))). \]
Hence this 1-morphism is weakly coCartesian.

Note that we write $\Hom_{\Dia^{\cor}(\DD),f}$ for the category of 1-morphisms which map to $f$ in $\Dia^{\cor}(\SSS)$ and those 2-morphisms that map to $\id_f$ in $\Dia^{\cor}(\SSS)$. 
 
It remains to be shown that the composition of weakly coCartesian 1-morphisms is weakly coCartesian (cf.\@ Proposition~\ref{LEMMAWEAKLY}). Let
 \[ g \in \Hom_{\Dia^{\cor}(\SSS)}(\alpha_{B,1}^*T_1, \dots, \alpha_{B,m}^*T_m; \beta_{B}^*U)  \]
 be another 1-morphism, composable with $f$, lying over a correspondence $y=(B; \alpha_{B,1}, \dots, \alpha_{B,m}; \beta_B)$ in $\Span^F(J_1, \dots, J_m; K)$.
Setting $J_i:=J$ and $T_i:=T$ the composition of $x$ and $y$ w.r.t.\@ the $i$-th slot is the correspondence:
 
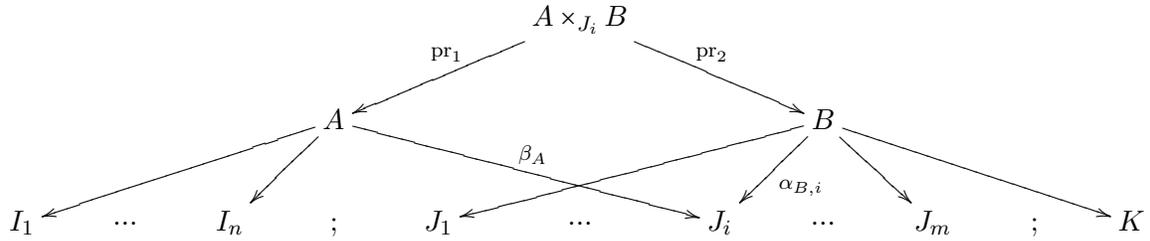
\[ \xymatrix{
&&&&& A \times_{J_i} B  \ar[lld]_{\pr_1} \ar[rrd]^{\pr_2} \\
&&& A \ar[llld] \ar[ld] \ar[rrrd]^{\beta_A} &&&&  B \ar[llld] \ar[ld]^{\alpha_{B,i}}  \ar[rd] \ar[rrrd] \\
I_1 & \cdots & I_n & ; & J_1 & \cdots & J_i & \cdots & J_m & ; & K
} \]
The composition of $g$ and $f$ is the morphism
\[ \pr_2^*g \circ_{i} \pr_1^* f  \]
\[ \in  \Hom(\pr_2^*\alpha_{B,1}^*T_1,\dots,\underbrace{\pr_1^*\alpha_{A,1}^*S_1, \dots, \pr_1^*\alpha_{A,n}^*S_n}_{\text{at }i},\dots,\pr_2^*\alpha_{B,m}^*T_{m}; \pr_2^*\beta_{B}^*U)  \]
We have to show that the natural map
\begin{eqnarray*} 
&&\beta_{B,!} g_\bullet(\alpha_{B,1}^*\mathcal{F}_{1}, \dots,  \underbrace{\alpha_{B,i}^* \beta_{A,!} f_\bullet (\alpha_{A,1}^*\mathcal{E}_1, \dots, \alpha_{A,n}^*\mathcal{E}_n)}_{\text{at $i$}}, \dots  \alpha_{B,m}^*\mathcal{F}_m) \\
&\rightarrow &\beta_{B,!} \pr_{2,!} (\pr_2^*g \circ_{i} \pr_1^* f)_\bullet  (\pr_2^* \alpha_{B,1}^*\mathcal{F}_{1}, \dots, \underbrace{\pr_1^* \alpha_{A,1}^* \mathcal{E}_1, \dots, \pr_1^* \alpha_{A,n}^*\mathcal{E}_n}_{\text{at $i$}}, \dots \pr_2^* \alpha_{B,m}^*\mathcal{F}_m)
 \end{eqnarray*}
 is an isomorphism.  It is the composition of the following morphisms which are all isomorphisms respectively by (FDer4 left) using \cite[Proposition~1.3.23.\@ 2.]{Hor15}, by (FDer5 left) observing that $\pr_2$ is a opfibration, by the second part of (FDer0 left) for $\pr_1$, and by the first part of (FDer0 left) in the form that the composition of coCartesian morphisms is coCartesian.
\begin{eqnarray*} 
&&\beta_{B,!} g_\bullet(\alpha_{B,1}^*\mathcal{F}_{1}, \dots,  \underbrace{\alpha_{B,i}^* \beta_{A,!} f_\bullet (\alpha_{A,1}^*\mathcal{E}_1, \dots, \alpha_{A,n}^*\mathcal{E}_n)}_{\text{at $i$}}, \dots,  \alpha_{B,m}^*\mathcal{F}_m) \\
&\rightarrow& \beta_{B,!} g_\bullet(\alpha_{B,1}^*\mathcal{F}_{1}, \dots,  \underbrace{\pr_{2,!}\pr_1^* f_\bullet (\alpha_{A,1}^*\mathcal{E}_1, \dots, \alpha_{A,n}^*\mathcal{E}_n)}_{\text{at $i$}}, \dots,  \alpha_{B,m}^*\mathcal{F}_m) \\
&\rightarrow& \beta_{B,!} \pr_{2,!} (\pr_{2}^*g)_\bullet(\pr_{2}^* \alpha_{B,1}^*\mathcal{F}_{1}, \dots,  \underbrace{\pr_1^* f_\bullet (\alpha_{A,1}^*\mathcal{E}_1, \dots, \alpha_{A,n}^*\mathcal{E}_n)}_{\text{at $i$}}, \dots,  \pr_{2}^*\alpha_{B,m}^*\mathcal{F}_m) \\
&\rightarrow& \beta_{B,!} \pr_{2,!} (\pr_{2}^*g)_\bullet(\pr_{2}^* \alpha_{B,1}^*\mathcal{F}_{1}, \dots,  \underbrace{ (\pr_1^*f)_\bullet (\pr_1^* \alpha_{A,1}^*\mathcal{E}_1, \dots, \pr_1^*\alpha_{A,n}^*\mathcal{E}_n)}_{\text{at $i$}}, \dots,  \pr_{2}^*\alpha_{B,m}^*\mathcal{F}_m) \\
&\rightarrow& \beta_{B,!} \pr_{2,!} (\pr_2^*g \circ_{i} \pr_1^* f)_\bullet  (\pr_2^* \alpha_{B,1}^*\mathcal{F}_{1}, \dots, \underbrace{ \pr_1^* \alpha_{A,1}^* \mathcal{E}_1, \dots, \pr_1^* \alpha_{A,n}^*\mathcal{E}_n}_{\text{at $i$}}, \dots, \pr_2^* \alpha_{B,m}^*\mathcal{F}_m).
 \end{eqnarray*}
 
 Now we proceed to prove the converse, hence we assume that $\Dia^{\cor}(\DD) \rightarrow \Dia^{\cor}(\SSS)$ is a 1-opfibration and have to show all axioms of a left fibered derivator:

(FDer0 left) 
 First we have an obvious pseudo-functor of 2-multicategories
 \begin{eqnarray*}
  F: \SSS(I) &\hookrightarrow& \Dia^{\cor}(\SSS) \\
   S & \mapsto & (I, S) \\
    f & \mapsto & [f].
 \end{eqnarray*}
 By Proposition~\ref{PROPPULLBACK} the pull-back $F^*\Dia^{\cor}(\DD) \rightarrow \SSS(I)$ (in the sense of Definition~\ref{DEFPULLBACK}) is 1-opfibered and 2-fibered, if $\Dia^{\cor}(\DD) \rightarrow \Dia^{\cor}(\SSS)$ is 1-opfibered and 2-fibered. 
 To show that $\DD(I) \rightarrow \SSS(I)$ is a 1-opfibration and 2-fibration of multicategories, it thus suffices to show that $F^*\Dia^{\cor}(\DD)$ is equivalent to $\DD(I)$ over $\SSS(I)$. 
 The class of objects of $F^*\Dia^{\cor}(\DD)$ is by definition isomorphic to the class of objects of $\DD(I)$. Therefore we are left to show that there are equivalences of categories (compatible with composition)
 \[ \Hom_{\DD(I),f}(\mathcal{E}_1, \dots, \mathcal{E}_n; \mathcal{F}) \rightarrow \Hom_{F^*\Dia^{\cor}(\DD),f}(\mathcal{E}_1, \dots, \mathcal{E}_n; \mathcal{F})  \]
 for any morphism $f \in \Hom_{\SSS(I)}(S_1, \dots, S_n; T)$, where $\mathcal{E}_i$ is an object of $\DD(I)$ over $S_i$ and $\mathcal{F}$ is an object over $T$. 
Note that the left-hand side is a set.
 
We have a 2-Cartesian diagram of categories
\[ \xymatrix{
\Hom_{F^*\Dia^{\cor}(\DD),f}(\mathcal{E}_1, \dots, \mathcal{E}_n; \mathcal{F}) \ar[r] \ar[d] & \Hom_{\Dia^{\cor}(\DD)}((I,\mathcal{E}_1), \dots, (I, \mathcal{E}_n); (I, \mathcal{F})) \ar[d] \\
\{f\} \ar[r]^-F & \Hom_{\Dia^{\cor}(\SSS)}((I,S_1), \dots, (I,S_n); (I,T))
} \]
Since the right vertical morphism is a fibration (cf.\@ Lemma~\ref{PAREQSPANFIBERED}) the diagram is also Cartesian. Futhermore by Lemma~\ref{PAREQSPANFIBERED} the right vertical morphism is equivalent to 
\[ \xymatrix{
  \tau_1\left(\Span_{\DD}((I, \mathcal{E}_1), \dots, (I, \mathcal{E}_n); (I, \mathcal{F}))\right) \ar[d] \\
  \tau_1\left(\Span_{\SSS}((I, S_1), \dots, (I, S_n); (I, T))\right).
} \]
(Here $\Span_{\DD}^F(\dots)$ was changed to $\Span_{\DD}(\dots)$ and similarly for $\Span_{\SSS}^F(\dots)$.)

In the category $\tau_1(\Span_{\SSS}((I_1, S_1), \dots, (I_n, S_n); (J, T)))$, the object $F(f)$ is isomorphic to the pair consisting of the trivial correspondence $(\id_I, \dots, \id_I; \id_I)$ and $f$ over it, whose fiber in the category $\tau_1(\Span_{\DD}((I, \mathcal{E}_1), \dots, (I, \mathcal{E}_n); (I, \mathcal{F})))$ is precisely the discrete category $\Hom_{\DD(I),f}(\mathcal{E}_1, \dots, \mathcal{E}_n; \mathcal{F})$.
The remaining part of (FDer0 left) will be shown below. 

Since we have a 1-opfibration and 2-fibration we can equivalently
see the given datum as a pseudo-functor
\[ \Psi: \Dia^{\cor}(\SSS) \rightarrow \mathcal{CAT} \]
and by Proposition~\ref{PROPBASICDIACORS}, $\Psi$ maps
$[\alpha^{(S)}]$ to a functor natural isomorphic to $\alpha^*: \DD(J)_S \rightarrow \DD(I)_{\alpha^*S}$. We have the freedom to 
choose $\Psi$ in such a way that it maps $[\alpha^{(S)}]$ precisely to $\alpha^*$. 

Axiom (FDer3 left) follows from Lemma~\ref{LEMMAPROPDIACORS}, 1.\@ stating that $[\alpha^{(S)}]$ has a left adjoint $[\alpha^{(S)}]'$ in the category $\Dia^{\cor}(\SSS)$ (cf.\@ also Proposition~\ref{PROPBASICDIACORS}). 

Axiom (FDer4 left) follows by applying $\Psi$ to the (first) 2-isomorphism of Lemma~\ref{LEMMAPROPDIACORS}, 2.

Axiom (FDer5 left) follows by applying $\Psi$ to the 2-isomorphism of Lemma~\ref{LEMMAPROPDIACORS}, 4.

The remaining part of (FDer0 left), i.e.\@ that $\alpha^*$ maps coCartesian arrows to coCartesian arrows follows by applying $\Psi$ to the 2-isomorphism of Lemma~\ref{LEMMAPROPDIACORS}, 3.
\end{proof}
 
\begin{BEM}
Given a 1-fibration (resp.\@ 1-opfibration) and 2-fibration of 2-multicategories with 1-categorical fibers $\mathcal{D} \rightarrow \Dia^{\cor}$ we can also reconstruct a (non-strict) pre-multiderivator
$\DD$. This will be briefly explained in \ref{PARDERCANWIRTH}. We will however not use this fact, but assume that our (op)fibered  2-multicategories $\Dia^{\cor}(\DD) \rightarrow \Dia^{\cor}$ come from a strict pre-multiderivator. 
\end{BEM}

\begin{BEM}
In view of Main~Theorem~\ref{MAINTHEOREMFIBDER} the basic results of \cite[\S 1]{Hor15} appear in a much clearer fashion. For example, from the transitivity for
bifibered 2-multicategories (Lemma~\ref{PROPBIFIBTRANSITIVITY}) follows immediately the transitivity for fibered multiderivators \cite[Proposition 1.4.1.]{Hor15}. 
\end{BEM}

\section{Internal and external monoidal structure.}\label{ABSRELMONOIDAL}

\begin{PAR}
Let $\mathcal{D}, \mathcal{S}$ be symmetric (for simplicity) 2-multicategories with all 2-morphisms invertible. 
Let $\mathcal{D} \rightarrow \mathcal{S}$ be a symmetric 1-bifibered and 2-isofibered 2-multicategory such that also $\mathcal{S} \rightarrow \{\cdot\}$ is 1-bifibered.
Then any pseudo-functor of 2-multicategories $s: \{\cdot\} \rightarrow \mathcal{S}$ with value $S$ gives rise to a symmetric closed monoidal structure $\otimes$ on the 2-category $\mathcal{D}_S$.
Moreover $\mathcal{D} \rightarrow \{\cdot\}$ is also fibered by transitivity (cf.\@ Proposition~\ref{PROPBIFIBTRANSITIVITY}). Therefore the whole 2-category $\mathcal{D}$  carries a closed monoidal structure $\boxtimes$ as well.
We call $\otimes$ the {\bf internal product}, and $\boxtimes$ the {\bf external product}, and write $\mathcal{HOM}$, and $\mathbf{HOM}$, respectively for the adjoints. We also denote by $\otimes$ the monoidal product in $\mathcal{S}$ itself and by $\mathcal{HOM}$ its adjoint. 
\end{PAR}

\begin{PAR}
The functor $s$ specifies, in particular, a distinguished 1-multimorphism $\Delta \in \Hom(S, S; S)$.
By abuse of notation, we denote by $\Delta$ (resp.\@ $\Delta'$) the corresponding 1-morphisms
\[ \Delta: S \otimes S \rightarrow S \qquad \Delta': S \rightarrow \mathcal{HOM}(S,S). \] 
\end{PAR}

By the arguments in the  proof of the transitivity of bifibrations of multicategories (cf.\@ Proposition~\ref{PROPBIFIBTRANSITIVITY}), we see that we actually have
\begin{KOR}
\begin{eqnarray*}
\mathcal{E} \otimes \mathcal{F} &\cong& \Delta_\bullet (\mathcal{E} \boxtimes \mathcal{F}), \\
\mathcal{HOM}(\mathcal{E}, \mathcal{F}) &\cong& (\Delta')^\bullet (\mathbf{HOM}(\mathcal{E} , \mathcal{F} )).
\end{eqnarray*}
\end{KOR}

\begin{BEISPIEL}\label{EXINTEXTMONOIDALDIACOR}
Let us investigate the internal and external monoidal structure in the case $\mathcal{S} = \mathcal{S}^{\cor}$ (cf.\@ Definition~\ref{DEFSCOR}). 
Here the 1-morphisms $\Delta$ and $\Delta'$ are respectively given by the correspondences
\[ \vcenter{\xymatrix{
& S   \ar[ld]_{\mathrm{diag}} \ar@{=}[rd] \\
S \times S & ; & S
}}
\quad \text{and} \quad 
\vcenter{\xymatrix{
& S   \ar[rd]^{\mathrm{diag}} \ar@{=}[ld] \\
S & ; & S \times S.
}} 
\] 
From this we see that $\Delta_\bullet \cong \Delta^*$ and $(\Delta')^{\bullet} \cong \Delta^!$ hold. 

In the other direction, we can also reconstruct the external monoidal product and its adjoint from the internal one.
The functor $\boxtimes$ is the push-forward along the coCartesian 1-morphism
\[ \xymatrix{
& S \times S  \ar[ld]_{\pr_1}  \ar[d]^{\pr_2} \ar@{=}[rrd]  \\
S & S & ; & S \times S,
} \]
hence we have $(- \boxtimes -) \cong (\pr_1^* - \otimes \pr_2^* -)$.
The functor $\mathbf{HOM}$ is the pull-back w.r.t.\@ the first slot (say) along the Cartesian (w.r.t.\@ the first slot) 1-morphism:
\[ \xymatrix{
& S \times S  \ar@{=}[ld]  \ar[d]^{\pr_1} \ar[rrd]^{\pr_2}  \\
S \times S & S & ; & S,
} \]
hence we have $\mathbf{HOM}(-,-) \cong \mathcal{HOM}(\pr_1^* - , \pr_2^! - )$.
\end{BEISPIEL}

\begin{BEISPIEL}
Let us investigate the internal and external monoidal structure in the case $\mathcal{S} = \Dia^{\cor}$ (cf.\@ Definition~\ref{DEFDIACOR}).
By Proposition~\ref{PROPDIACORMONOIDAL} we know $I \otimes I \cong I \times I$ and $\mathcal{HOM}(I,I) \cong I^{\op} \times I$. 
The 1-morphism $\Delta$ is given by the correspondence
\[ \xymatrix{
& I   \ar[ld]_{\Delta} \ar@{=}[rd] \\
I \times I & ; & I.
} \]

To determine $\Delta'$, observe that the correspondence
\[ \xymatrix{
& I   \ar@{=}[ld]_{} \ar@{=}[d]_{} \ar@{=}[rrd] \\
I & I & ; & I
} \]
belongs (via \ref{KOREQSPAN}) to the following functor in $\Fun(I^{\op} \times I^{\op} \times I, \Dia)$:
\begin{eqnarray*}
F_I: I^{\op} \times I^{\op} \times I &\rightarrow& \Dia \\
i, i', i'' & \mapsto & \Hom(i, i'') \times \Hom(i', i'')
\end{eqnarray*}
which yields (via \ref{KOREQSPAN} again) the correspondence $\Delta'$:
\[ \xymatrix{
& \int \nabla F_I  \ar[ld]_{p} \ar[rd]^q \\
  I &  & I^{\op} \times I
} \]
and we have $\int \nabla F_I = I \times_{/I }\tw(I)$. 
 
We see that $\Delta_\bullet \cong \Delta^*$ and $(\Delta')^{\bullet} \cong p_* q^*$ hold. The latter is also the same as $\pr_{2,*} \pi_* \pi^*$ for the following functors:
\[ \xymatrix{ \tw(I) \ar[r]^\pi & I^{\op} \times I \ar[r]^-{\pr_2} & I. } \]
Given a bifibration of 1-multicategories $\mathcal{D} \rightarrow \mathcal{S}$, this explains more conceptually the construction of the ``multi-pull-back'' in the multicategory of functors $\Fun(I, \mathcal{D})$ in \cite[Proposition~4.1.6.]{Hor15}. Using Proposition~\ref{PROPBIFIBTRANSITIVITYCONVERSE} one can even reprove
the Proposition [loc.\@ cit.] in case that $\mathcal{S}$ is closed monoidal (i.e.\@ bifibered over $\{ \cdot \}$).
Applying Propositions~\ref{PROPBIFIBTRANSITIVITY} and \ref{PROPBIFIBTRANSITIVITYCONVERSE} to the composition 
\[ \Dia^{\cor}(\DD) \rightarrow \Dia^{\cor}  \rightarrow \{\cdot\} \] we can show that for a derivator $\DD$ it is the same
\begin{enumerate}
\item to define an absolute monoidal product and absolute Hom which commute with left, resp.\@ right Kan extensions in the correct way (conditions 1.--3. of \ref{PROPBIFIBTRANSITIVITYCONVERSE}) or
\item to give $\DD$ the structure of a closed monoidal derivator. 
\end{enumerate}
\end{BEISPIEL}

\section{Grothendieck and Wirthm\"uller}\label{GROTHWIRTH}

\begin{PAR}
Let $\mathcal{S}$ be a category with fiber products and final object and let $\mathcal{S}_0$ be a class of morphisms in $\mathcal{S}$.
We can define a subcategory $\mathcal{S}^{\cor,0}$ of $\mathcal{S}^{\cor}$ where the 2-morphisms are those
 
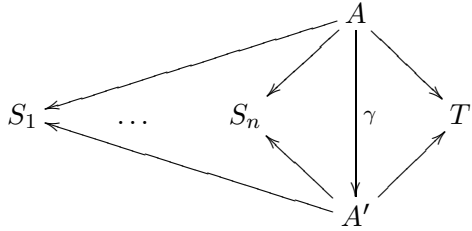
\[ \xymatrix{
& & & A \ar[ld] \ar[llld] \ar[rd] \ar[dd]^\gamma\\
S_1 & \dots & S_n  &  &  T \\
&& &  A' \ar[lu] \ar[lllu] \ar[ru]
} \]
in which $\gamma \in \mathcal{S}_0$. 
If $\mathcal{S}_0$ is the class of {\em all} morphisms in $\mathcal{S}$, then we denote $\mathcal{S}^{\cor,0}$ by  $\mathcal{S}^{\cor,G}$.
\end{PAR}

\begin{LEMMA}
Consider the category $\mathcal{S}^{\cor,0}$ and a morphism $f: S \rightarrow T$ in $\mathcal{S}_0$ such that also $\Delta_f: S \rightarrow S \times_T S$ is in $\mathcal{S}_0$. Then the morphisms
\[ f^{\op}: \vcenter{\xymatrix{
  & S \ar[ld]_{f} \ar@{=}[rd] \\
 T & & S
}} \qquad f: \vcenter{ \xymatrix{
  & S \ar[rd]^{f} \ar@{=}[ld] \\
 S & & T
} } \]
are adjoints in the 2-category $\mathcal{S}^{\cor,0}$. 
\end{LEMMA}
\begin{proof}
We give unit and counit:
\[ f \circ f^{\op} \Rightarrow \id: \vcenter{ \xymatrix{
  & T \ar@{=}[ld] \ar@{=}[rd] \\
 T & & T \\
  & S \ar[lu]^{f} \ar[ru]_{f} \ar[uu]^f 
} } \]

\begin{equation}\label{eqdelta}
 \id \Rightarrow f^{\op} \circ f: \vcenter{  \xymatrix{
  & S \ar@{=}[ld] \ar@{=}[rd]  \ar[dd]^{\Delta_f}  \\
 S & & S \\
  & S \times_T S \ar[lu]^{\pr_1} \ar[ru]_{\pr_2}
} } 
\end{equation}
One easily checks the unit/counit equations. 
\end{proof}

\begin{PROP}\label{PROPPROPERETALE}
\begin{enumerate}
\item
Let $\mathcal{D} \rightarrow \mathcal{S}^{\cor, 0}$ be a proper six-functor-formalism (cf.\@ \ref{DEF6FU}).
If $\Delta_f \in \mathcal{S}_0$ (in many examples this is always the case) then there is a canonical natural transformation
\[ f_! \rightarrow f_* \]
which is an isomorphism if $f \in \mathcal{S}_0$.
\item 
Let $\mathcal{D} \rightarrow \mathcal{S}^{\cor, 0}$ be an etale six-functor-formalism (cf.\@ \ref{DEF6FU}).
If $f \in \mathcal{S}_0$ then there is a canonical natural transformation
\[ f^* \rightarrow f^! \]
which is an isomorphism if $\Delta_f \in \mathcal{S}_0$.
\end{enumerate}
In particular, for a Wirthm\"uller context, we have a canonical isomorphism $f^* \cong f^!$ for all morphisms $f$ in $\mathcal{S}$, and for a
Grothendieck context, we have a canonical isomorphism $f_! \cong f_*$ for all morphisms $f$ in $\mathcal{S}$.  This justifies the naming.
\end{PROP}
\begin{proof}
We prove the first assertion, the second is shown analogously. To give a natural transformation as claimed is equivalent to give a morphism
\[ f^* f_! \rightarrow \id, \]
or equivalently 
\[ \pr_{2,!} \pr_1^* \rightarrow \id \]
with $\pr_1$ and $\pr_2$ as in (\ref{eqdelta}).
This natural transformation is given by means of the 2-pullback along 
the 2-morphism of correspondences (\ref{eqdelta}). If $f$ is in $\mathcal{S}_0$ then this is the counit of an adjunction and hence it induces a canonical isomorphism $f_! \cong f_*$ (uniqueness of adjoints up to canonical isomorphism). 
\end{proof}

\begin{BEISPIEL}\label{EXPROPERETALE}
From the properties of 1/2 (op)fibrations of 2-multicategories one can derive many compatibilities of the morphism $f_! \rightarrow f_*$. For example in a proper six-functor-formalism $\mathcal{D} \rightarrow \mathcal{S}^{\cor}$ for a Cartesian square
\[ \xymatrix{
S \ar[r]^F \ar[d]_G & T \ar[d]^g \\
U \ar[r]_f & V
} \]
in $\mathcal{S}$ the following diagram is 2-commutative
\[ \xymatrix{
G^* F_!  \ar[r]^{\sim} \ar[d] &  f_! g^*\ar[d] \\
G^* F_*  \ar[r]^{\mathrm{exc.}} &  f_* g^*
} \]
provided that $\Delta_f, \Delta_F$ are in $\mathcal{S}_0$.
\end{BEISPIEL}

\begin{BEISPIEL}
The following diagram, which is depicted on the front cover of Lipman's book \cite{LH09} (there a specific Grothendieck context is considered, namely quasi-coherent sheaves on a certain class of proper schemes), is commutative:
\begin{equation}\label{LIPMAN} \vcenter{ \xymatrix{
f_* \Hom(-, f^! -)  \ar[r]^{\sim} \ar[d]^{} &  \Hom(f_!-,-)  \\
f_* \Hom(f^*f_! -, f^! -)  \ar[r]^\sim &  \Hom(f_!-, f_*f^! -) \ar[u]
} } \end{equation}
Here the horizontal morphisms are induced by the natural transformations
\begin{equation}\label{eqf1} 
f^*f_! \rightarrow \id,  
\end{equation}
and 
\begin{equation}\label{eqf2} 
  f_* f^!  \rightarrow \id, 
\end{equation}
respectively, which are the natural transformations on the push-forward (resp.\@ the pull-back) induced by the 2-morphisms of correspondences given by
\[ f \circ f^{\op} \Rightarrow \id: \vcenter{ \xymatrix{
  & T \ar@{=}[ld] \ar@{=}[rd] \\
 T & & T \\
  & S \ar[lu]^{f} \ar[ru]_{f} \ar[uu]^f 
} } \quad \text{and} \quad  \id \Rightarrow f^{\op} \circ f: \vcenter{  \xymatrix{
  & S \ar@{=}[ld] \ar@{=}[rd]  \ar[dd]^{\Delta_f}  \\
 S & & S \\
  & S \times_T S \ar[lu]^{\pr_1} \ar[ru]_{\pr_2}
} } \]
Note: The isomorphism $f_! \cong f_*$ of Proposition~\ref{PROPPROPERETALE},~1.\@ is constructed in such a way that the
two morphisms (\ref{eqf1}) and (\ref{eqf2}) are identified with the two counits
\[ f^*f_* \rightarrow \id 
\qquad \text{and} \qquad  f_! f^!  \rightarrow \id. \]
\end{BEISPIEL}

\begin{proof}
Taking adjoints this is the same as to show that the diagram
\[ \xymatrix{
f_!(- \otimes f^*-)   &  (f_!- ) \otimes - \ar[l]^-{\sim}  \ar[d] \\
f_! ((f^*f_!-) \otimes (f^*-))  \ar[u] &  f_! f^* ((f_! -) \otimes -)  \ar[l]^-\sim
} \]
is commutative. 
This is just the diagram induced on push-forwards by the following commutative diagram of 2-morphisms of multicorrespondences.
\begin{gather*}
\begin{array}{ccc}
 \left( \vcenter{ 
\xymatrix{ & &  S \ar@{=}[lld] \ar[ld]^f \ar[rd]^f \\
 S & T & ; & T  } } \right)
  & \iso &  \cdots \\  
\downarrow &&  \\
\left( \vcenter{ 
\xymatrix{ & &  S \ar[lld]^f \ar[ld]^f \ar[rd]^f \\
 T & T & ; & T  } } \right) \circ_1
\left( \vcenter{ 
\xymatrix{ &  S  \ar@{=}[ld] \ar[rd]^f \\
 S & ; & T  } } \right)
  & \iso &  \cdots
\end{array}\\
\phantom{x} \\
\hline
\phantom{x} \\
\begin{array}{ccc}
\cdots & \iso & \left( \vcenter{ 
\xymatrix{ & &  T \ar@{=}[lld] \ar@{=}[ld] \ar@{=}[rd] \\
 T & T & ; & T  } } \right) \circ_1
\left( \vcenter{ 
\xymatrix{ &  S  \ar@{=}[ld] \ar[rd]^f \\
 S & ; & T  } } \right) \\
& & \uparrow \\
\cdots & \iso & \left( \vcenter{ 
\xymatrix{ &  S  \ar[ld]^f \ar[rd]^f \\
 T & ; & T  } } \right) \circ
\left( \vcenter{ 
\xymatrix{ & &  T \ar@{=}[lld] \ar@{=}[ld] \ar@{=}[rd] \\
 T & T & ; & T  } } \right) \circ_1
\left( \vcenter{ 
\xymatrix{ &  S  \ar@{=}[ld] \ar[rd]^f \\
 S & ; & T  } } \right)
\end{array}
 \end{gather*}
\end{proof}

Hence, for a Grothendieck context given by Definiton~\ref{DEFCANGROTHENDIECK} (as is the context considered in \cite{LH09}) the commutativity of the diagram (\ref{LIPMAN}) follows from Proposition~\ref{PROPCANGROTHENDIECK}.

\begin{PAR}\label{PARDERCANWIRTH}
Analogously we can say that a 1-bifibration and 2-opfibration over $\Dia^{\cor}$ is a Wirthm\"uller context. Note that in $\Dia^{\cor}$ all functors
supply valid 2-morphisms. This shows that to construct e.g.\@ a monoidal derivator one does not have to start with a pre-multiderivator but
could use an arbitrary 1-bifibration and and 2-opfibration over $\Dia^{\cor}$. In detail:

Let $\alpha: I \rightarrow J$ be a functor between diagrams in $\Dia$. 
Recall from Lemma~\ref{LEMMAPROPDIACORS}, 1.\@ that in the category $\Dia^{\cor}$ the correspondences 
\[ [\alpha]': \vcenter{\xymatrix{
  & J \times_{/J} I \ar[ld]_{\alpha} \ar[rd] \\
 J & & I
}} \qquad [\alpha]: \vcenter{ \xymatrix{
  & I \times_{/J} J \ar[rd]^{\alpha} \ar[ld] \\
 I & & J
} } \]
are adjoints. 
Using this, we can reconstruct from a strict 2-functor $\mathcal{D} \rightarrow \Dia^{\cor}$ which is 1-opfibered and 2-fibered with 1-categorical fibers a (non-strict) pre-multiderivator as follows: Consider the embedding $\iota$ from Proposition~\ref{PROPPSEUDOFUNCTDIA} and consider the pull-back (cf.\@ Definition~\ref{DEFPULLBACK}) of $\mathcal{D}$
\[ \xymatrix{
\iota^* \mathcal{D} \ar[r] \ar[d] & \mathcal{D} \ar[d] \\
\Dia^{1-\op} \ar[r]^\iota & \Dia^{\cor}.
} \]
The 1-opfibration, and 2-fibration of 2-multicategories $\iota^* \mathcal{D} \rightarrow \Dia^{1-\op}$ with 1-categorical fibers can be seen (cf.\@ Proposition~\ref{PROPGROTHCONSTR}) as a pseudo-functor
\[ \DD: \Dia^{1-\op} \rightarrow \mathcal{MCAT}. \]
The adjointness of $[\alpha]$ and $[\alpha]'$ shows that $\mathcal{D}$ {\em is determined by} $\iota^* \mathcal{D}$ and can thus be reconstructed by the construction in \ref{PARDEFDIACORD} that associates the 2-multicategory $\Dia^{\cor}(\DD)$ with
a pre-multiderivator $\DD$. The only difference is that the $\DD$ reconstructed from $\mathcal{D}$ might not be a strict 2-functor. 
\end{PAR}
 
\newpage
\bibliographystyle{abbrvnat}
\bibliography{6fu}

\begin{thebibliography}{15}
\providecommand{\natexlab}[1]{#1}
\providecommand{\url}[1]{\texttt{#1}}
\expandafter\ifx\csname urlstyle\endcsname\relax
  \providecommand{\doi}[1]{doi: #1}\else
  \providecommand{\doi}{doi: \begingroup \urlstyle{rm}\Url}\fi

\bibitem[Ayoub(2007{\natexlab{a}})]{Ayo07I}
J.~Ayoub.
\newblock Les six op\'erations de {G}rothendieck et le formalisme des cycles
  \'evanescents dans le monde motivique. {I}.
\newblock \emph{Ast\'erisque}, \penalty0 (314):\penalty0 x+466 pp.,
  2007{\natexlab{a}}.

\bibitem[Ayoub(2007{\natexlab{b}})]{Ayo07II}
J.~Ayoub.
\newblock Les six op\'erations de {G}rothendieck et le formalisme des cycles
  \'evanescents dans le monde motivique. {II}.
\newblock \emph{Ast\'erisque}, \penalty0 (315):\penalty0 vi+364 pp.,
  2007{\natexlab{b}}.

\bibitem[Bakovic(2009)]{Bak09}
I.~Bakovic.
\newblock Fibrations of bicategories.
\newblock Preprint, \\
  \url{http://www.irb.hr/korisnici/ibakovic/groth2fib.pdf}, 2009.

\bibitem[Buckley(2014)]{Buc14}
M.~Buckley.
\newblock Fibred 2-categories and bicategories.
\newblock \emph{J. Pure Appl. Algebra}, 218\penalty0 (6):\penalty0 1034--1074,
  2014.

\bibitem[Cisinski(2003)]{Cis03}
D.~C. Cisinski.
\newblock Images directes cohomologiques dans les cat\'egories de mod\`eles.
\newblock \emph{Ann. Math. Blaise Pascal}, 10\penalty0 (2):\penalty0 195--244,
  2003.

\bibitem[Gordon et~al.(1995)Gordon, Power, and Street]{GPS95}
R.~Gordon, A.~J. Power, and R.~Street.
\newblock Coherence for tricategories.
\newblock \emph{Mem. Amer. Math. Soc.}, 117\penalty0 (558):\penalty0 vi+81,
  1995.

\bibitem[Gurski(2006)]{Gur06}
N.~Gurski.
\newblock An algebraic theory of tricategories.
\newblock PhD thesis, University of Chicago,
  \url{http://www.math.yale.edu/~mg622/tricats.pdf}, 2006.

\bibitem[Hermida(1999)]{Her99}
C.~Hermida.
\newblock Some properties of {${\bf Fib}$} as a fibred {$2$}-category.
\newblock \emph{J. Pure Appl. Algebra}, 134\penalty0 (1):\penalty0 83--109,
  1999.

\bibitem[Hermida(2000)]{Her00}
C.~Hermida.
\newblock Representable multicategories.
\newblock \emph{Adv. Math.}, 151\penalty0 (2):\penalty0 164--225, 2000.

\bibitem[Hermida(2004)]{Her04}
C.~Hermida.
\newblock Fibrations for abstract multicategories.
\newblock In \emph{Galois theory, {H}opf algebras, and semiabelian categories},
  volume~43 of \emph{Fields Inst. Commun.}, pages 281--293. 2004.

\bibitem[H\"ormann(2015)]{Hor15}
F.~H\"ormann.
\newblock Fibered multiderivators and (co)homological descent.
\newblock arXiv: \href{http://arxiv.org/abs/1505.00974}{1505.00974}, 2015.

\bibitem[H\"ormann(2017)]{Hor16}
F.~H\"ormann.
\newblock Derivator {S}ix {F}unctor {F}ormalisms --- {D}efinition and
  {C}onstruction {I}.
\newblock arXiv: \href{http://arxiv.org/abs/1701.02152}{1701.02152}, 2017.

\bibitem[Kapranov and Voevodsky(1994)]{KV94}
M.~M. Kapranov and V.~A. Voevodsky.
\newblock {$2$}-categories and {Z}amolodchikov tetrahedra equations.
\newblock In \emph{Algebraic groups and their generalizations: quantum and
  infinite-dimensional methods ({U}niversity {P}ark, {PA}, 1991)}, volume~56 of
  \emph{Proc. Sympos. Pure Math.}, pages 177--259. Amer. Math. Soc.,
  Providence, RI, 1994.

\bibitem[Lipman and Hashimoto(2009)]{LH09}
J.~Lipman and M.~Hashimoto.
\newblock \emph{Foundations of {G}rothendieck duality for diagrams of schemes},
  volume 1960 of \emph{Lecture Notes in Mathematics}.
\newblock Springer-Verlag, Berlin, 2009.

\bibitem[Shulman(2010)]{Shu10}
M.~A. Shulman.
\newblock Constructing symmetric monoidal bicategories.
\newblock arXiv: \href{http://arxiv.org/abs/1004.0993}{1004.0993}, 2010.

\end{thebibliography}

\end{document}